\documentclass[12pt]{amsart}

\usepackage{amsthm} 
\usepackage{amsmath} 
\usepackage{amssymb} 

\usepackage{microtype} 
\usepackage{pinlabel} 

\usepackage[
pdfborderstyle={},
pdfborder={0 0 0},
pagebackref,
pdftex]{hyperref}

\renewcommand*{\backref}[1]{}
\renewcommand*{\backrefalt}[4]{
  \ifcase #1 %
   [No citations.]%
  \or
   [#2]%
  \else
   [#2]%
  \fi
}

\newcommand{\sic}{{\it sic}}
\newcommand{\apriori}{{\it a\thinspace priori}}

\newcommand{\calA}{\mathcal{A}}
\newcommand{\calB}{\mathcal{B}}
\newcommand{\calC}{\mathcal{C}}
\newcommand{\calD}{\mathcal{D}}

\newcommand{\calG}{\mathcal{G}}

\newcommand{\calM}{\mathcal{M}}

\newcommand{\calP}{\mathcal{P}}
\newcommand{\calQ}{\mathcal{Q}}
\newcommand{\calR}{\mathcal{R}}

\newcommand{\calT}{\mathcal{T}}

\newcommand{\calX}{\mathcal{X}}
\newcommand{\calY}{\mathcal{Y}}

\newcommand{\DD}{\mathbb{D}}
\newcommand{\EE}{\mathbb{E}}

\newcommand{\NN}{\mathbb{N}}

\newcommand{\RR}{\mathbb{R}}

\newcommand{\ZZ}{\mathbb{Z}}

\renewcommand{\setminus}{{\smallsetminus}}

\newcommand{\Id}{\operatorname{Id}}

\newcommand{\supp}{\operatorname{supp}} 

\newcommand{\st}{\mathbin{\mid}} 
\newcommand{\from}{\colon} 

\newcommand{\nin}{\mathbin{\notin}}

\newcommand{\homeo}{\mathrel{\cong}} 

\newcommand{\cross}{\times}

\newcommand{\frontier}{\operatorname{fr}} 
\newcommand{\bdy}{\partial} 

\newcommand{\diam}{\operatorname{diam}} 

\newcommand{\interior}{{\operatorname{interior}}}

\newcommand{\RRPP}{\mathbb{RP}} 

\newcommand{\MCG}{\mathcal{MCG}} 

\newcommand{\PML}{\mathcal{PML}}  
\newcommand{\Teich}{{Teichm\"uller~}}

\theoremstyle{plain}
\numberwithin{equation}{section} 
\newtheorem{theorem}[equation]{Theorem}
\newtheorem{corollary}[equation]{Corollary}
\newtheorem{lemma}[equation]{Lemma}

\newtheorem{proposition}[equation]{Proposition}

\newtheorem{axiom}[equation]{Axiom}

\theoremstyle{definition}
\newtheorem{definition}[equation]{Definition}
\newtheorem{remark}[equation]{Remark}
\newtheorem*{remark*}{Remark}
\newtheorem{claim}[equation]{Claim}
\newtheorem*{claim*}{Claim}

\newtheorem{example}[equation]{Example}

\newtheorem*{question*}{Question}
\newtheorem*{answer*}{Answer}
\newtheorem*{case*}{Case}
\newtheorem*{application*}{Application}
\newtheorem{algorithm}[equation]{Algorithm}
\newtheorem*{proofclaim}{Claim}

\theoremstyle{definition}

\newcommand{\refsec}[1]{Section~\ref{Sec:#1}}

\newcommand{\refthm}[1]{Theorem~\ref{Thm:#1}}
\newcommand{\refcor}[1]{Corollary~\ref{Cor:#1}}
\newcommand{\reflem}[1]{Lemma~\ref{Lem:#1}}
\newcommand{\refprop}[1]{Proposition~\ref{Prop:#1}}
\newcommand{\refclm}[1]{Claim~\ref{Clm:#1}}
\newcommand{\refrem}[1]{Remark~\ref{Rem:#1}}
\newcommand{\refexa}[1]{Example~\ref{Exa:#1}}

\newcommand{\reffig}[1]{Figure~\ref{Fig:#1}}
\newcommand{\refdef}[1]{Definition~\ref{Def:#1}}

\newcommand{\refalg}[1]{Algorithm~\ref{Alg:#1}}
\newcommand{\refax}[1]{Axiom~\ref{Ax:#1}}
\newcommand{\refeqn}[1]{Equation~\ref{Eqn:#1}}

\newcommand{\fakeenv}{} 

\newenvironment{restate}[2]  
{ 
 \renewcommand{\fakeenv}{#2} 
 \theoremstyle{plain} 
 \newtheorem*{\fakeenv}{#1~\ref{#2}} 
 \begin{\fakeenv}
}
{
 \end{\fakeenv}
}

\newcommand{\quasileq}{\mathbin{\leq_{A}}} 
 
\newcommand{\quasieq}{\mathbin{=_A}} 

\newcommand{\ProjectError}{R_0}
\newcommand{\DistError}{R_1}

\newcommand{\IncompConst}{61} 
\newcommand{\RealIncompConst}{43} 
\newcommand{\ReallyIncompConst}{37} 
\newcommand{\AnnulusConst}{6}
\newcommand{\HalfAnnulusConst}{3}

\newcommand{\Stable}{M_3}
\newcommand{\Quasi}{Q}
\newcommand{\Radius}{R_2}
\newcommand{\RadiusTemp}{R_3}
\newcommand{\LargeProj}{L_0}

\newcommand{\GeodConst}{M_0} 
\newcommand{\ZugConst}{{M_1}} 
\newcommand{\ZugConstTwo}{{M_2}} 

\newcommand{\NearConstTwo}{{N_2}}

\newcommand{\PathConst}{{K_1}}

\newcommand{\CutOff}{{C_0}}

\newcommand{\Reverse}{{C_1}}
\newcommand{\Combin}{{C_2}}
\newcommand{\AccessTemp}{{B_3}}
\newcommand{\Access}{{C_3}}
\newcommand{\Replace}{{C_4}}

\newcommand{\Induct}{{L_0}}
\newcommand{\Lower}{{L_1}}
\newcommand{\Upper}{{L_2}}
\newcommand{\Shortcut}{{L_3}}
\newcommand{\Paired}{L_4}

\newcommand{\short}{{\text{short}}}
\newcommand{\str}{{\text{str}}}
\newcommand{\induct}{{\text{ind}}}
\newcommand{\elect}{{\text{ele}}}

\newcommand{\Bound}{B_0}

\newcommand{\ind}{\operatorname{index}} 
\newcommand{\AC}{\mathcal{AC}} 
\newcommand{\base}{\operatorname{base}} 

\begin{document}


\title{The geometry of the disk complex}

\author{Howard Masur}
\address{\hskip-\parindent
        Department of Mathematics\\
        University of Chicago\\
        Chicago, Illinois 60637}
\email{masur@math.uic.edu}
\author{Saul Schleimer}
\address{\hskip-\parindent
        Department of Mathematics\\
        University of Warwick\\
        Coventry, CV4 7AL, UK}
\email{s.schleimer@warwick.ac.uk}

\thanks{This work is in the public domain.}

\date{\today}

\begin{abstract}
We give a distance estimate for the metric on the disk complex and
show that it is Gromov hyperbolic.  As another application of our
techniques, we find an algorithm which computes the Hempel distance of
a Heegaard splitting, up to an error depending only on the genus.
\end{abstract}
\maketitle

\setcounter{tocdepth}{1}
\tableofcontents

\section{Introduction}
\label{Sec:Introduction}

In this paper we initiate the study of the geometry of the disk
complex of a handlebody $V$.  The disk complex $\calD(V)$ has a
natural simplicial inclusion into the curve complex $\calC(S)$ of the
boundary of the handlebody.  Surprisingly, this inclusion is not a
quasi-isometric embedding; there are disks which are close in the
curve complex yet very far apart in the disk complex.  As we will
show, any obstruction to joining such disks via a short path is a
topologically meaningful subsurface of $S = \bdy V$.  We call such
subsurfaces {\em holes}.  A path in the disk complex must travel into
and then out of these holes; paths in the curve complex may skip over
a hole by using the vertex representing the boundary of the
subsurface.  We classify the holes:

\begin{theorem}
\label{Thm:ClassificationHolesDiskComplex}
Suppose $V$ is a handlebody.  If $X \subset \bdy V$ is a hole for the
disk complex $\calD(V)$ of diameter at least $\IncompConst$ then:
\begin{itemize}
\item
$X$ is not an annulus.
\item
If $X$ is compressible then there are disks $D, E$ with boundary
contained in $X$ so that the boundaries fill $X$.
\item
If $X$ is incompressible then there is an $I$-bundle $\rho_F \from T
\to F$ so that $T$ is a component of $V \setminus \bdy_v T$ and $X$ is
a component of $\bdy_h T$.
\end{itemize}
\end{theorem}


See Theorems~\ref{Thm:Annuli}, \ref{Thm:CompressibleHoles} and
\ref{Thm:IncompressibleHoles} for more precise statements.  The
$I$--bundles appearing in the classification lead us to study the arc
complex $\calA(F)$ of the base surface $F$.  Since the $I$--bundle $T$
may be twisted the surface $F$ may be non-orientable.

Thus, as a necessary warm-up to the difficult case of the disk
complex, we also analyze the holes for the curve complex of an
non-orientable surface, as well as the holes for the arc complex.

\subsection*{Topological application}
It is a long-standing open problem to decide, given a Heegaard
diagram, whether the underlying splitting surface is reducible.  
This question has deep connections to the geometry, topology, and
algebra of the ambient three-manifold.  For example, a resolution of
this problem would give new solutions to both the three-sphere
recognition problem and the triviality problem for three-manifold
groups.
The difficulty of deciding reducibility is underlined by its
connection to the Poincar\'e conjecture: several approaches to the
Poincar\'e Conjecture fell at essentially this point.
See~\cite{CavicchioliSpaggiari06} for a entrance into the literature.


One generalization of deciding reducibility is to find an algorithm
that, given a Heegaard diagram, computes the {\em distance} of the
Heegaard splitting as defined by Hempel~\cite{Hempel01}.  (For
example, see~\cite[Section~2]{Birman06}.)  The classification of holes
for the disk complex leads to a coarse answer to this question.


\begin{restate}{Theorem}{Thm:CoarselyComputeDistance}
In every genus $g$ there is a constant $K = K(g)$ and an algorithm
that, given a Heegaard diagram, computes the distance of the Heegaard
splitting with error at most $K$.
\end{restate}

In addition to the classification of holes, the algorithm relies on
the Gromov hyperbolicity of the curve complex~\cite{MasurMinsky99} and
the quasi-convexity of the disk set inside of the curve
complex~\cite{MasurMinsky04}. However the algorithm does not depend on
our geometric applications of \refthm{ClassificationHolesDiskComplex}.

\subsection*{Geometric application}

The hyperbolicity of the curve complex and the classification of holes
allows us to prove:

\begin{restate}{Theorem}{Thm:DiskComplexHyperbolic}
The disk complex is Gromov hyperbolic.
\end{restate}

Again, as a warm-up to the proof of \refthm{DiskComplexHyperbolic} we
prove that $\calC(F)$ and $\calA(S)$ are hyperbolic in
\refcor{NonorientableCurveComplexHyperbolic} and
\refthm{ArcComplexHyperbolic}.  Note that Bestvina and
Fujiwara~\cite{BestvinaFujiwara07} have previously dealt with the
curve complex of a non-orientable surface, following
Bowditch~\cite{Bowditch06}.

These results cannot be deduced from the fact that $\calD(V)$,
$\calC(F)$, and $\calA(S)$ can be realized as quasi-convex subsets of
$\calC(S)$.  This is because the curve complex is locally infinite.
As simple example consider the Cayley graph of $\ZZ^2$ with the
standard generating set.  Then the cone $C(\ZZ^2)$ of height one-half
is a Gromov hyperbolic space and $\ZZ^2$ is a quasi-convex subset.
Another instructive example, very much in-line with our work, is the
usual embedding of the three-valent tree $T_3$ into the Farey
tessellation.

The proof of \refthm{DiskComplexHyperbolic} requires the {\em distance
estimate} \refthm{DiskComplexDistanceEstimate}: the distance in
$\calC(F)$, $\calA(S)$, and $\calD(V)$ is coarsely equal to the sum of
subsurface projection distances in holes.  However, we do not use the
hierarchy machine introduced in~\cite{MasurMinsky00}.  This is because
hierarchies are too flexible to respect a symmetry, such as the
involution giving a non-orientable surface, and at the same time too
rigid for the disk complex.  For $\calC(F)$ we use the highly rigid
\Teich geodesic machine, due to Rafi~\cite{Rafi10}.  For $\calD(V)$ we
use the extremely flexible train track machine, developed by ourselves
and Mosher~\cite{MasurEtAl10}.

Theorems~\ref{Thm:DiskComplexDistanceEstimate} and
\ref{Thm:DiskComplexHyperbolic} are part of a more general framework.
Namely, given a combinatorial complex $\calG$ we understand its
geometry by classifying the holes: the geometric obstructions lying
between $\calG$ and the curve complex.  In Sections~\ref{Sec:Axioms}
and \ref{Sec:Partition} we show that any complex $\calG$ satisfying
certain axioms necessarily satisfies a distance estimate.  That
hyperbolicity follows from the axioms is proven in
\refsec{Hyperbolicity}.

Our axioms are stated in terms of a path of markings, a path in the
the combinatorial complex, and their relationship.  For the disk
complex the combinatorial paths are surgery sequences of essential
disks while the marking paths are provided by train track splitting
sequences; both constructions are due to the first author and
Minsky~\cite{MasurMinsky04} (\refsec{BackgroundTrainTracks}). The
verification of the axioms (\refsec{PathsDisk}) relies on our work
with Mosher, analyzing train track splitting sequences in terms of
subsurface projections~\cite{MasurEtAl10}.

We do not study non-orientable surfaces directly; instead we focus on
symmetric multicurves in the double cover.  This time marking paths
are provided by \Teich geodesics, using the fact that the symmetric
Riemann surfaces form a totally geodesic subset of \Teich space.  The
combinatorial path is given by the systole map.  We use results of
Rafi~\cite{Rafi10} to verify the axioms for the complex of symmetric
curves.  (See \refsec{PathsNonorientable}.) \refsec{PathsArc} verifies
the axioms for the arc complex again using \Teich geodesics and the
systole map.  It is interesting to note that the axioms for the arc
complex can also be verified using hierarchies or, indeed, train track
splitting sequences. 


The distance estimates for the marking graph and the pants graph, as
given by the first author and Minsky~\cite{MasurMinsky00}, inspired
the work here, but do not fit our framework.  Indeed, neither the
marking graph nor the pants graph are Gromov hyperbolic.  It is
crucial here that all holes {\em interfere}; this leads to
hyperbolicity.  When there are non-interfering holes, it is unclear
how to partition the marking path to obtain the distance estimate.


\subsection*{Acknowledgments}
We thank Jason Behrstock, Brian Bowditch, Yair Minsky, Lee
Mosher, Hossein Namazi, and Kasra Rafi for many enlightening
conversations.

We thank Tao Li for pointing out that our original bound inside of
\refthm{IncompressibleHoles} of $O(\log g(V))$ could be reduced to a
constant.

\section{Background on complexes}
\label{Sec:BackgroundComplexes}

We use $S_{g,b,c}$ to denote the compact connected surface of
genus $g$ with $b$ boundary components and $c$ cross-caps.  
If the surface is orientable we omit the subscript $c$ and write
$S_{g,b}$.  The {\em complexity} of $S = S_{g, b}$ is $\xi(S) = 3g - 3
+ b$.  If the surface is closed and orientable we simply write $S_g$.


\subsection{Arcs and curves}

A simple closed curve $\alpha \subset S$ is {\em essential} if
$\alpha$ does not bound a disk in $S$.  The curve $\alpha$ is {\em
non-peripheral} if $\alpha$ is not isotopic to a component of $\bdy
S$.  A simple arc $\beta \subset S$ is proper if $\beta \cap \bdy S =
\bdy \beta$.  An isotopy of $S$ is proper if it preserves the boundary
setwise.  A proper arc $\beta \subset S$ is {\em essential} if $\beta$
is not properly isotopic into a regular neighborhood of $\bdy S$.

Define $\calC(S)$ to be the set of isotopy classes of essential,
non-peripheral curves in $S$.  Define $\calA(S)$ to be the set of
proper isotopy classes of essential arcs.  When $S = S_{0,2}$ is an
annulus define $\calA(S)$ to be the set of essential arcs, up to
isotopies fixing the boundary pointwise. For any surface define
$\AC(S) = \calA(S) \cup \calC(S)$.


For $\alpha, \beta \in \AC(S)$ the geometric intersection number
$\iota(\alpha, \beta)$ is the minimum intersection possible between
$\alpha$ and any $\beta'$ equivalent to $\beta$.  When $S = S_{0,2}$
we do not count intersection points occurring on the boundary.  If
$\alpha$ and $\beta$ realize their geometric intersection number then
$\alpha$ is {\em tight} with respect to $\beta$.  If they do not
realize their geometric intersection then we may {\em tighten} $\beta$
until they do.

Define $\Delta \subset \AC(S)$ to be a {\em multicurve} if for all
$\alpha, \beta \in \Delta$ we have $\iota(\alpha, \beta) = 0$.
Following Harvey~\cite{Harvey81} we may impose the structure of a
simplical complex on $\AC(S)$: the simplices are exactly the
multicurves.  Also, $\calC(S)$ and $\calA(S)$ naturally span
sub-complexes.

Note that the curve complexes $\calC(S_{1,1})$ and $\calC(S_{0,4})$
have no edges.  It is useful to alter the definition in these cases.
Place edges between all vertices with geometric intersection exactly
one if $S = S_{1,1}$ or two if $S = S_{0,4}$.  In both cases the
result is the Farey graph.  Also, with the current definition
$\calC(S)$ is empty if $S = S_{0,2}$.  Thus for the annulus only we
set $\AC(S) = \calC(S) = \calA(S)$.


\begin{definition}
\label{Def:Distance}
For vertices $\alpha, \beta \in \calC(S)$ define the {\em distance}
$d_S(\alpha, \beta)$ to be the minimum possible number of edges of a
path in the one-skeleton $\calC^1(S)$ which starts at $\alpha$ and
ends at $\beta$.
\end{definition}

Note that if $d_S(\alpha, \beta) \geq 3$ then $\alpha$ and $\beta$
{\em fill} the surface $S$.  We denote distance in the one-skeleton of
$\calA(S)$ and of $\AC(S)$ by $d_\calA$ and $d_\AC$ respectively.
Recall that the geometric intersection of a pair of curves gives an
upper bound for their distance.

\begin{lemma}
\label{Lem:Hempel}
Suppose that $S$ is a compact connected surface which is not an
annulus.  For any $\alpha, \beta \in \calC^0(S)$ with $\iota(\alpha,
\beta) > 0$ we have $d_S(\alpha, \beta) \leq 2 \log_2(\iota(\alpha,
\beta)) + 2$.  \qed
\end{lemma}

\noindent
This form of the inequality, stated for closed orientable surfaces,
may be found in~\cite{Hempel01}.  A proof in the bounded orientable
case is given in~\cite{Schleimer06b}.  The non-orientable case is then
an exercise.  When $S = S_{0,2}$ an induction proves
\begin{equation}
\label{Eqn:DistanceInAnnulus}
d_X(\alpha, \beta) = 1 + \iota(\alpha, \beta)
\end{equation}
for distinct vertices $\alpha, \beta \in \calC(X)$.
See~\cite[Equation 2.3]{MasurMinsky00}.


\subsection{Subsurfaces}

Suppose that $X \subset S$ is a connected compact subsurface.  We say
$X$ is {\em essential} exactly when all boundary components of $X$ are
essential in $S$.   We say that $\alpha \in \AC(S)$ {\em cuts} $X$ if
all representatives of $\alpha$ intersect $X$.  If some representative
is disjoint then we say $\alpha$ {\em misses} $X$. 

\begin{definition}
\label{Def:CleanlyEmbedded}
An essential subsurface $X \subset S$ is {\em cleanly embedded} if for
all components $\delta \subset \bdy X$ we have: $\delta$ is isotopic
into $\bdy S$ if and only if $\delta$ is equal to a component of $\bdy
S$.
\end{definition}


\begin{definition}
\label{Def:Overlap}
Suppose $X, Y \subset S$ are essential subsurfaces.  If $X$ is cleanly
embedded in $Y$ then we say that $X$ is {\em nested} in $Y$.  If $\bdy
X$ cuts $Y$ and also $\bdy Y$ cuts $X$ then we say that $X$ and $Y$
{\em overlap}.
\end{definition}

A compact connected surface $S$ is {\em simple} if $\AC(S)$ has
finite diameter.

\begin{lemma}
\label{Lem:SimpleSurfaces}
Suppose $S$ is a connected compact surface.  The following are equivalent:
\begin{itemize}
\item $S$ is not simple.
\item The diameter of $\AC(S)$ is at least five. 
\item $S$ admits an ending lamination or $S = S_1$ or $S_{0,2}$.
\item $S$ admits a pseudo-Anosov map or $S = S_1$ or $S_{0,2}$.
\item $\chi(S) < -1$ or $S = S_{1,1}, S_1, S_{0,2}$. 
\end{itemize}
\end{lemma}

\noindent
Lemma~4.6 of~\cite{MasurMinsky99} shows that pseudo-Anosov maps have
quasi-geodesic orbits, when acting on the associated curve complex.  A
Dehn twist acting on $\calC(S_{0,2})$ has geodesic orbits.

Note that \reflem{SimpleSurfaces} is only used in this paper
when $\bdy S$ is non-empty.  The closed case is included for
completeness.

\begin{proof}[Proof sketch of \reflem{SimpleSurfaces}]
If $S$ admits a pseudo-Anosov map then the stable lamination is an
ending lamination.
If $S$ admits a filling lamination then, by an argument of
Kobayashi~\cite{Kobayashi88b}, $\AC(S)$ has infinite diameter.  (This
argument is also sketched in~\cite{MasurMinsky99}, page 124, after the
statement of Proposition~4.6.)

If the diameter of $\AC$ is infinite then the diameter is at least
five. To finish, one may check directly that all surfaces with
$\chi(S) > -2$, other than $S_{1,1}$, $S_1$ and the annulus have
$\AC(S)$ with diameter at most four.  (The difficult cases, $S_{012}$
and $S_{003}$, are discussed by Scharlemann~\cite{Scharlemann82}.)
Alternatively, all surfaces with $\chi(S) < -1$, and also $S_{1,1}$,
admit pseudo-Anosov maps.  The orientable cases follow from Thurston's
construction~\cite{Thurston88}.  Penner's
generalization~\cite{Penner88} covers the non-orientable cases.
\end{proof}

\subsection{Handlebodies and disks}

Let $V_g$ denote the {\em handlebody} of genus $g$: the three-manifold
obtained by taking a closed regular neighborhood of a polygonal,
finite, connected graph in $\RR^3$.  The genus of the boundary is the
{\em genus} of the handlebody.  A properly embedded disk $D \subset V$
is {\em essential} if $\bdy D \subset \bdy V$ is essential.


Let $\calD(V)$ be the set of essential disks $D \subset V$, up to
proper isotopy.  A subset $\Delta \subset \calD(V)$ is a multidisk if
for every $D, E \in \Delta$ we have $\iota(\bdy D, \bdy E) = 0$.
Following McCullough~\cite{McCullough91} we place a simplical
structure on $\calD(V)$ by taking multidisks to be simplices. 
As with the curve complex, define $d_\calD$ to be the distance in the
one-skeleton of $\calD(V)$.  

\subsection{Markings}
\label{Sec:Markings}

A finite subset $\mu \subset \AC(S)$ {\em fills} $S$ if for all $\beta
\in \calC(S)$ there is some $\alpha \in \mu$ so that $\iota(\alpha,
\beta) > 0$.  For any pair of finite subsets $\mu, \nu \subset \AC(S)$
we extend the intersection number:
\[
\iota(\mu,\nu) = 
  \sum_{\alpha \in \mu, \beta \in \nu} \iota(\alpha, \beta).
\]
We say that $\mu, \nu$ are {\em $L$--close} if $\iota(\mu, \nu) \leq
L$.  We say that $\mu$ is a {\em $K$--marking} if $\iota(\mu, \mu)
\leq K$.  For any $K,L$ we may define $\calM_{K,L}(S)$ to be the graph
where vertices are filling $K$--markings and edges are given by
$L$--closeness.  

As defined in~\cite{MasurMinsky00} we have:

\begin{definition}
A {\em complete clean marking} $\mu = \{ \alpha_i \} \cup \{ \beta_i
\}$ consists of
\begin{itemize}
\item 
A collection of {\em base} curves $\base(\mu) = \{ \alpha_i \}$: a
maximal simplex in $\calC(S)$.
\item 
A collection of {\em transversal} curves $\{\beta_i\}$: for each $i$
define $X_i = S \setminus \bigcup_{j \neq i} \alpha_j$ and take
$\beta_i \in \calC(X_i)$ to be a Farey neighbor of $\alpha_i$. 
\end{itemize}
\end{definition}

\noindent
If $\mu$ is a complete clean marking then $\iota(\mu, \mu) \leq 2
\xi(S) + 6 \chi(S)$.  As discussed in~\cite{MasurMinsky00} there are
two kinds of {\em elementary moves} which connected markings.  There
is a {\em twist} about a pants curve $\alpha$, replacing its
transversal $\beta$ by a new transversal $\beta'$ which is a Farey
neighbor of both $\alpha$ and $\beta$.  We can {\em flip} by swapping
the roles of $\alpha_i$ and $\beta_i$.  (In the case of the flip move,
some of the other transversals must be {\em cleaned}.)

It follows that for any surface $S$ there are choices of $K, L$ so
that $\calM(S)$ is non-empty and connected.   We use $d_\calM(\mu,
\nu)$ to denote distance in the marking graph. 

\section{Background on coarse geometry}
\label{Sec:BackgroundGeometry}

Here we review a few ideas from coarse geometry.
See~\cite{Bridson99}, \cite{CDP90}, or \cite{Gromov87} for a fuller
discussion.

\subsection{Quasi-isometry}


Suppose $r, s, A$ are non-negative real numbers, with $A \geq 1$.  If
$s \leq A \cdot r + A$ then we write $s \quasileq r$.  If $s \quasileq
r$ and $r \quasileq s$ then we write $s \quasieq r$ and call $r$ and
$s$ {\em quasi-equal} with constant $A$.
We also define the {\em cut-off function} $[r]_c$ where $[r]_c = 0$ if
$r < c$ and $[r]_c = r$ if $r \geq c$.

Suppose that $(\calX, d_\calX)$ and $(\calY, d_\calY)$ are metric
spaces.  A relation $f \from \calX \to \calY$ is an $A$--{\em
quasi-isometric embedding} for $A \geq 1$ if, for every $x, y \in
\calX$,
\[
d_\calX(x, y) \quasieq d_\calY(f(x), f(y)).
\] 
The relation $f$ is a {\em quasi-isometry}, and $\calX$ is {\em
quasi-isometric} to $\calY$, if $f$ is an $A$--quasi-isometric
embedding and the image of $f$ is $A$--{\em dense}: the
$A$--neighborhood of the image equals all of $\calY$.

\subsection{Geodesics}

Fix an interval $[u,v] \subset \RR$.  A {\em geodesic}, connecting $x$
to $y$ in $\calX$, is an isometric embedding $f \from [u, v] \to
\calX$ with $f(u) = x$ and $f(v) = y$.  Often the exact choice of $f$
is unimportant and all that matters are the endpoints $x$ and $y$.  We
then denote the image of $f$ by $[x,y] \subset \calX$.

Fix now intervals $[m,n], [p,q] \subset \ZZ$.  An
$A$--quasi-isometric embedding $g \from [m,n] \to \calX$ is called
an $A$--{\em quasi-geodesic} in $\calX$.  A function $g \from
[m,n] \to \calX$ is an $A$--{\em unparameterized quasi-geodesic}
in $\calX$ if
\begin{itemize}
\item
there is an increasing function $\rho \from [p,q] \to [m,n]$ so that $g
\circ \rho \from [p,q] \to \calX$ is an $A$--{\em quasi-geodesic} in
$\calX$ and
\item
for all $i \in [p,q-1]$, 
$\diam_\calX\left(g\left[\rho(i), \rho(i+1)\right]\right) \leq A$.
\end{itemize}
(Compare to the definition of $(K, \delta, s)$--quasi-geodesics found
in~\cite{MasurMinsky99}.)

A subset $\calY \subset \calX$ is $\Quasi$--{\em quasi-convex} if
every $\calX$--geodesic connecting a pair of points of $\calY$ lies
within a $\Quasi$--neighborhood of $\calY$.

\subsection{Hyperbolicity}

We now assume that $\calX$ is a connected graph with metric induced by
giving all edges length one.

\begin{definition}
\label{Def:GromovHyperbolic}
The space $\calX$ is $\delta$--{\em hyperbolic} if, for any three
points $x, y, z$ in $\calX$ and for any geodesics $k = [x, y]$, $g =
[y, z]$, $h = [z, x]$, the triangle $ghk$ is $\delta$--{\em slim}: the
$\delta$--neighborhood of any two sides contains the third.
\end{definition}

An important tool for this paper is the following theorem of the first
author and Minsky~\cite{MasurMinsky99}:

\begin{theorem}
\label{Thm:C(S)IsHyperbolic}
The curve complex of an orientable surface is Gromov hyperbolic. \qed
\end{theorem}

For the remainder of this section we assume that $\calX$ is
$\delta$--hyperbolic graph, $x, y, z \in \calX$ are points, and $k =
[x, y], g = [y, z], h = [z, x]$ are geodesics.

\begin{definition}
\label{Def:ProjectionToGeodesic}
We take $\rho_k \from \calX \to k$ to be the {\em closest points
relation}:
\[
\rho_k(z) = \big\{ w \in k \st \mbox{ for all $v \in k$, 
      $d_\calX(z, w) \leq d_\calX(z, v)$ } \big\}.
\]
\end{definition}

We now list several lemmas useful in the sequel.

\begin{lemma}
\label{Lem:RightTriangle}
There is a point on $g$ within distance $2\delta$ of $\rho_k(z)$.  The
same holds for $h$. \qed
\end{lemma}



\begin{lemma}
\label{Lem:ProjectionHasBoundedDiameter}
The closest points $\rho_k(z)$ have diameter at most $4\delta$.  \qed
\end{lemma}


\begin{lemma}
\label{Lem:CenterExists}
The diameter of $\rho_g(x) \cup \rho_h(y) \cup \rho_k(z)$ is at most
$6\delta$. \qed
\end{lemma}


\begin{lemma}
\label{Lem:MovePoint}
Suppose that $z'$ is another point in $\calX$ so that $d_\calX(z, z')
\leq R$.  Then $d_\calX(\rho_k(z), \rho_k(z')) \leq R + 6\delta.$
\qed
\end{lemma}


\begin{lemma}
\label{Lem:MoveGeodesic}
Suppose that $k'$ is another geodesic in $X$ so that the endpoints of
$k'$ are within distance $R$ of the points $x$ and $y$.  Then
$d_X(\rho_k(z), \rho_{k'}(z)) \leq R + 11\delta$. \qed
\end{lemma}


We now turn to a useful consequence of the Morse stability of
quasi-geodesics in hyperbolic spaces.


\begin{lemma}
\label{Lem:FirstReverse}
For every $\delta$ and $A$ there is a constant $C$ with the following
property: If $\calX$ is $\delta$--hyperbolic and $g \from [0, N] \to
\calX$ is an $A$--unparameterized quasi-geodesic then for any $m < n <
p$ in $[0, N]$ we have:
\[
d_\calX(x, y) + d_\calX(y, z) < d_\calX(x, z) + C
\]
where $x, y, z = g(m), g(n), g(p)$. \qed
\end{lemma}

\subsection{A hyperbolicity criterion}
\label{Sec:HyperbolicityCriterion}

Here we give a hyperbolicity criterion tailored to our setting.  We
thank Brian Bowditch for both finding an error in our first proof of
\refthm{HyperbolicityCriterion} and for informing us of Gilman's
work~\cite{Gilman94, Gilman02}.

Suppose that $\calX$ is a graph with all edge-lengths equal to one.
Suppose that $\gamma \from [0, N] \to \calX$ is a loop in $\calX$ with
unit speed.  Any pair of points $a, b \in [0, N]$ gives a {\em chord}
of $\gamma$.  If $a < b$, $N/4 \leq b - a$ and $N/4 \leq a + (N - b)$
then the chord is {\em $1/4$--separated}.  The length of the chord is
$d_\calX(\gamma(a), \gamma(b))$.

Following Gilman~\cite[Theorem~B]{Gilman94} we have:

\begin{theorem}
\label{Thm:Gilman}
Suppose that $\calX$ is a graph with all edge-lengths equal to one.
Then $\calX$ is Gromov hyperbolic if and only if there is a constant
$K$ so that every loop $\gamma \from [0, N] \to \calX$ has a
$1/4$--separated chord of length at most $N/7 + K$. \qed
\end{theorem}

Gilman's proof goes via the subquadratic isoperimetric inequality.  We
now give our criterion, noting that it is closely related to another
paper of Gilman~\cite{Gilman02}.

\begin{theorem}
\label{Thm:HyperbolicityCriterion}
Suppose that $\calX$ is a graph with all edge-lengths equal to one.
Then $\calX$ is Gromov hyperbolic if and only if there is a constant
$M \geq 0$ and, for all unordered pairs $x, y \in \calX^0$, there is
a connected subgraph $g_{x, y}$ containing $x$ and $y$ with the
following properties:
\begin{itemize}
\item (Local) If $d_\calX(x, y) \leq 1$ then $g_{x,y}$ has diameter at
  most $M$.
\item (Slim triangles) For all $x, y, z \in \calX^0$ the subgraph
  $g_{x,y}$ is contained in an $M$--neighborhood of $g_{y,z} \cup
  g_{z,x}$.
\end{itemize}
\end{theorem}

\begin{proof}
Suppose that $\gamma \from [0, N] \to \calX$ is a loop.  If $\epsilon$
is the empty string let $I_\epsilon = [0,N]$.  For any binary string
$\omega$ let $I_{\omega0}$ and $I_{\omega1}$ be the first and second
half of $I_\omega$.  Note that if $|\omega| \geq \lceil \log_2 N
\rceil$ then $|I_\omega| \leq 1$.

Fix a string $\omega$ and let $[a, b] = I_\omega$.  Let $g_\omega$ be
the subgraph connecting $\gamma(a)$ to $\gamma(b)$.  Note that $g_0 =
g_1$ because $\gamma(0) = \gamma(N)$.  Also, for any binary string
$\omega$ the subgraphs $g_{\omega}, g_{\omega 0}, g_{\omega 1}$ form
an $M$--slim triangle.  If $|\omega| \leq \lceil \log_2 N \rceil$ then
every $x \in g_\omega$ has some point $b \in I_\omega$ so that
\[
d_\calX(x, \gamma(b)) \leq M (\lceil \log_2 N \rceil - |\omega|) + 2M.
\]

Since $g_0$ is connected there is a point $x \in g_0$ that lies within
the $M$--neighborhoods both of $g_{00}$ and of $g_{01}$.  Pick some $b
\in I_1$ so that $d_\calX(x, \gamma(b))$ is bounded as in the previous
paragraph.  It follows that there is a point $a \in I_0$ so that $a,
b$ are $1/4$--separated and so that
\[
d_\calX(\gamma(a), \gamma(b)) \leq 2M \lceil \log_2 N \rceil + 2M.
\]
Thus there is an additive error $K$ large enough so that $\calX$
satisfies the criterion of \refthm{Gilman} and we are done. 
\end{proof}

\section{Natural maps}
\label{Sec:Natural}

There are several natural maps between the complexes and graphs
defined in \refsec{BackgroundComplexes}.  Here we review what is known
about their geometric properties, and give examples relevant to the
rest of the paper.  


\subsection{Lifting, surgery, and subsurface projection}



Suppose that $S$ is not simple. Choose a hyperbolic metric on the
interior of $S$ so that all ends have infinite areas.  Fix a compact
essential subsurface $X \subset S$ which is not a peripheral
annulus. Let $S^X$ be the cover of $S$ so that $X$ lifts
homeomorphically and so that $S^X \homeo \interior(X)$.  For any
$\alpha \in \AC(S)$ let $\alpha^X$ be the full preimage.  

Since there is a homeomorphism between $X$ and the Gromov
compactification of $S^X$ in a small abuse of notation we identify
$\AC(X)$ with the arc and curve complex of $S^X$.

\begin{definition}
\label{Def:CuttingRel}
We define the {\em cutting relation} $\kappa_X \from \AC(S) \to
\AC(X)$ as follows: $\alpha' \in \kappa_X(\alpha)$ if and only if
$\alpha'$ is an essential non-peripheral component of $\alpha^X$.
\end{definition}

Note that $\alpha$ cuts $X$ if and only if $\kappa_X(\alpha)$ is
non-empty.  Now suppose that $S$ is not an annulus.

\begin{definition}
\label{Def:SurgeryRel}
We define the {\em surgery relation} $\sigma_X \from \AC(S) \to
\calC(S)$ as follows: $\alpha' \in \sigma_S(\alpha)$ if and only if
$\alpha' \in \calC(S)$ is a boundary component of a regular
neighborhood of $\alpha \cup \bdy S$. 
\end{definition}

With $S$ and $X$ as above:

\begin{definition}
\label{Def:SubsurfaceProjection}
The {\em subsurface projection relation} $\pi_X \from \AC(S) \to
\calC(X)$ is defined as follows: If $X$ is not an annulus then define
$\pi_X = \sigma_X \circ \kappa_X$.  When $X$ is an annulus $\pi_X =
\kappa_X$.
\end{definition}


If $\alpha, \beta \in \AC(S)$ both cut $X$ we write $d_X(\alpha,
\beta) = \diam_X(\pi_X(\alpha) \cup \pi_X(\beta))$.  This is the {\em
subsurface projection distance} between $\alpha$ and $\beta$ in $X$.

\begin{lemma}
\label{Lem:SubsurfaceProjectionLipschitz}
Suppose $\alpha, \beta \in \AC(S)$ are disjoint and cut $X$.  Then
$\diam_X(\pi_X(\alpha)), d_X(\alpha, \beta) \leq 3$.  \qed
\end{lemma}


See Lemma~2.3 of~\cite{MasurMinsky00} and the remarks in the section
Projection Bounds in~\cite{Minsky10}.  

\begin{corollary}
\label{Cor:ProjectionOfPaths}
Fix $X \subset S$.  Suppose that $\{\beta_i\}_{i=0}^N$ is a path in
$\AC(S)$.  Suppose that $\beta_i$ cuts $X$ for all $i$.  Then
$d_X(\beta_0, \beta_N) \leq 3N + 3$. \qed
\end{corollary}

It is crucial to note that if some vertex of $\{ \beta_i \}$ {\em
misses} $X$ then the projection distance $d_X(\beta_0, \beta_n)$ may
be arbitrarily large compared to $n$.  \refcor{ProjectionOfPaths} can
be greatly strengthened when the path is a
geodesic~\cite{MasurMinsky00}:

\begin{theorem}
\label{Thm:BoundedGeodesicImage}[Bounded Geodesic Image]
There is constant $\GeodConst$ with the following property.  Fix $X
\subset S$.  Suppose that $\{\beta_i\}_{i=0}^n$ is a geodesic in
$\calC(S)$.  Suppose that $\beta_i$ cuts $X$ for all $i$.  Then
$d_X(\beta_0, \beta_n) \leq \GeodConst$. \qed
\end{theorem}

Here is a converse for \reflem{SubsurfaceProjectionLipschitz}.

\begin{lemma}
\label{Lem:BoundedProjectionImpliesBoundedIntersection}
For every $a \in \NN$ there is a number $b \in \NN$ with the following
property: for any $\alpha, \beta \in \AC(S)$ if $d_X(\alpha, \beta)
\leq a$ for all $X \subset S$ then $\iota(\alpha, \beta) \leq b$.
\end{lemma}

Corollary~D of~\cite{ChoiRafi07} gives a more precise relation
between projection distance and intersection number.

\begin{proof}[Proof of
\reflem{BoundedProjectionImpliesBoundedIntersection}] We only
sketch the contrapositive: Suppose we are given a sequence of curves
$\alpha_n, \beta_n$ so that $\iota(\alpha_n, \beta_n)$ tends to
infinity.  Passing to subsequences and applying elements of the
mapping class group we may assume that $\alpha_n = \alpha_0$ for all
$n$.  Setting $c_n = \iota(\alpha_0, \beta_n)$ and passing to
subsequences again we may assume that $\beta_n / c_n$ converges to
$\lambda \in \PML(S)$, the projectivization of Thurston's space of
measured laminations. Let $Y$ be any connected component of the
subsurface filled by $\lambda$, chosen so that $\alpha_0$ cuts $Y$.
Note that $\pi_Y(\beta_n)$ converges to $\lambda|_Y$.  Again applying
Kobayashi's argument~\cite{Kobayashi88b}, the distance $d_Y(\alpha_0,
\beta_n)$ tends to infinity.
\end{proof}

\subsection{Inclusions}

We now record a well known fact: 

\begin{lemma}
\label{Lem:C(S)QuasiIsometricToAC(S)}
The inclusion $\nu \from \calC(S) \to \AC(S)$ is a quasi-isometry.
The surgery map $\sigma_S \from \AC(S) \to \calC(S)$ is a
quasi-inverse for $\nu$.
\end{lemma}

\begin{proof}
Fix $\alpha, \beta \in \calC(S)$.  Since $\nu$ is an inclusion we have
$d_\AC(\alpha, \beta) \leq d_S(\alpha, \beta)$.  In the other
direction, let $\{ \alpha_i \}_{i = 0}^N$ be a geodesic in $\AC(S)$
connecting $\alpha$ to $\beta$.  Since every $\alpha_i$ cuts $S$ we
apply \refcor{ProjectionOfPaths} and deduce $d_S(\alpha, \beta) \leq
3N + 3$.  

Note that the composition $\sigma_S \circ \nu = \Id|\calC(S)$.
Also, for any arc $\alpha \in \calA(S)$ we have $d_\AC(\alpha,
\nu(\sigma_S(\alpha))) = 1$.  Finally, $\calC(S)$ is $1$--dense in
$\AC(S)$, as any arc $\gamma \subset S$ is disjoint from the one or
two curves of $\sigma_S(\gamma)$.
\end{proof}

Brian Bowditch raised the question, at the Newton Institute in August
2003, of the geometric properties of the inclusion $\calA(S) \to
\AC(S)$.  The natural assumption, that this inclusion is again a
quasi-isometric embedding, is false.  In this paper we will exactly
characterize how the inclusion distorts distance.

We now move up a dimension.  Suppose that $V$ is a handlebody and $S =
\bdy V$.  We may take any disk $D \in \calD(V)$ to its boundary $\bdy
D \in \calC(S)$, giving an inclusion $\nu \from \calD(V) \to
\calC(S)$.  It is important to distinguish the disk complex from its
image $\nu(\calD(V))$; thus we will call the image the {\em disk set}.

The first author and Minsky~\cite{MasurMinsky04} have shown:

\begin{theorem}
\label{Thm:DiskComplexConvex}
The disk set is a quasi-convex subset of the curve complex. \qed
\end{theorem}

It is natural to ask if this map is a quasi-isometric embedding.  If
so, the hyperbolicity of $\calC(V)$ immediately follows.  In fact, the
inclusion again badly distorts distance and we investigate exactly
how, below. 

\subsection{Markings and the mapping class group}

Once the connectedness of $\calM(S)$ is in hand, it is possible to use
local finiteness to show that $\calM(S)$ is quasi-isometric to the
Cayley graph of the mapping class group~\cite{MasurMinsky00}.  

Using subsurface projections the first author and
Minsky~\cite{MasurMinsky00} obtained a {\em distance estimate} for the
marking complex and thus for the mapping class group. 

\begin{theorem}
\label{Thm:MarkingGraphDistanceEstimate}
There is a constant $\CutOff = \CutOff(S)$ so that, for any $c \geq
\CutOff$ there is a constant $A$ with
\[
d_\calM(\mu, \mu') \,\, \quasieq \,\, \sum [d_X(\mu, \mu')]_c
\] 
independent of the choice of $\mu$ and $\mu'$.  Here the sum ranges
over all essential, non-peripheral subsurfaces $X \subset S$.
\end{theorem}

This, and their similar estimate for the pants graph, is a model for
the distance estimates given below.  Notice that a filling marking
$\mu \in \calM(S)$ cuts all essential, non-peripheral subsurfaces of
$S$.  It is not an accident that the sum ranges over the same set. 


\section{Holes in general and the lower bound on distance}
\label{Sec:Holes}

Suppose that $S$ is a compact connected surface.  In this paper a {\em
combinatorial complex} $\calG(S)$ will have vertices being isotopy
classes of certain multicurves in $S$.  We will assume throughout that
vertices of $\calG(S)$ are connected by edges only if there are
representatives which are disjoint.  This assumption is made only to
simplify the proofs --- all arguments work in the case where adjacent
vertices are allowed to have uniformly bounded intersection.  In all
cases $\calG$ will be connected.  There is a natural map $\nu \from
\calG \to \AC(S)$ taking a vertex of $\calG$ to the isotopy classes of
the components.  Examples in the literature include the marking
complex~\cite{MasurMinsky00}, the pants complex~\cite{Brock03a}
\cite{BehrstockEtAl05}, the Hatcher-Thurston
complex~\cite{HatcherThurston80},
the complex of separating curves~\cite{BrendleMargalit04}, the arc
complex and the curve complexes themselves. 

For any combinatorial complex $\calG$ defined in this paper {\em other
than the curve complex} we will denote distance in the one-skeleton of
$\calG$ by $d_\calG(\cdot,\cdot)$.  Distance in $\calC(S)$ will always
be denoted by $d_S(\cdot, \cdot)$.

\subsection{Holes, defined}

Suppose that $S$ is non-simple.  Suppose that $\calG(S)$ is a
combinatorial complex.  Suppose that $X \subset S$ is an cleanly
embedded subsurface.  A vertex $\alpha \in \calG$ {\em cuts} $X$ if
some component of $\alpha$ cuts $X$.

\begin{definition}
\label{Def:Hole}
We say $X \subset S$ is a {\em hole} for $\calG$ if every vertex of
$\calG$ cuts $X$.
\end{definition}


Almost equivalently, if $X$ is a hole then the subsurface projection
$\pi_X \from \calG(S) \to \calC(X)$ never takes the empty set as a
value.
Note that the entire surface $S$ is always a hole, regardless of our
choice of $\calG$.  
A boundary parallel annulus cannot be cleanly embedded (unless $S$ is
also an annulus), so generally cannot be a hole.  A hole $X \subset S$
is {\em strict} if $X$ is not homeomorphic to $S$.

We now classify the holes for $\calA(S)$.

\begin{example}
\label{Exa:HolesArcComplex}
Suppose that $S = S_{g,b}$ with $b > 0$ and consider the arc complex
$\calA(S)$.  The holes, up to isotopy, are exactly the cleanly
embedded surfaces which contain $\bdy S$.  So, for example, if $S$ is
planar then only $S$ is a hole for $\calA(S)$.  The same holds for $S
= S_{1,1}$.  In these cases it is an exercise to show that $\calC(S)$
and $\calA(S)$ are quasi-isometric.
In all other cases the arc complex admits infinitely many holes.
\end{example}

\begin{definition}
\label{Def:DiameterOfHole}
If $X$ is a hole and if $\pi_X(\calG) \subset \calC(X)$ has diameter
at least $R$ we say that the hole $X$ has {\em diameter} at least $R$.
\end{definition}




\begin{example}
Continuing the example above: Since the mapping class group acts on
the arc complex, all non-simple holes for $\calA(S)$ have infinite
diameter.
\end{example}

Suppose now that $X, X' \subset S$ are disjoint holes for $\calG$.  In
the presence of symmetry there can be a relationship between
$\pi_X|\calG$ and $\pi_{X'}|\calG$ as follows:

\begin{definition}
\label{Def:PairedHoles}
Suppose that $X, X'$ are holes for $\calG$, both of infinite diameter.
Then $X$ and $X'$ are {\em paired} if there is a homeomorphism $\tau
\from X \to X'$ and a constant $\Paired$ so that
$$
d_{X'}(\pi_{X'}(\gamma), \tau(\pi_X(\gamma))) \leq \Paired
$$ 
for every $\gamma \in \calG$.  Furthermore, if $Y \subset X$ is a hole
then $\tau$ pairs $Y$ with $Y' = \tau(Y)$.  Lastly, pairing is required
to be symmetric; if $\tau$ pairs $X$ with $X'$ then $\tau^{-1}$ pairs
$X'$ with $X$.
\end{definition}

\begin{definition}
\label{Def:Interfere}
Two holes $X$ and $Y$ {\em interfere} if either $X \cap Y \neq
\emptyset$ or $X$ is paired with $X'$ and $X' \cap Y \neq \emptyset$.
\end{definition}

Examples arise in the symmetric arc complex and in the discussion of
twisted $I$--bundles inside of a handlebody. 

\subsection{Projection to holes is coarsely Lipschitz}

The following lem\-ma is used repeatedly throughout the paper:

\begin{lemma}
\label{Lem:LipschitzToHoles}
Suppose that $\calG(S)$ is a combinatorial complex.  Suppose that $X$
is a hole for $\calG$.  Then for any $\alpha, \beta \in \calG$ we have
$$d_X(\alpha, \beta) \leq 3 + 3 \cdot d_\calG(\alpha, \beta).$$ The
additive error is required only when $\alpha = \beta$.
\end{lemma}


\begin{proof}
This follows directly from Corollary~\ref{Cor:ProjectionOfPaths} and
our assumption that vertices of $\calG$ connected by an edge represent
disjoint multicurves.
\end{proof}



\subsection{Infinite diameter holes}
\label{Sec:InfiniteDiameterHoles}

We may now state a first answer to Bowditch's question.

\begin{lemma}
\label{Lem:InfiniteDiameterHoleImpliesNotQIEmbedded}
Suppose that $\calG(S)$ is a combinatorial complex.  Suppose that
there is a strict hole $X \subset S$ having infinite diameter.  Then
$\nu \from \calG \to \AC(S)$ is not a quasi-isometric embedding. \qed
\end{lemma}




This lemma and \refexa{HolesArcComplex} completely determines when the
inclusion of $\calA(S)$ into $\AC(S)$ is a quasi-isometric embedding.
It quickly becomes clear that the set of holes tightly constrains the
intrinsic geometry of a combinatorial complex.

\begin{lemma}
\label{Lem:DisjointHolesImpliesNotHyperbolic}
Suppose that $\calG(S)$ is a combinatorial complex invariant under the
natural action of $\MCG(S)$.  Then every non-simple hole for $\calG$
has infinite diameter.  Furthermore, if $X, Y \subset S$ are disjoint
non-simple holes for $\calG$ then there is a quasi-isometric embedding
of $\ZZ^2$ into $\calG$. \qed
\end{lemma}




We will not use
Lemmas~\ref{Lem:InfiniteDiameterHoleImpliesNotQIEmbedded}
or~\ref{Lem:DisjointHolesImpliesNotHyperbolic} and so omit the proofs.
Instead our interest lies in proving the far more powerful distance
estimate (Theorems~\ref{Thm:LowerBound} and~\ref{Thm:UpperBound}) for
$\calG(S)$.

\subsection{A lower bound on distance}

Here we see that the sum of projection distances in holes gives a
lower bound for distance.

\begin{theorem}
\label{Thm:LowerBound}
Fix $S$, a compact connected non-simple surface.  Suppose that
$\calG(S)$ is a combinatorial complex.  Then there is a constant
$\CutOff$ so that for all $c \geq \CutOff$ there is a constant $A$
satisfying
\[
\sum [d_X(\alpha, \beta)]_c \quasileq d_\calG(\alpha, \beta).
\] 
Here $\alpha, \beta \in \calG$ and the sum is taken over all holes $X$
for the complex $\calG$. \qed
\end{theorem}

The proof follows the proof of Theorems~6.10 and~6.12
of~\cite{MasurMinsky00}, practically word for word.  The only changes
necessary are to
\begin{itemize}
\item
replace the sum over {\em all} subsurfaces by the sum over all holes,
\item
replace Lemma~2.5 of~\cite{MasurMinsky00}, which records how markings
differing by an elementary move project to an essential subsurface, by
\reflem{LipschitzToHoles} of this paper, which records how
$\calG$ projects to a hole.
\end{itemize}



One major goal of this paper is to give criteria sufficient obtain the
reverse inequality; \refthm{UpperBound}.

\section{Holes for the non-orientable surface}
\label{Sec:HolesNonorientable}


Fix $F$ a compact, connected, and non-orientable surface.  Let $S$ be
the orientation double cover with covering map $\rho_F \from S \to F$.
Let $\tau \from S \to S$ be the associated involution; so for all $x
\in S$, $\rho_F(x) = \rho_F(\tau(x))$.

\begin{definition}
A multicurve $\gamma \subset \AC(S)$ is {\em symmetric} if
$\tau(\gamma) \cap \gamma = \emptyset$ or $\tau(\gamma) = \gamma$.
A multicurve $\gamma$ is {\em invariant} if there is a curve or arc
$\gamma' \subset F$ so that $\gamma = \rho_F^{-1}(\gamma')$.  The same
definitions holds for subsurfaces $X \subset S$.
\end{definition}

\begin{definition}
The {\em invariant complex} $\calC^\tau(S)$ is the simplicial complex
with vertex set being isotopy classes of invariant multicurves.
There is a $k$--simplex for every collection of $k+1$ distinct isotopy
classes having pairwise disjoint representatives.
\end{definition}

Notice that $\calC^\tau(S)$ is simplicially isomorphic to $\calC(F)$.
There is also a natural map $\nu \from \calC^\tau(S) \to \calC(S)$.
We will prove:

\begin{lemma}
\label{Lem:InvariantQI}
$\nu \from \calC^\tau(S) \to \calC(S)$ is a quasi-isometric embedding.
\end{lemma}

It thus follows from the hyperbolicity of $\calC(S)$ that:

\begin{corollary}[\cite{BestvinaFujiwara07}]
\label{Cor:NonorientableCurveComplexHyperbolic}
$\calC(F)$ is Gromov hyperbolic. \qed
\end{corollary}

We begin the proof of \reflem{InvariantQI}: since $\nu$ sends adjacent
vertices to adjacent edges we have
\begin{equation}
\label{Eqn:NonOrientableLowerBound}
d_S(\alpha, \beta) \leq d_{\calC^\tau}(\alpha, \beta),
\end{equation}
as long as $\alpha$ and $\beta$ are distinct in $\calC^\tau(S)$.  In
fact, since the surface $S$ itself is a hole for $\calC^\tau(S)$ we
may deduce a slightly weaker lower bound from
\reflem{LipschitzToHoles} or indeed from \refthm{LowerBound}.

The other half of the proof of \reflem{InvariantQI} consists of
showing that $S$ is the {\em only} hole for $\calC^\tau(S)$ with large
diameter.  After a discussion of \Teich geodesics we will prove:

\begin{restate}{Lemma}{Lem:SymmetricSurfaces}
There is a constant $K$ with the following property: Suppose that
$\alpha, \beta$ are invariant multicurves in $S$.  Suppose that $X
\subset S$ is an essential subsurface where $d_X(\alpha, \beta) > K$.
Then $X$ is symmetric.
\end{restate}

From this it follows that:

\begin{corollary}
\label{Cor:NonorientableHoles}
With $K$ as in \reflem{SymmetricSurfaces}: If $X \subset S$ is a hole
for $\calC^\tau(S)$ with diameter greater than $K$ then $X = S$.
\end{corollary}

\begin{proof}
Suppose that $X \subset S$ is a strict subsurface, cleanly embedded.
Suppose that $\diam_X(\calC^\tau(S)) > K$.  Thus $X$ is symmetric.  It
follows that $\bdy X \setminus \bdy S$ is also symmetric.  Since $\bdy
X$ does not cut $X$ deduce that $X$ is not a hole for $\calC^\tau(S)$.
\end{proof}

This corollary, together with the upper bound (\refthm{UpperBound}),
proves \reflem{InvariantQI}.

\section{Holes for the arc complex}
\label{Sec:HolesArcComplex}

Here we generalize the definition of the arc complex and classify its
holes.  

\begin{definition}
\label{Def:RelativeArcComplex}
Suppose that $S$ is a non-simple surface with boundary.  Let $\Delta$
be a non-empty collection of components of $\bdy S$.  The {\em arc
complex} $\calA(S, \Delta)$ is the subcomplex of $\calA(S)$ spanned by
essential arcs $\alpha \subset S$ with $\bdy \alpha \subset \Delta$.
\end{definition}

Note that $\calA(S, \bdy S)$ and $\calA(S)$ are identical.  

\begin{lemma}
\label{Lem:ArcComplexHoles}
Suppose $X \subset S$ is cleanly embedded.  Then $X$ is a hole for
$\calA(S, \Delta)$ if and only if $\Delta \subset \bdy X$. \qed
\end{lemma}

This follows directly from the definition of a hole.  We now have an
straight-forward observation:

\begin{lemma}
\label{Lem:ArcComplexHolesIntersect}
If $X, Y \subset S$ are holes for $\calA(S, \Delta)$ then $X \cap Y
\neq \emptyset$. \qed
\end{lemma}

The proof follows immediately from \reflem{ArcComplexHoles}.
\reflem{DisjointHolesImpliesNotHyperbolic} indicates that
\reflem{ArcComplexHolesIntersect} is essential to proving that
$\calA(S, \Delta)$ is Gromov hyperbolic.

In order to prove the upper bound theorem for $\calA$ we will use
pants decompositions of the surface $S$.  In an attempt to avoid
complications in the non-orientable case we must carefully lift to the
orientation cover.  

Suppose that $F$ is non-simple, non-orientable, and has
non-empty boundary.  Let $\rho_F \from S \to F$ be the orientation
double cover and let $\tau \from S \to S$ be the induced involution.
Fix $\Delta' \subset \bdy F$ and let $\Delta = \rho_F^{-1}(\Delta')$.

\begin{definition}
We define $\calA^\tau(S, \Delta)$ to be the {\em invariant arc
complex}: vertices are invariant multi-arcs and simplices arise from
disjointness.  
\end{definition}

Again, $\calA^\tau(S, \Delta)$ is simplicially isomorphic to $\calA(F,
\Delta')$.  If $X \cap \tau(X) = \emptyset$ and $\Delta \subset X \cup
\tau(X)$ then the subsurfaces $X$ and $\tau(X)$ are paired holes, as
in \refdef{PairedHoles}.  Notice as well that all non-simple symmetric
holes $X \subset S$ for $\calA^\tau(S, \Delta)$ have infinite
diameter.

Unlike $\calA(F, \Delta')$ the complex $\calA^\tau(S, \Delta)$ may
have disjoint holes.  None\-the\-less, we have:

\begin{lemma}
\label{Lem:ArcComplexHolesInterfere}
Any two non-simple holes for $\calA^\tau(S, \Delta)$ interfere.
\end{lemma}

\begin{proof}
Suppose that $X, Y$ are holes for the $\tau$--invariant arc complex,
$\calA^\tau(S, \Delta)$.  It follows from \reflem{SymmetricSurfaces}
that $X$ is symmetric with $\Delta \subset X \cup \tau(X)$.  The same
holds for $Y$.  Thus $Y$ must cut either $X$ or $\tau(X)$.
\end{proof}

\section{Background on three-manifolds}
\label{Sec:BackgroundThreeManifolds}

Before discussing the holes in the disk complex, we record a few facts
about handlebodies and $I$--bundles.

Fix $M$ a compact connected irreducible three-manifold.  Recall that
$M$ is {\em irreducible} if every embedded two-sphere in $M$ bounds a
three-ball. Recall that if $N$ is a closed submanifold of $M$ then
$\frontier(N)$, the frontier of $N$ in $M$, is the closure of $\bdy N
\setminus \bdy M$.

\subsection{Compressions}

Suppose that $F$ is a surface embedded in $M$.  Then $F$ is {\em
compressible} if there is a disk $B$ embedded in $M$ with $B \cap \bdy
M = \emptyset$, $B \cap F = \bdy B$, and $\bdy B$ essential in $F$.
Any such disk $B$ is called a {\em compression} of $F$.

In this situation form a new surface $F'$ as follows: Let $N$ be a
closed regular neighborhood of $B$.  First remove from $F$ the annulus
$N \cap F$.  Now form $F'$ by gluing on both disk components of
$\bdy N \setminus F$.  We say that $F'$ is obtained by {\em
compressing} $F$ along $B$.  If no such disk exists we say $F$ is {\em
incompressible}.

\begin{definition}
\label{Def:BdyCompression}
A properly embedded surface $F$ is {\em boundary compressible} if
there is a disk $B$ embedded in $M$ with
\begin{itemize}
\item $\interior(B) \cap \bdy M = \emptyset$,
\item $\bdy B$ is a union of connected arcs $\alpha$ and $\beta$,
\item $\alpha \cap \beta = \bdy \alpha = \bdy \beta$,
\item $B \cap F = \alpha$ and $\alpha$ is properly embedded in $F$,
\item $B \cap \bdy M = \beta$, and
\item $\beta$ is essential in $\bdy M \setminus \bdy F$.
\end{itemize}
\end{definition}


A disk, like $B$, with boundary partitioned into two arcs is called a
{\em bigon}.  Note that this definition of boundary compression is
slightly weaker than some found in the literature; the arc $\alpha$ is
often required to be essential in $F$.  We do not require this
additional property because, for us, $F$ will usually be a properly
embedded disk in a handlebody.

Just as for compressing disks we may {\em boundary compress} $F$ along
$B$ to obtain a new surface $F'$: Let $N$ be a closed regular
neighborhood of $B$.  First remove from $F$ the rectangle $N \cap F$.
Now form $F'$ by gluing on both bigon components of $\frontier(N)
\setminus F$.  Again, $F'$ is obtained by {\em boundary compressing}
$F$ along $B$.  Note that the relevant boundary components of $F$ and
$F'$ cobound a pair of pants embedded in $\bdy M$.  If no boundary
compression exists then $F$ is {\em boundary incompressible}.

\begin{remark}
\label{Rem:SurfacesInHandlebodies}
Recall that any surface $F$ properly embedded in a handlebody $V_g$,
$g \geq 2$, is either compressible or boundary compressible.
\end{remark}


Suppose now that $F$ is properly embedded in $M$ and $\Gamma$ is a
multicurve in $\bdy M$.

\begin{remark}
\label{Rem:IntersectionNumber}
Suppose that $F'$ is obtained by a boundary compression of $F$
performed in the complement of $\Gamma$.  Suppose that $F' = F_1 \cap
F_2$ is disconnected and each $F_i$ cuts $\Gamma$.  Then $\iota(\bdy
F_i, \Gamma) < \iota(\bdy F, \Gamma)$ for $i = 1, 2$.
\end{remark}

It is often useful to restrict our attention to boundary compressions
meeting a single subsurface of $\bdy M$.  So suppose that $X \subset
\bdy M$ is an essential subsurface.  Suppose that $\bdy F$ is tight
with respect to $\bdy X$.  Suppose $B$ is a boundary compression of
$F$.  If $B \cap \bdy M \subset X$ we say that $F$ is {\em boundary
compressible into $X$}.

\begin{lemma}
\label{Lem:XCompressibleImpliesBdyXCompressible}
Suppose that $M$ is irreducible.  Fix $X$ a connected essential
subsurface of $\bdy M$.  Let $F \subset M$ be a properly embedded,
incompressible surface.  Suppose that $\bdy X$ and $\bdy F$ are tight
and that $X$ compresses in $M$.  Then either:
\begin{itemize}
\item
$F \cap X = \emptyset$, 
\item
$F$ is boundary compressible into $X$, or 
\item
$F$ is a disk with $\bdy F \subset X$.
\end{itemize}
\end{lemma}


\begin{proof}
Suppose that $X$ is compressible via a disk $E$.  Isotope $E$ to make
$\bdy E$ tight with respect to $\bdy F$.  This can be done while
maintaining $\bdy E \subset X$ because $\bdy F$ and $\bdy X$ are
tight.  Since $M$ is irreducible and $F$ is incompressible we may
isotope $E$, rel $\bdy$, to remove all simple closed curves of $F \cap
E$.  If $F \cap E$ is non-empty then an outermost bigon of $E$ gives
the desired boundary compression lying in $X$.

Suppose instead that $F \cap E = \emptyset$ but $F$ does cut $X$.  Let
$\delta \subset X$ be a simple arc meeting each of $F$ and $E$ in
exactly one endpoint.  Let $N$ be a closed regular neighborhood of
$\delta \cup E$.  Note that $\frontier(N) \setminus F$ has three
components.  One is a properly embedded disk parallel to $E$ and the
other two $B, B'$ are bigons attached to $F$.  At least one of these,
say $B'$ is trivial in the sense that $B' \cap \bdy M$ is a trivial
arc embedded in $\bdy M \setminus \bdy F$.  If $B$ is non-trivial then
$B$ provides the desired boundary compression.

Suppose that $B$ is also trivial.  It follows that $\bdy E$ and one
component $\gamma \subset \bdy F$ cobound an annulus $A \subset X$.
So $D = A \cup E$ is a disk with $(D, \bdy D) \subset (M, F)$.  As
$\bdy D = \gamma$ and $F$ is incompressible and $M$ is irreducible
deduce that $F$ is isotopic to $E$.
\end{proof}

\subsection{Band sums}

A {\em band sum} is the inverse operation to boundary compression: Fix a
pair of disjoint properly embedded surfaces $F_1, F_2 \subset M$.  Let
$F' = F_1 \cup F_2$.  Fix a simple arc $\delta \subset \bdy M$ so that
$\delta$ meets each of $F_1$ and $F_2$ in exactly one point of $\bdy
\delta$.  Let $N \subset M$ be a closed regular neighborhood of
$\delta$. Form a new surface by adding to $F' \setminus N$ the
rectangle component of $\frontier(N) \setminus F'$.
The surface $F$ obtained is the result of {\em band summing} $F_1$ to
$F_2$ along $\delta$.  Note that $F$ has a boundary compression {\em
dual} to $\delta$ yielding $F'$: that is, there is a boundary
compression $B$ for $F$ so that $\delta \cap B$ is a single point and
compressing $F$ along $B$ gives $F'$.


\subsection{Handlebodies and I-bundles}

Recall that handlebodies are irreducible.


Suppose that $F$ is a compact connected surface with at least one
boundary component.  Let $T$ be the orientation $I$--bundle over $F$.
If $F$ is orientable then $T \homeo F \cross I$.  If $F$ is not
orientable then $T$ is the unique $I$--bundle over $F$ with orientable
total space.  We call $T$ the {\em $I$--bundle} and $F$ the {\em base
space}.  Let $\rho_F \from T \to F$ be the associated bundle map.
Note that $T$ is homeomorphic to a handlebody.

If $A \subset T$ is a union of fibers of the map $\rho_F$ then $A$ is
{\em vertical} with respect to $T$.  In particular take $\bdy_v T =
\rho_F^{-1}(\bdy F)$ to be the {\em vertical boundary} of $T$.  Take
$\bdy_h T$ to be the union of the boundaries of all of the fibers:
this is the {\em horizontal boundary} of $T$.  Note that $\bdy_h T$ is
always incompressible in $T$ while $\bdy_v T$ is incompressible in $T$
as long as $F$ is not homeomorphic to a disk.

Note that, as $|\bdy_v T| \geq 1$, any vertical surface in $T$ can be
boundary compressed.  However no vertical surface in $T$ may be
boundary compressed into $\bdy_h T$.


We end this section with:

\begin{lemma}
\label{Lem:BdyIncompImpliesVertical}
Suppose that $F$ is a compact, connected surface with $\bdy F \neq
\emptyset$.  Let $\rho_F \from T \to F$ be the orientation $I$--bundle
over $F$.  Let $X$ be a component of $\bdy_h T$.  Let $D \subset T$ be
a properly embedded disk.  If
\begin{itemize}
\item $\bdy D$ is essential in $\bdy T$,
\item $\bdy D$ and $\bdy X$ are tight, and
\item $D$ cannot be boundary compressed into $X$
\end{itemize}
then $D$ may be properly isotoped to be vertical with respect to
$T$. \qed
\end{lemma}




\section{Holes for the disk complex}
\label{Sec:HolesDisk}


Here we begin to classify the holes for the disk complex, a more
difficult analysis than that of the arc complex.  To fix notation let
$V$ be a handlebody.  Let $S = S_g = \bdy V$.  Recall that there is a
natural inclusion $\nu \from \calD(V) \to \calC(S)$.

\begin{remark}
\label{Rem:ComplementOfHoleIsIncompressible}
The notion of a hole $X \subset \bdy V$ for $\calD(V)$ may be phrased
in several different ways:
\begin{itemize}
\item
every essential disk $D \subset V$ cuts the surface $X$,
\item
$\overline{S \setminus X}$ is incompressible in $V$, or
\item
$X$ is {\em disk-busting} in $V$.
\end{itemize}
\end{remark}

The classification of holes $X \subset S$ for $\calD(V)$ breaks
roughly into three cases: either $X$ is an annulus, is compressible in
$V$, or is incompressible in $V$.  In each case we obtain a result:

\begin{restate}{Theorem}{Thm:Annuli}
Suppose $X$ is a hole for $\calD(V)$ and $X$ is an annulus. Then the
diameter of $X$ is at most $5$.
\end{restate}

\begin{restate}{Theorem}{Thm:CompressibleHoles}
Suppose $X$ is a compressible hole for $\calD(V)$ with diameter at
least $15$.  Then there are a pair of essential disks $D, E \subset V$
so that
\begin{itemize}
\item 
$\bdy D, \bdy E \subset X$ and
\item
$\bdy D$ and $\bdy E$ fill $X$. 
\end{itemize}
\end{restate}

\begin{restate}{Theorem}{Thm:IncompressibleHoles} 
Suppose $X$ is an incompressible hole for $\calD(V)$ with diameter at
least $\IncompConst$.  Then there is an $I$--bundle $\rho_F \from T
\to F$ embedded in $V$ so that
\begin{itemize}
\item 
$\bdy_h T \subset S$,
\item
$X$ is isotopic in $S$ to a component of $\bdy_h T$,
\item
some component of $\bdy_v T$ is boundary parallel into $S$,
\item
$F$ supports a pseudo-Anosov map.
\end{itemize}
\end{restate}



As a corollary of these theorems we have:

\begin{corollary}
\label{Cor:LargeImpliesInfiniteDiam}
If $X$ is hole for $\calD(V)$ with diameter at least $\IncompConst$
then $X$ has infinite diameter.
\end{corollary}

\begin{proof}
If $X$ is a hole with diameter at least $\IncompConst$ then either
\refthm{CompressibleHoles} or~\refthm{IncompressibleHoles} applies.

If $X$ is compressible then Dehn twists, in opposite directions, about
the given disks $D$ and $E$ yields an automorphism $f \from V \to V$
so that $f|X$ is pseudo-Anosov.  This follows from Thurston's
construction~\cite{Thurston88}.  By \reflem{SimpleSurfaces} the
hole $X$ has infinite diameter.

If $X$ is incompressible then $X \subset \bdy_h T$ where $\rho_F \from
T \to F$ is the given $I$--bundle.  Let $f \from F \to F$ be the given
pseudo-Anosov map.  So $g$, the suspension of $f$, gives a
automorphism of $V$.  Again it follows that the hole $X$ has infinite
diameter.
\end{proof}

Applying \reflem{InfiniteDiameterHoleImpliesNotQIEmbedded} we
find another corollary:

\begin{theorem}
\label{Thm:D(V)NotQuasiIsomEmbeddedInC(S)}
If $S = \bdy V$ contains a strict hole with diameter at least
$\IncompConst$ then the inclusion $\nu \from \calD(V) \to \calC(S)$ is
not a quasi-isometric embedding. \qed
\end{theorem}


\section{Holes for the disk complex -- annuli}
\label{Sec:Annuli}


The proof of \refthm{Annuli} occupies the rest of this section.  This
proof shares many features with the proofs of
Theorems~\ref{Thm:CompressibleHoles}
and~\ref{Thm:IncompressibleHoles}.  However, the exceptional
definition of $\calC(S_{0,2})$ prevents a unified approach.  Fix $V$,
a handlebody.

\begin{theorem}
\label{Thm:Annuli}
Suppose $X$ is a hole for $\calD(V)$ and $X$ is an annulus. Then the
diameter of $X$ is at most $5$.  
\end{theorem}


We begin with:

\begin{proofclaim}
For all $D \in \calD(V)$, $|D \cap X| \geq 2$. 
\end{proofclaim}

\begin{proof}
Since $X$ is a hole, every disk cuts $X$.  Since $X$ is an annulus,
let $\alpha$ be a core curve for $X$.  If $|D \cap X| = 1$, then we
may band sum parallel copies of $D$ along an subarc of $\alpha$.
The resulting disk misses $\alpha$, a contradiction.
\end{proof}

Assume, to obtain a contradiction, that $X$ has diameter at least
$\AnnulusConst$.
Suppose that $D \in \calD(V)$ is a disk chosen to minimize $D \cap
X$.  Among all disks $E \in \calD(V)$ with $d_X(D, E) \geq
\HalfAnnulusConst$
choose one which minimizes $|D \cap E|$.  Isotope $D$ and $E$ to make
the boundaries tight and also tight with respect to $\bdy X$.
Tightening triples of curves is not canonical; nonetheless there is a
tightening so that $S \setminus (\bdy D \cup \bdy E \cup
X)$ contains no triangles.  See Figure~\ref{Fig:NoTriangles}.

\begin{figure}[htbp]
\[
\begin{array}{ccc}
\includegraphics[height = 3.5cm]{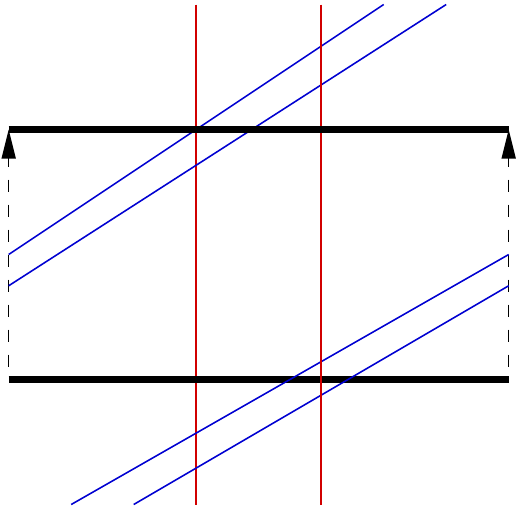} &
\hspace{0.5in} & 
\includegraphics[height = 3.5cm]{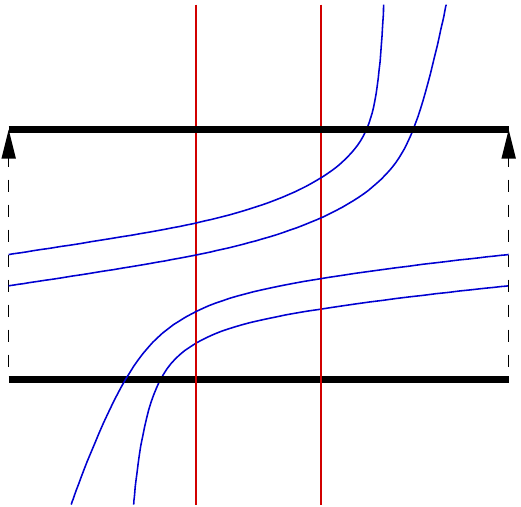} \\
\end{array}
\]
\caption{Triangles outside of $X$ (see the left side) can be moved in
  (see the right side).  This decreases the number of points of $D
  \cap E \cap (S \setminus X)$.}
\label{Fig:NoTriangles}
\end{figure}

After this tightening we have:
\begin{proofclaim}
Every arc of $\bdy D \cap X$ meets every arc of $\bdy E \cap X$ at
least once.
\end{proofclaim}


\begin{proof}
Fix components arcs $\alpha \subset D \cap X$ and $\beta \subset E
\cap X$.  Let $\alpha', \beta'$ be the corresponding arcs in $S^X$ the
annular cover of $S$ corresponding to $X$.  After the tightening we
find that 
\[
|\alpha \cap \beta| \geq |\alpha' \cap \beta'| - 1. 
\]
Since $d_X(D, E) \geq \HalfAnnulusConst$ \refeqn{DistanceInAnnulus}
implies that $|\alpha' \cap \beta'| \geq 2$.  Thus $|\alpha \cap
\beta| \geq 1$, as desired. 
\end{proof}

\begin{proofclaim}
There is an outermost bigon $B \subset E \setminus D$ with the
following properties:
\begin{itemize}
\item
$\bdy B = \alpha \cup \beta$ where $\alpha = B \cap D$, $\beta = \bdy
  B \setminus \alpha \subset \bdy E$,
\item
$\bdy \alpha = \bdy \beta \subset X$, and
\item
$|\beta \cap X| = 2$.
\end{itemize}
Furthermore, $|D \cap X| = 2$.
\end{proofclaim}

See the lower right of Figure~\ref{Fig:PossibleAlphas} for a picture.

\begin{proof}
Consider the intersection of $D$ and $E$, thought of as a collection
of arcs and curves in $E$.  Any simple closed curve component of $D
\cap E$ can be removed by an isotopy of $E$, fixed on the boundary.
(This follows from the irreducibility of $V$ and an innermost disk
argument.)  Since we have assumed that $|D \cap E|$ is minimal it
follows that there are no simple closed curves in $D \cap E$.

So consider any outermost bigon $B \subset E \setminus D$.  Let
$\alpha = B \cap D$.  Let $\beta = \bdy B \setminus \alpha = B \cap
\bdy V$.  Note that $\beta$ cannot completely contain a component of
$E \cap X$ as this would contradict either the fact that $B$ is
outermost or the claim that every arc of $E \cap X$ meets some arc of
$D \cap X$.  Using this observation, Figure~\ref{Fig:PossibleAlphas}
lists the possible ways for $B$ to lie inside of $E$.

\begin{figure}[htbp]
\labellist
\small\hair 2pt
\pinlabel {$E$} [bl] at 22 2
\pinlabel {$\alpha$} [tl] at 78 30
\endlabellist
\[
\begin{array}{cc}
\includegraphics[height = 3.2cm]{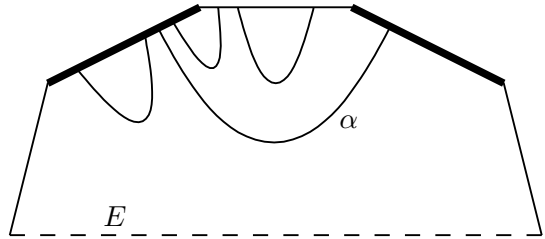}
\end{array}
\]
\caption{The arc $\alpha$ cuts a bigon $B$ off of $E$.  The darker
  part of $\bdy E$ are the arcs of $E \cap X$.  Either $\beta$ is
  disjoint from $X$, $\beta$ is contained in $X$, $\beta$ meets $X$ in
  a single subarc, or $\beta$ meets $X$ in two subarcs.}
\label{Fig:PossibleAlphas}
\end{figure}

Let $D'$ and $D''$ be the two essential disks obtained by boundary
compressing $D$ along the bigon $B$.  Suppose $\alpha$ is as shown in
one of the first three pictures of Figure~\ref{Fig:PossibleAlphas}.
It follows that either $D'$ or $D''$ has, after tightening, smaller
intersection with $X$ than $D$ does, a contradiction.  We deduce that
$\alpha$ is as pictured in lower right of
Figure~\ref{Fig:PossibleAlphas}.

Boundary compressing $D$ along $B$ still gives disks $D', D'' \in
\calD(V)$.  As these cannot have smaller intersection with $X$ we
deduce that $|D \cap X| \leq 2$ and the claim holds. 
\end{proof}

Using the same notation as in the proof above, let $B$ be an outermost
bigon of $E \setminus D$.  We now study how $\alpha \subset \bdy B$
lies inside of $D$.

\begin{proofclaim}
The arc $\alpha \subset D$ connects distinct components of $D \cap X$.
\end{proofclaim}

\begin{proof}
Suppose not.  Then there is a bigon $C \subset D \setminus \alpha$
with $\bdy C = \alpha \cup \gamma$ and $\gamma \subset \bdy D \cap X$.
The disk $C \cup B$ is essential and intersects $X$ at most once after
tightening, contradicting our first claim. 
\end{proof}

We finish the proof of Theorem~\ref{Thm:Annuli} by noting that $D \cup
B$ is homeomorphic to $\Upsilon \cross I$ where $\Upsilon$ is the
simplicial tree with three edges and three leaves.  We may choose the
homeomorphism so that $(D \cup B) \cap X = \Upsilon \cross \bdy I$.
It follows that we may properly isotope $D \cup B$ until $(D \cup B)
\cap X$ is a pair of arcs.
Recall that $D'$ and $D''$ are the disks obtained by boundary
compressing $D$ along $B$.  It follows that one of $D'$ or $D''$ (or
both) meets $X$ in at most a single arc, contradicting our first
claim. \qed

\section{Holes for the disk complex -- compressible}
\label{Sec:CompressibleHoles}


The proof of Theorem~\ref{Thm:CompressibleHoles} occupies the second
half of this section.  

\subsection{Compression sequences of essential disks}
\label{Sec:CompressionSequences}

Fix a multicurve $\Gamma \subset S = \bdy V$.  Fix also an essential
disk $D \subset V$.  Properly isotope $D$ to make $\bdy D$ tight with
respect to $\Gamma$.

If $D \cap \Gamma \neq \emptyset$ we may define:

\begin{definition}
\label{Def:Sequence}
A {\em compression sequence} $\{ \Delta_k \}_{k = 1}^n$ starting at
$D$ has $\Delta_1 = \{D\}$ and $\Delta_{k+1}$ is obtained from
$\Delta_k$ via a boundary compression, disjoint from $\Gamma$, and
tightening.
Note that $\Delta_k$ is a collection of exactly $k$ pairwise disjoint
disks properly embedded in $V$.  We further require, for $k \leq n$,
that every disk of $\Delta_k$ meets some component of $\Gamma$.  We
call a compression sequence {\em maximal} if either
\begin{itemize}
\item 
no disk of $\Delta_n$ can be boundary compressed into $S \setminus
\Gamma$ or
\item 
there is a component $Z \subset S \setminus \Gamma$ and a boundary
compression of $\Delta_n$ into $S \setminus \Gamma$ yielding an
essential disk $E$ with $\bdy E \subset Z$.
\end{itemize}
We say that such maximal sequences {\em end essentially} or {\em end in
$Z$}, respectively.
\end{definition}

All compression sequences must end, by
Remark~\ref{Rem:IntersectionNumber}.  Given a maximal sequence we may
relate the various disks in the sequence as follows:

\begin{definition}
\label{Def:DisjointnessPair}
Fix $X$, a component of $S \setminus \Gamma$.  Fix $D_k \in
\Delta_k$.  A {\em disjointness pair} for $D_k$ is an ordered
pair $(\alpha, \beta)$ of essential arcs in $X$ where
\begin{itemize}
\item $\alpha \subset D_k \cap X$,
\item $\beta \subset \Delta_n \cap X$, and
\item $d_\calA(\alpha, \beta) \leq 1$.
\end{itemize}
\end{definition}

If $\alpha \neq \alpha'$ then the two disjointness pairs $(\alpha,
\beta)$ and $(\alpha', \beta)$ are distinct, even if $\alpha$ is
properly isotopic to $\alpha'$.  A similar remark holds for the second
coordinate.  

The following lemma controls how subsurface projection distance
changes in maximal sequences.

\begin{lemma}
\label{Lem:DisjointnessPairs}
Fix a multicurve $\Gamma \subset S$.  Suppose that $D$ cuts $\Gamma$
and choose a maximal sequence starting at $D$.  Fix any component $X
\subset S \setminus \Gamma$.  Fix any disk $D_k \in \Delta_k$.  Then
either $D_k \in \Delta_n$ or there are four distinct disjointness
pairs $\{ (\alpha_i, \beta_i) \}_{i = 1}^4$ for $D_k$ in $X$ where
each of the arcs $\{ \alpha_i \}$ appears as the first coordinate of
at most two pairs.
\end{lemma}

\begin{proof}
We induct on $n - k$.  If $D_k$ is contained in $\Delta_n$ there is
nothing to prove.  If $D_k$ is contained in $\Delta_{k+1}$ we are done
by induction.  Thus we may assume that $D_k$ is the disk of $\Delta_k$
which is boundary compressed at stage $k$.  Let $D_{k+1}, D_{k+1}' \in
\Delta_{k+1}$ be the two disks obtained after boundary compressing
$D_k$ along the bigon $B$.  See \reffig{Pants} for a picture of the
pair of pants cobounded by $\bdy D_k$ and $\bdy D_{k+1} \cup \bdy
D_{k+1}'$.

\begin{figure}[htbp]
\labellist
\small\hair 2pt
\pinlabel {$\delta$} [bl] at 82 55.5
\pinlabel {$D_k$} [bl] at 184 95
\pinlabel {$D_{k+1}$} [bl] at 40 49
\pinlabel {$D'_{k+1}$} [bl] at 148.5 47.5
\pinlabel {$\Gamma$} [l] at 120 28
\endlabellist
\[
\begin{array}{c}
\includegraphics[height = 3.5cm]{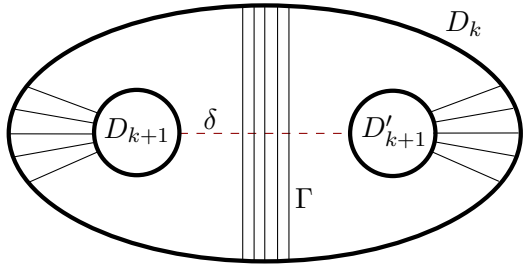}
\end{array}
\]
\caption{All arcs connecting $D_k$ to itself or to $D_{k+1} \cup
  D'_{k+1}$ are arcs of $\Gamma \cap P$.  The boundary compressing arc
  $B \cup S$ meets $D_k$ twice and is parallel to the vertical arcs of
  $\Gamma \cap P$.}
\label{Fig:Pants}
\end{figure}

Let $\delta$ be a band sum arc dual to $B$ (the dotted arc in
\reffig{Pants}).  We may assume that $|\Gamma \cap \delta|$ is minimal
over all arcs dual to $B$.  It follows that the band sum of $D_{k+1}$
with $D_{k+1}'$ along $\delta$ is tight, without any isotopy.  (This
is where we use the fact that $B$ is a boundary compression {\em in
the complement of $\Gamma$}, as opposed to being a general boundary
compression of $D_k$ in $V$.)


There are now three possibilities: neither, one, or both points of
$\bdy \delta$ are contained in $X$. 

First suppose that $X \cap \bdy \delta = \emptyset$.  Then every arc
of $D_{k+1} \cap X$ is parallel to an arc of $D_k \cap X$, and
similarly for $D_{k+1}'$.  If $D_{k+1}$ and $D_{k+1}'$ are both
components of $\Delta_n$ then choose any arcs $\beta, \beta'$ of
$D_{k+1} \cap X$ and of $D_{k+1}' \cap X$.  Let $\alpha, \alpha'$ be
the parallel components of $D_k \cap X$.  The four disjointness pairs
are then $(\alpha, \beta)$, $(\alpha, \beta')$, $(\alpha', \beta)$,
$(\alpha', \beta')$.  Suppose instead that $D_{k+1}$ is not a
component of $\Delta_n$.  Then $D_k$ inherits four disjointness pairs
from $D_{k+1}$.

Second suppose that exactly one endpoint $x \in \bdy \delta$ meets
$X$.  Let $\gamma \subset D_{k+1}$ be the component of $D_{k+1} \cap
X$ containing $x$.  Let $X'$ be the component of $X \cap P$ that
contains $x$ and let $\alpha, \alpha'$ be the two components of $D_k
\cap X'$.  Let $\beta$ be any arc of $D_{k+1}' \cap X$.

If $D_{k+1} \nin \Delta_n$ and $\gamma$ is not the first coordinate of
one of $D_{k+1}$'s four pairs then $D_k$ inherits disjointness pairs
from $D_{k+1}$.  If $D_{k+1}' \nin \Delta_n$ then $D_k$ inherits
disjointness pairs from $D_{k+1}'$.

Thus we may assume that both $D_{k+1}$ and $D_{k+1}'$ are in
$\Delta_n$ {\em or} that only $D_{k+1}' \in \Delta_n$ while $\gamma$
appears as the first arc of disjointness pair for $D_{k+1}$.  In case
of the former the required disjointness pairs are $(\alpha, \beta)$,
$(\alpha', \beta)$, $(\alpha, \gamma)$, and $(\alpha', \gamma)$.  In
case of the latter we do not know if $\gamma$ is allowed to appear as
the second coordinate of a pair.  However we are given four
disjointness pairs for $D_{k+1}$ and are told that $\gamma$ appears as
the first coordinate of at most two of these pairs.  Hence the other
two pairs are inherited by $D_k$.  The pairs $(\alpha, \beta)$ and
$(\alpha', \beta)$ give the desired conclusion.

Third suppose that the endpoints of $\delta$ meet $\gamma \subset
D_{k+1}$ and $\gamma' \subset D_{k+1}'$.  Let $X'$ be a component of
$X \cap P$ containing $\gamma$.  Let $\alpha$ and $\alpha'$ be the two
arcs of $D_k \cap X'$. 
Suppose both $D_{k+1}$ and $D_{k+1}'$ lie in $\Delta_n$.  Then the
desired pairs are $(\alpha, \gamma)$, $(\alpha', \gamma)$, $(\alpha,
\gamma')$, and $(\alpha', \gamma')$.  If $D_{k+1}' \in \Delta_n$ while
$D_{k+1}$ is not then $D_k$ inherits two pairs from $D_{k+1}$.  We add
to these the pairs $(\alpha, \gamma')$, and $(\alpha', \gamma')$.  If
neither disk lies in $\Delta_n$ then $D_k$ inherits two pairs from
each disk and the proof is complete.
\end{proof}

Given a disk $D \in \calD(V)$ and a hole $X \subset S$ our
\reflem{DisjointnessPairs} allows us to adapt $D$ to $X$. 

\begin{lemma}
\label{Lem:DistanceIsUnchangedInSequences}
Fix a hole $X \subset S$ for $\calD(V)$.  For any disk $D \in
\calD(V)$ there is a disk $D'$ with the following properties:
\begin{itemize}
\item $\bdy X$ and $\bdy D'$ are tight.
\item If $X$ is incompressible then $D'$ is not boundary compressible
  into $X$ and $d_\calA(D, D') \leq 3$.
\item If $X$ is compressible then $\bdy D' \subset X$ and
  $d_\AC(D, D') \leq 3$.
\end{itemize}
Here $\calA = \calA(X)$ and $\AC = \AC(X)$.
\end{lemma}


\begin{proof}
If $\bdy D \subset X$ then the lemma is trivial.  So assume, by
Remark~\ref{Rem:ComplementOfHoleIsIncompressible}, that $D$ cuts $\bdy
X$.  Choose a maximal sequence with respect to $\bdy X$ starting at
$D$.

Suppose that the sequence is non-trivial ($n > 1$). By
\reflem{DisjointnessPairs} there is a disk $E \in \Delta_n$ so that $D
\cap X$ and $E \cap X$ contain disjoint arcs.

If the sequence ends essentially then choose $D' = E$ and the lemma is
proved.  If the sequence ends in $X$ then there is a boundary
compression of $\Delta_n$, disjoint from $\bdy X$, yielding the
desired disk $D'$ with $\bdy D' \subset X$.  Since $E \cap D' =
\emptyset$ we again obtain the desired bound.

Assume now that the sequence is trivial ($n = 1$).  Then take $E = D
\in \Delta_n$ and the proof is identical to that of the previous
paragraph.  
\end{proof}


\begin{remark}
\label{Rem:WhyDistanceIsUnchangedInSequences}
\reflem{DistanceIsUnchangedInSequences} is unexpected: after all, any
pair of curves in $\calC(X)$ can be connected by a sequence of band
sums.  Thus arbitrary band sums can change the subsurface projection
to $X$.  However, the sequences of band sums arising in
\reflem{DistanceIsUnchangedInSequences} are very special.  Firstly
they do not cross $\bdy X$ and secondly they are ``tree-like'' due to
the fact every arc in $D$ is separating.

When $D$ is replaced by a surface with genus then
\reflem{DistanceIsUnchangedInSequences} does not hold in general; this
is a fundamental observation due to Kobayashi~\cite{Kobayashi88b} (see
also~\cite{Hartshorn02}).  
Namazi points out that even if $D$ is only replaced by a planar
surface \reflem{DistanceIsUnchangedInSequences} does not hold in
general.
\end{remark}





\subsection{Proving the theorem}

We now prove:

\begin{theorem}
\label{Thm:CompressibleHoles}
Suppose $X$ is a compressible hole for $\calD(V)$ with diameter at
least $15$.  Then there are a pair of essential disks $D, E \subset V$
so that
\begin{itemize}
\item 
$\bdy D, \bdy E \subset X$ and
\item
$\bdy D$ and $\bdy E$ fill $X$. 
\end{itemize}
\end{theorem}

\begin{proof}
Choose disks $D'$ and $E'$ in $\calD(V)$ so that $d_X(D', E') \geq
15$.  By \reflem{DistanceIsUnchangedInSequences} there are
disks $D$ and $E$ so that $\bdy D, \bdy E \subset X$, $d_X(D', D) \leq
6$, and $d_X(E', E) \leq 6$.  It follows from the triangle inequality
that $d_X(D, E) \geq 3$.
\end{proof}

\section{Holes for the disk complex -- incompressible}
\label{Sec:IncompressibleHoles}

This section classifies incompressible holes for the disk complex.

\begin{theorem}
\label{Thm:IncompressibleHoles}
Suppose $X$ is an incompressible hole for $\calD(V)$ with diameter at
least $\IncompConst$.  Then there is an $I$--bundle $\rho_F \from T \to
F$ embedded in $V$ so that
\begin{itemize}
\item 
$\bdy_h T \subset \bdy V$,
\item
$X$ is a component of $\bdy_h T$, 
\item
some component of $\bdy_v T$ is boundary parallel into $\bdy V$,
\item
$F$ supports a pseudo-Anosov map.
\end{itemize}
\end{theorem}


Here is a short plan of the proof: We are given $X$, an incompressible
hole for $\calD(V)$.  Following
\reflem{DistanceIsUnchangedInSequences} we may assume that $D, E$ are
essential disks, without boundary compressions into $X$ or $S
\setminus X$, with $d_X(D,E) > \RealIncompConst$.
Examine the intersection pattern of $D$ and $E$ to find two families
of rectangles $\calR$ and $\calQ$. The intersection pattern of these
rectangles in $V$ will determine the desired $I$--bundle $T$.  The
third conclusion of the theorem follows from standard facts about
primitive annuli.  The fourth requires another application of
\reflem{DistanceIsUnchangedInSequences} as well as
\reflem{SimpleSurfaces}.

\subsection{Diagonals of polygons}

To understand the intersection pattern of $D$ and $E$ we discuss
diagonals of polygons.  Let $D$ be a $2n$ sided regular polygon.
Label the sides of $D$ with the letters $X$ and $Y$ in alternating
fashion.  Any side labeled $X$ (or $Y$) will be called an {\em $X$
side} (or {\em $Y$ side}).

\begin{definition}
\label{Def:Diagonal}
An arc $\gamma$ properly embedded in $D$ is a {\em diagonal} if the
points of $\bdy \gamma$ lie in the interiors of distinct sides of $D$.
If $\gamma$ and $\gamma'$ are diagonals for $D$ which together meet
three different sides then $\gamma$ and $\gamma'$ are {\em
non-parallel}.
\end{definition}

\begin{lemma}
\label{Lem:EightDiagonals}
Suppose that $\Gamma \subset D$ is a collection of pairwise disjoint
non-parallel diagonals.  Then there is an $X$ side of $D$ meeting at
most eight diagonals of $\Gamma$.
\end{lemma}

\begin{proof}
A counting argument shows that $|\Gamma| \leq 4n - 3$.  If every $X$
side meets at least nine non-parallel diagonals then $|\Gamma| \geq
\frac{9}{2}n > 4n -3$, a contradiction.
\end{proof}


\subsection{Improving disks}
\label{Sec:ImprovingDisks}

Suppose now that $X$ is an incompressible hole for $\calD(V)$ with
diameter at least $\IncompConst$.  Note that, by
Theorem~\ref{Thm:Annuli}, $X$ is not an annulus.  Let $Y = \overline{S
\setminus X}$.



Choose disks $D'$ and $E'$ in $V$ so that $d_X(D', E') \geq
\IncompConst$.  By \reflem{DistanceIsUnchangedInSequences}
there are a pair of disks $D$ and $E$ so that both are essential in
$V$, cannot be boundary compressed into $X$ or into $Y$, and so that
$d_{\calA(X)}(D', D) \leq 3$ and $d_{\calA(X)}(E', E) \leq 3$.  Thus
$d_X(D', D) \leq 9$ and $d_X(E', E) \leq 9$
(\reflem{LipschitzToHoles}).  By the triangle inequality $d_X(D, E)
\geq \IncompConst - 18 = \RealIncompConst$.

Recall, as well, that $\bdy D$ and $\bdy E$ are tight with respect to
$\bdy X$.  We may further assume that $\bdy D$ and $\bdy E$ are tight
with respect to each other.  Also, minimize the quantities $|X \cap
(\bdy D \cap \bdy E)|$ and $|D \cap E|$ while keeping everything
tight.  In particular, there are no triangle components of $\bdy V
\setminus (D \cup E \cup \bdy X)$.  
Now consider $D$ and $E$ to be even-sided polygons, with vertices
being the points $\bdy D \cap \bdy X$ and $\bdy E \cap \bdy X$
respectively.  Let $\Gamma = D \cap E$.  See
Figure~\ref{Fig:RectangleWithBadArcs} for one \apriori~possible
collection $\Gamma \subset D$.

\begin{figure}[htbp]
\[
\begin{array}{c}
\includegraphics[height = 3.5cm]{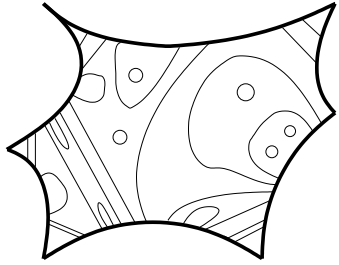}
\end{array}
\]
\caption{In fact, $\Gamma \subset D$ cannot contain simple closed
  curves or non-diagonals.}
\label{Fig:RectangleWithBadArcs}
\end{figure}

From our assumptions and the irreducibility of $V$ it follows that
$\Gamma$ contains no simple closed curves.  Suppose now that there is
a $\gamma \subset \Gamma$ so that, in $D$, both endpoints of $\gamma$
lie in the same side of $D$.  Then there is an outermost such arc, say
$\gamma' \subset \Gamma$, cutting a bigon $B$ out of $D$.  It follows
that $B$ is a boundary compression of $E$ which is disjoint from $\bdy
X$.  But this contradicts the construction of $E$.
We deduce that all arcs of $\Gamma$ are diagonals for $D$ and, via a
similar argument, for $E$.

Let $\alpha \subset D \cap X$ be an $X$ side of $D$ meeting at most
eight distinct types of diagonal of $\Gamma$.  Choose $\beta \subset E
\cap X$ similarly.  As $d_X(D, E) \geq \RealIncompConst$ we have that
$d_X(\alpha, \beta) \geq \RealIncompConst - 6 = \ReallyIncompConst$.

Now break each of $\alpha$ and $\beta$ into at most eight subarcs $\{
\alpha_i \}$ and $\{ \beta_j \}$ so that each subarc meets all of the
diagonals of fixed type and only of that type.  Let $R_i \subset D$ be
the rectangle with upper boundary $\alpha_i$ and containing all of the
diagonals meeting $\alpha_i$.  Let $\alpha_i'$ be the lower boundary
of $R_i$.  Define $Q_j$ and $\beta_j'$ similarly.  See
\reffig{RectanglesInDisks} for a picture of $R_i$.

\begin{figure}[htbp]
\labellist
\small\hair 2pt
\pinlabel {$R_i$} [l] at 92 61
\pinlabel {$\alpha_i$} [bl] at 93 146
\pinlabel {$\alpha_i'$} [tr] at 63 3
\endlabellist
\[
\begin{array}{c}
\includegraphics[height = 3.5cm]{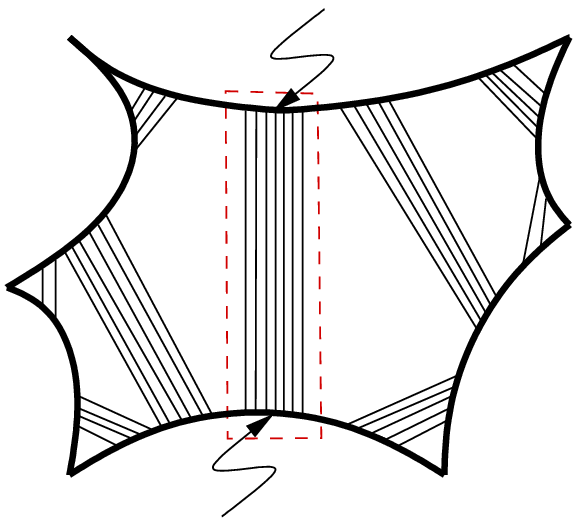}
\end{array}
\]
\caption{The rectangle $R_i \subset D$ is surrounded by the dotted line.
  The arc $\alpha_i$ in $\bdy D \cap X$ is indicated.  In general the
  arc $\alpha'_i$ may lie in $X$ or in $Y$.}
\label{Fig:RectanglesInDisks}
\end{figure}

Call an arc $\alpha_i$ {\em large} if there is an arc $\beta_j$ so
that $|\alpha_i \cap \beta_j| \geq 3$.  We use the same notation for
$\beta_j$. Let $\Theta$ be the union of all of the large $\alpha_i$
and $\beta_j$.  Thus $\Theta$ is a four-valent graph in $X$.  Let
$\Theta'$ be the union of the corresponding large $\alpha_i'$ and
$\beta_i'$.  

\begin{claim}
\label{Clm:ThetaNonEmpty}
The graph $\Theta$ is non-empty. 
\end{claim}

\begin{proof}
If $\Theta = \emptyset$, then all $\alpha_i$ are small.  It follows
that $|\alpha \cap \beta| \leq 128$ and thus $d_X(\alpha, \beta) \leq
16$, by \reflem{Hempel}.  As $d_X(\alpha, \beta) \geq
\ReallyIncompConst$ this is a contradiction.
\end{proof}


Let $Z \subset \bdy V$ be a small regular neighborhood of $\Theta$ and
define $Z'$ similarly.

\begin{claim}
\label{Clm:ThetaEssential}
No component of $\Theta$ or of $\Theta'$ is contained in a disk $D
\subset \bdy V$.  No component of $\Theta$ or of $\Theta'$ is
contained in an annulus $A \subset \bdy V$ that is peripheral in $X$.
\end{claim}

\begin{proof}
For a contradiction suppose that $W$ is a component of $Z$ contained
in a disk.  Then there is some pair $\alpha_i, \beta_j$ having a bigon
in $\bdy V$.  This contradicts the tightness of $\bdy D$ and $\bdy
E$. The same holds for $Z'$. 

Suppose now that some component $W$ is contained in an annulus $A$,
peripheral in $X$.  Thus $W$ fills $A$.  Suppose that $\alpha_i$ and
$\beta_j$ are large and contained in $W$.  By the classification of
arcs in $A$ we deduce that either $\alpha_i$ and $\beta_j$ form a
bigon in $A$ or $\bdy X$, $\alpha_i$ and $\beta_j$ form a triangle.
Either conclusion gives a contradiction.
\end{proof}

\begin{claim}
\label{Clm:ThetaFills}
The graph $\Theta$ fills $X$. 
\end{claim}

\begin{proof}
Suppose not.  Fix attention on any component $W \subset Z$.  Since
$\Theta$ does not fill, the previous claim implies that there is a
component $\gamma \subset \bdy W$ that is essential and non-peripheral
in $X$.  Note that any large $\alpha_i$ meets $\bdy W$ in at most two
points, while any small $\alpha_i$ meets $\bdy W$ in at most $32$
points. Thus $|\alpha \cap \bdy W| \leq 256$ and the same holds for
$\beta$.  Thus $d_X(\alpha, \beta) \leq 36$ by the triangle
inequality.  As $d_X(\alpha, \beta) \geq \ReallyIncompConst$ this is a
contradiction.
\end{proof}


The previous two claims imply:

\begin{claim}
\label{Clm:ThetaConnected}
The graph $\Theta$ is connected. \qed
\end{claim}

There are now two possibilities: either $\Theta \cap \Theta'$ is empty
or not.  In the first case set $\Sigma = \Theta$ and in the second set
$\Sigma = \Theta \cup \Theta'$.  By the claims above, $\Sigma$ is
connected and fills $X$.  Let $\calR = \{ R_i \}$ and $\calQ = \{ Q_j
\}$ be the collections of large rectangles.

\subsection{Building the I-bundle}

We are given $\Sigma$, $\calR$ and $\calQ$ as above.  Note that $\calR
\cup \calQ$ is an $I$--bundle and $\Sigma$ is the component of its
horizontal boundary meeting $X$.  See Figure~\ref{Fig:RandQ} for a
simple case. 

\begin{figure}[htbp]
\labellist
\small\hair 2pt
\pinlabel {$R_i$} [l] at 312 173
\pinlabel {$Q_j$} [l] at 388 371
\endlabellist
\[
\begin{array}{c}
\includegraphics[height = 4.5cm]{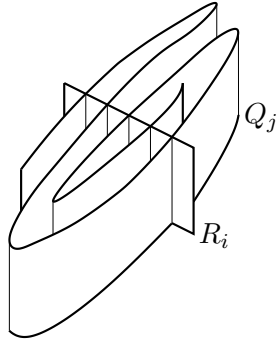}
\end{array}
\]
\caption{$\calR \cup \calQ$ is an $I$--bundle: all arcs of intersection
are parallel.}
\label{Fig:RandQ}
\end{figure}

Let $T_0$ be a regular neighborhood of $\calR \cup \calQ$, taken in
$V$.  Again $T_0$ has the structure of an $I$--bundle.  Note that
$\bdy_h T_0 \subset \bdy V$, $\bdy_h T_0 \cap X$ is a component of
$\bdy_h T_0$, and this component fills $X$ due to
Claim~\ref{Clm:ThetaFills}.  We will enlarge $T_0$ to obtain the
correct $I$--bundle in $V$.

Begin by enumerating all annuli $\{ A_i \} \subset \bdy_v T_0$ with
the property that some component of $\bdy A_i$ is inessential in $\bdy
V$.  Suppose that we have built the $I$--bundle $T_i$ and are now
considering the annulus $A = A_i$.  Let $\gamma \cup \gamma' = \bdy A
\subset \bdy V$ with $\gamma$ inessential in $\bdy V$.  Let $B \subset
\bdy V$ be the disk which $\gamma$ bounds.  By induction we assume
that no component of $\bdy_h T_i$ is contained in a disk embedded in
$\bdy V$ (the base case holds by Claim~\ref{Clm:ThetaEssential}).  It
follows that $B \cap T_i = \bdy B = \gamma$.  Thus $B \cup A$ is
isotopic, rel $\gamma'$, to be a properly embedded disk $B' \subset
V$.  As $\gamma'$ lies in $X$ or $Y$, both incompressible, $\gamma'$
must bound a disk $C \subset \bdy V$.  Note that $C \cap T_i = \bdy C
= \gamma'$, again using the induction hypothesis.

It follows that $B \cup A \cup C$ is an embedded two-sphere in $V$.
As $V$ is a handlebody $V$ is irreducible.  Thus $B \cup A \cup C$
bounds a three-ball $U_i$ in $V$.  Choose a homeomorphism $U_i \homeo
B \cross I$ so that $B$ is identified with $B \cross \{0\}$, $C$ is
identified with $B \cross \{1\}$, and $A$ is identified with $\bdy B
\cross I$.  We form $T_{i+1} = T_i \cup U_i$ and note that $T_{i+1}$
still has the structure of an $I$--bundle.  Recalling that $A = A_i$ we
have $\bdy_v T_{i+1} = \bdy_v T_i \setminus A_i$.  Also $\bdy_h
T_{i+1} = \bdy_h T_i \cup (B \cup C) \subset \bdy V$.  It follows that no
component of $\bdy_h T_{i+1}$ is contained in a disk embedded in $\bdy
V$.  Similarly, $\bdy_h T_{i+1} \cap X$ is a component of $\bdy_h
T_{i+1}$ and this component fills $X$.


After dealing with all of the annuli $\{ A_i \}$ in this fashion we
are left with an $I$--bundle $T$.  Now all components of $\bdy \bdy_v
T$ [\sic] are essential in $\bdy V$.  All of these lying in $X$ are
peripheral in $X$.  This is because they are disjoint from $\Sigma
\subset \bdy_h T$, which fills $X$, by induction.  It follows that the
component of $\bdy_h T$ containing $\Sigma$ is isotopic to $X$.

This finishes the construction of the promised $I$--bundle $T$ and
demonstrates the first two conclusions of
\refthm{IncompressibleHoles}.  For future use we record:

\begin{remark}
\label{Rem:CornersAreEssential}
Every curve of $\bdy \bdy_v T = \bdy \bdy_h T$ is essential in $S =
\bdy V$.
\end{remark}

\subsection{A vertical annulus parallel into the boundary}

Here we obtain the third conclusion of \refthm{IncompressibleHoles}: at
least one component of $\bdy_v T$ is boundary parallel in $\bdy V$.

Fix $T$ an $I$--bundle with the incompressible hole $X$ a component of
$\bdy_h T$.  


\begin{claim}
\label{Clm:VerticalAnnuliIncompressible}
All components of $\bdy_v T$ are incompressible in $V$.
\end{claim}

\begin{proof}
Suppose that $A \subset \bdy_v T$ was compressible.  By
Remark~\ref{Rem:CornersAreEssential} we may compress $A$ to obtain a
pair of essential disks $B$ and $C$.  Note that $\bdy B$ is isotopic
into the complement of $\bdy_h T$.  So $\overline{S \setminus X}$ is
compressible, contradicting
Remark~\ref{Rem:ComplementOfHoleIsIncompressible}.
\end{proof}

\begin{claim}
\label{Clm:VerticalAnnulusBdyParallel}
Some component of $\bdy_v T$ is boundary parallel.
\end{claim}

\begin{proof}
Since $\bdy_v T$ is incompressible
(Claim~\ref{Clm:VerticalAnnuliIncompressible}) by
Remark~\ref{Rem:SurfacesInHandlebodies}, we find that $\bdy_v T$ is
boundary compressible in $V$.  Let $B$ be a boundary compression for
$\bdy_v T$.  Let $A$ be the component of $\bdy_v T$ meeting $B$.  Let
$\alpha$ denote the arc $A \cap B$.

The arc $\alpha$ is either essential or inessential in $A$.  Suppose
$\alpha$ is inessential in $A$.  Then $\alpha$ cuts a bigon, $C$, out
of $A$. Since $B$ was a boundary compression the disk $D = B \cup C$
is essential in $V$.  Since $B$ meets $\bdy_v T$ in a single arc,
either $D \subset T$ or $D \subset \overline{V \setminus T}$.  The
former implies that $\bdy_h T$ is compressible and the latter that $X$
is not a hole.  Either gives a contradiction.

It follows that $\alpha$ is essential in $A$.  Now carefully boundary
compress $A$: Let $N$ be the closure of a regular neighborhood of $B$,
taken in $V \setminus A$.  Let $A'$ be the closure of $A \setminus N$
(so $A'$ is a rectangle).  Let $B' \cup B''$ be the closure of
$\frontier(N) \setminus A$.  Both $B'$ and $B''$ are bigons, parallel
to $B$.  Form $D = A' \cup B' \cup B''$: a properly embedded disk in
$V$.  If $D$ is essential then, as above, either $D \subset T$ or $D
\subset \overline{V \setminus T}$.  Again, either gives a
contradiction.

It follows that $D$ is inessential in $V$.  Thus $D$ cuts a closed
three-ball $U$ out of $V$.  There are two final cases: either $N
\subset U$ or $N \cap U = B' \cup B''$.  If $U$ contains $N$ then $U$
contains $A$.  Thus $\bdy A$ is contained in the disk $U \cap \bdy V$.
This contradicts Remark~\ref{Rem:CornersAreEssential}.  Deduce instead
that $W = U \cup N$ is a solid torus with meridional disk $B$.  Thus
$W$ gives a parallelism between $A$ and the annulus $\bdy V \cap \bdy
W$, as desired.
\end{proof}

\begin{remark}
\label{Rem:DiskBusting}
Similar considerations prove that the multicurve 
$$\{ \bdy A \st 
       \mbox{$A$ is a boundary parallel component of $\bdy_v T$} \}$$
is disk-busting for $V$.
\end{remark}

\subsection{Finding a pseudo-Anosov map}

Here we prove that the base surface $F$ of the $I$--bundle $T$ admits
a pseudo-Anosov map.

As in Section~\ref{Sec:ImprovingDisks}, pick essential disks $D'$ and
$E'$ in $V$ so that $d_X(D', E') \geq \IncompConst$.
\reflem{DistanceIsUnchangedInSequences} provides disks $D$ and $E$
which cannot be boundary compressed into $X$ or into $\overline{S
\setminus X}$ -- thus $D$ and $E$ cannot be boundary compressed into
$\bdy_h T$.  Also, as above, $d_X(D, E) \geq \IncompConst - 18 =
\RealIncompConst$.

After isotoping $D$ to minimize intersection with $\bdy_v T$ it must
be the case that all components of $D \cap \bdy_v T$ are essential
arcs in $\bdy_v T$.  By \reflem{BdyIncompImpliesVertical} we conclude
that $D$ may be isotoped in $V$ so that $D \cap T$ is vertical in $T$.
The same holds of $E$.  Choose $A$ and $B$, components of $D \cap T$
and $E \cap T$.  Each are vertical rectangles.  Since
$\diam_X(\pi_X(D)) \leq 3$ (\reflem{SubsurfaceProjectionLipschitz}) we
now have $d_X(A, B) \geq \RealIncompConst - 6 = 37$.

We now begin to work in the base surface $F$.  Recall that $\rho_F
\from T \to F$ is an $I$--bundle.  Take $\alpha = \rho_F(A)$ and
$\beta = \rho_F(B)$.  Note that the natural map $\calC(F) \to
\calC(X)$, defined by taking a curve to its lift, is distance
non-increasing (see Equation~\ref{Eqn:NonOrientableLowerBound}).  Thus
$d_F(\alpha, \beta) \geq 37$.  By \refthm{Annuli} the
surface $F$ cannot be an annulus.  Thus, by \reflem{SimpleSurfaces}
the subsurface $F$ supports a pseudo-Anosov map and we are done.

\subsection{Corollaries}

We now deal with the possibility of disjoint holes for the disk
complex.  

\begin{lemma}
\label{Lem:DiskComplexPairedHoles}
Suppose that $X$ is a large incompressible hole for $\calD(V)$
supported by the $I$--bundle $\rho_F \from T \to F$.  Let $Y = \bdy_h
T \setminus X$.  Let $\tau \from \bdy_h T \to \bdy_h T$ be the
involution switching the ends of the $I$--fibres.  Suppose that $D \in
\calD(V)$ is an essential disk.
\begin{itemize}
\item
If $F$ is orientable then $d_{\calA(F)}(D \cap X, D \cap Y) \leq 6$.
\item
If $F$ is non-orientable then $d_X(D, \calC^\tau(X)) \leq 3$.
\end{itemize}
\end{lemma}


\begin{proof}
By \reflem{DistanceIsUnchangedInSequences} there is a disk $D' \subset
V$ which is tight with respect to $\bdy_h T$ and which cannot be
boundary compressed into $\bdy_h T$ (or into the complement).  Also,
for any component $Z \subset \bdy_h T$ we have $d_{\calA(Z)}(D, D')
\leq 3$.

Properly isotope $D'$ to minimize $D' \cap \bdy_v T$.  Then $D' \cap
\bdy_v T$ is properly isotopic, in $\bdy_v T$, to a collection of
vertical arcs.  Let $E \subset D' \cap T$ be a component.
\reflem{BdyIncompImpliesVertical} implies that $E$ is vertical in $T$,
after an isotopy of $D'$ preserving $\bdy_h T$ setwise.  Since $E$ is
vertical, the arcs $E \cap \bdy_h T \subset D'$ are
$\tau$--invariant. The conclusion follows.
\end{proof}

Recall \reflem{ArcComplexHolesIntersect}: all holes for the arc
complex intersect.  This cannot hold for the disk complex.  For
example if $\rho_F \from T \to F$ is an $I$--bundle over an orientable
surface then take $V = T$ and notice that both components of $\bdy_h
T$ are holes for $\calD(V)$.  However, by the first conclusion of
\reflem{DiskComplexPairedHoles}, $X$ and $Y$ are paired holes, in the
sense of \refdef{PairedHoles}.  So, as with the invariant arc complex
(\reflem{ArcComplexHolesInterfere}), all holes for the disk complex
interfere:

\begin{lemma}
\label{Lem:DiskComplexHolesInterfere}
Suppose that $X, Z \subset \bdy V$ are large holes for $\calD(V)$.  If
$X \cap Z = \emptyset$ then there is an $I$--bundle $T \homeo F \cross
I$ in $V$ so that $\bdy_h T = X \cup Y$ and $Y \cap Z \neq \emptyset$.
\end{lemma}

\begin{proof}
Suppose that $X \cap Z = \emptyset$.  It follows from
Remark~\ref{Rem:ComplementOfHoleIsIncompressible} that both $X$ and
$Z$ are incompressible.  Let $\rho_F \from T \to F$ be the $I$--bundle
in $V$ with $X \subset \bdy_h T$, as provided by
\refthm{IncompressibleHoles}.  We also have a component $A \subset
\bdy_v T$ so that $A$ is boundary parallel.  Let $U$ be the solid
torus component of $V \setminus A$.  Note that $Z$ cannot be contained
in $\bdy U \setminus A$ because $Z$ is not an annulus (\refthm{Annuli}). 

Let $\alpha = \rho_F(A)$.  Choose any essential arc $\delta \subset F$
with both endpoints in $\alpha \subset \bdy F$.  It follows that
$\rho_F^{-1}(\delta)$, together with two meridional disks of $U$, forms
an essential disk $D$ in $V$.  Let $W = \bdy_h T \cup (U \setminus
A)$ and note that $\bdy D \subset W$.  

If $F$ is non-orientable then $Z \cap W = \emptyset$ and we have a
contradiction.  Deduce that $F$ is orientable. Now, if $Z$ misses $Y$
then $Z$ misses $W$ and we again have a contradiction.  It follows
that $Z$ cuts $Y$ and we are done.
\end{proof}


\section{Axioms for combinatorial complexes}
\label{Sec:Axioms}

The goal of this section and the next is to prove, inductively, an
upper bound on distance in a combinatorial complex $\calG(S) = \calG$.
This section presents our axioms on $\calG$: sufficient hypotheses
for \refthm{UpperBound}.  The axioms, apart from \refax{Holes}, are
quite general.  \refax{Holes} is necessary to prove hyperbolicity and
greatly simplifies the recursive construction in \refsec{Partition}.

\begin{theorem}
\label{Thm:UpperBound}
Fix $S$ a compact connected non-simple surface.  Suppose that $\calG =
\calG(S)$ is a combinatorial complex satisfying the axioms of
\refsec{Axioms}.  Let $X$ be a hole for $\calG$ and suppose that
$\alpha_X, \beta_X \in \calG$ are contained in $X$.  For any constant
$c > 0$ there is a constant $A$ satisfying:
\[
d_\calG(\alpha_X, \beta_X) \quasileq \sum [d_Y(\alpha_X, \beta_X)]_c
\]
where the sum is taken over all holes $Y \subseteq X$ for $\calG$.
\end{theorem}


The proof of the upper bound is more difficult than that of the lower
bound, \refthm{LowerBound}.  This is because naturally occurring paths
in $\calG$ between $\alpha_X$ and $\beta_X$ may waste time in
non-holes.  The first example of this is the path in $\calC(S)$
obtained by taking the short curves along a \Teich geodesic.  The
\Teich geodesic may spend time rearranging the geometry of a
subsurface.  Then the systole path in the curve complex must be much
longer than the curve complex distance between the endpoints.

In Sections~\ref{Sec:PathsNonorientable}, \ref{Sec:PathsArc},
\ref{Sec:PathsDisk} we will verify these axioms for the curve complex
of a non-orientable surface, the arc complex, and the disk complex.

\subsection{The axioms}

Suppose that $\calG = \calG(S)$ is a combinatorial complex.  We begin
with the axiom required for hyperbolicity. 


\begin{axiom}[Holes interfere]
\label{Ax:Holes}
All large holes for $\calG$ interfere, as given in \refdef{Interfere}.
\end{axiom}

Fix vertices $\alpha_X, \beta_X \in \calG$, both contained in a hole
$X$.  We are given $\Lambda = \{ \mu_n \}_{n = 0}^N$, a path of
markings in $X$.

\begin{axiom}[Marking path]
\label{Ax:Marking}
We require:
\begin{enumerate}
\item
The support of $\mu_{n+1}$ is contained inside the support of $\mu_n$.
\item
For any subsurface $Y \subseteq X$, if $\pi_Y(\mu_k) \neq \emptyset$
then for all $n \leq k$ the map $n \mapsto \pi_Y(\mu_n)$ is an
unparameterized quasi-geodesic with constants depending only on
$\calG$.
\end{enumerate}
\end{axiom}


\noindent
The second condition is crucial and often technically difficult to
obtain.

We are given, for every essential subsurface $Y \subset X$, a perhaps
empty interval $J_Y \subset [0, N]$ with the following properties.

\begin{axiom}[Accessibility]
\label{Ax:Access}
The interval for $X$ is $J_X = [0, N]$. There is a constant
$\AccessTemp$ so that
\begin{enumerate}
\item 
If $m \in J_Y$ then $Y$ is contained in the support of $\mu_m$.
\item 
If $m \in J_Y$ then $\iota(\partial Y, \mu_m) < \AccessTemp$.
\item
If $[m, n] \cap J_Y = \emptyset$ then $d_Y(\mu_m, \mu_n) < \AccessTemp$.
\end{enumerate}
\end{axiom}

\noindent
There is a combinatorial path $\Gamma = \{ \gamma_i \}_{i = 0}^K
\subset \calG$ starting with $\alpha_X$ ending with $\beta_X$ and each
$\gamma_i$ is contained in $X$.  There is a strictly increasing
reindexing function $r \from [0, K] \to [0, N]$ with $r(0) = 0$ and
$r(K) = N$. 

\begin{axiom}[Combinatorial]
\label{Ax:Combin}
There is a constant $\Combin$ so that:
\begin{itemize}
\item
$d_Y(\gamma_i, \mu_{r(i)}) < \Combin$, for every $i \in [0, K]$ and
  every hole $Y \subset X$,
\item
$d_\calG(\gamma_i, \gamma_{i+1}) < \Combin$, for every $i \in [0, K -
1]$.
\end{itemize}
\end{axiom}



\begin{axiom}[Replacement]
\label{Ax:Replace}
There is a constant $\Replace$ so that:
\begin{enumerate}
\item
If $Y \subset X$ is a hole and $r(i) \in J_Y$ then there is a vertex
$\gamma' \in \calG$ so that $\gamma'$ is contained in $Y$ and
$d_\calG(\gamma_i, \gamma') < \Replace$.
\item
If $Z \subset X$ is a non-hole and $r(i) \in J_Z$ then there is a
vertex $\gamma' \in \calG$ so that $d_\calG(\gamma_i, \gamma') <
\Replace$ and so that $\gamma'$ is contained in $Z$ or in $X \setminus
Z$.
\end{enumerate}
\end{axiom}

There is one axiom left: the axiom for straight intervals.  This is
given in the next subsection.

\subsection{Inductive, electric, shortcut and straight intervals}

We describe subintervals that arise in the partitioning of $[0,
K]$. As discussed carefully in \refsec{Deductions}, we will choose a
lower threshold $\Lower(Y)$ for every essential $Y \subset X$ and a
general upper threshold, $\Upper$.

\begin{definition}
\label{Def:Inductive}
Suppose that $[i, j] \subset [0, K]$ is a subinterval of the
combinatorial path.  Then $[i, j]$ is an {\em inductive interval}
associated to a hole $Y \subsetneq X$ if
\begin{itemize}
\item
$r([i, j]) \subset J_Y$ (for paired $Y$ we require $r([i, j]) \subset
  J_Y \cap J_{Y'}$) and
\item
$d_Y(\gamma_i, \gamma_j) \geq \Lower(Y)$.
\end{itemize}
\end{definition} 

When $X$ is the only relevant hole we have a simpler definition:

\begin{definition}
\label{Def:Electric}
Suppose that $[i, j] \subset [0, K]$ is a subinterval of the
combinatorial path.  Then $[i, j]$ is an {\em electric interval} if
$d_Y(\gamma_i, \gamma_j) < \Upper$ for all holes $Y \subsetneq X$.
\end{definition}

Electric intervals will be further partitioned into shortcut and
straight intervals.

\begin{definition}
\label{Def:Shortcut}
Suppose that $[p, q] \subset [0, K]$ is a subinterval of the
combinatorial path.  Then $[p, q]$ is a {\em shortcut} if
\begin{itemize}
\item
$d_Y(\gamma_p, \gamma_q) < \Upper$ for all holes $Y$, including
  $X$ itself, and
\item
there is a non-hole $Z \subset X$ so that $r([p, q]) \subset J_Z$.
\end{itemize}
\end{definition}

\begin{definition}
\label{Def:Straight}
Suppose that $[p, q] \subset [0, K]$ is a subinterval of the
combinatorial path and is contained in an electric interval $[i, j]$.
Then $[p, q]$ is a {\em straight interval} if $d_Y(\mu_{r(p)},
\mu_{r(q)}) < \Upper$ for all non-holes $Y$.
\end{definition}

Our final axiom is:

\begin{axiom}[Straight]
\label{Ax:Straight}
There is a constant $A$ depending only on $X$ and $\calG$ so that for
every straight interval $[p, q]$:
\[
d_\calG(\gamma_p, \gamma_q) \quasileq d_X(\gamma_p, \gamma_q)
\]
\end{axiom}

\subsection{Deductions from the axioms}
\label{Sec:Deductions}

\refax{Marking} and \reflem{FirstReverse} imply that the reverse
triangle inequality holds for projections of marking paths.

\begin{lemma}
\label{Lem:Reverse} 
There is a constant $\Reverse$ so that
\[
d_Y(\mu_m, \mu_n) + d_Y(\mu_n, \mu_p) < d_Y(\mu_m, \mu_p) + \Reverse
\]
for every essential $Y \subset X$ and for every $m < n < p$ in
$[0, N]$. \qed
\end{lemma}

We record three simple consequences of \refax{Access}. 

\begin{lemma}
\label{Lem:Access}
There is a constant $\Access$, depending only on $\AccessTemp$, with the
follow properties:
\begin{itemize}
\item[(i)] If $Y$ is strictly nested in $Z$ and $m \in J_Y$ then
  $d_Z(\partial Y,\mu_m) \leq \Access$.
\item[(ii)] If $Y$ is strictly nested in $Z$ then for any $m, n \in
  J_Y$, $d_Z(\mu_m, \mu_n) < \Access$.
\item[(iii)] If $Y$ and $Z$ overlap then for any $m, n \in J_Y \cap
  J_Z$ we have $d_Y(\mu_m, \mu_n), d_Z(\mu_m, \mu_n) < \Access$. 
\end{itemize}
\end{lemma}

\begin{proof}
We first prove conclusion (i): Since $Y$ is strictly nested in $Z$ and
since $Y$ is contained in the support of $\mu_m$ (part (1) of
\refax{Access}), both $\bdy Y$ and $\mu_m$ cut $Z$.  By
\refax{Access}, part (2), we have that $\iota(\bdy Y, \mu_m) \leq
\AccessTemp$.  It follows that $\iota(\bdy Y, \pi_Z(\mu_m)) \leq
2\AccessTemp$.  By \reflem{Hempel} we deduce that $d_Z(\bdy Y, \mu_m)
\leq 2 \log_2 \AccessTemp + 3$.  We take $\Access$ larger than this
right hand side.

Conclusion (ii) follows from a pair of applications of conclusion (i)
and the triangle inequality.

For conclusion (iii): As in (ii), to bound $d_Z(\mu_m, \mu_n)$ it
suffices to note that $\bdy Y$ cuts $Z$ and that $\bdy Y$ has bounded
intersection with both of $\mu_m, \mu_n$.
\end{proof}

We now have all of the constants $\Reverse, \Access, \Combin, \Replace$
in hand.  Recall that $\Paired$ is the pairing constant of
\refdef{PairedHoles} and that $\GeodConst$ is the constant of
\ref{Thm:BoundedGeodesicImage}.  We must choose a lower threshold
$\Lower(Y)$ for every essential $Y \subset X$.  We must also choose
the general upper threshold, $\Upper$ and general lower threshold
$\Induct$.  We require, for all essential $Z, Y$ in $X$, with $\xi(Z)
< \xi(Y) \leq \xi(X)$:
\begin{gather}
\label{Eqn:InductBigger} 
  \Induct > \Access + 2\Combin + 2\Paired \\
\label{Eqn:UpperBigger}  
   \Upper > \Lower(X) + 2\Paired + 6\Reverse + 2\Combin + 14\Access +
   10 \\
\label{Eqn:LowerBigger} 
   \Lower(Y) > \GeodConst + 2\Access + 4\Combin + 2\Paired +\Induct \\
\label{Eqn:Order} 
   \Lower(X) > \Lower(Z) + 2\Access + 4\Combin + 4\Paired 
\end{gather}

%

\section{Partition and the upper bound on distance}
\label{Sec:Partition}

In this section we prove \refthm{UpperBound} by induction on
$\xi(X)$. The first stage of the proof is to describe the {\em
inductive partition}: we partition the given interval $[0, K]$ into
inductive and electric intervals.  The inductive partition is closely
linked with the hierarchy machine~\cite{MasurMinsky00} and with the
notion of antichains introduced in~\cite{RafiSchleimer09}.

We next give the {\em electric partition}: each electric interval is
divided into straight and shortcut intervals.  Note that the electric
partition also gives the base case of the induction.  We finally bound
$d_\calG(\alpha_X, \beta_X)$ from above by combining the contributions
from the various intervals. 

\subsection{Inductive partition}

We begin by identifying the relevant surfaces for the construction of
the partition.  We are given a hole $X$ for $\calG$ and vertices
$\alpha_X, \beta_X \in \calG$ contained in $X$.  Define
\[
B_X = \{ Y \subsetneq X \st \mbox{$Y$ is a hole and~}
   d_Y(\alpha_X, \beta_X) \geq \Lower(X) \}.
\]
For any subinterval $[i,j] \subset [0, K]$ define
\[ 
B_X(i, j) = \{ Y \in B_X \st d_Y(\gamma_i, \gamma_j) \geq \Lower(X)\}.
\]

We now partition $[0, K]$ into inductive and electric intervals.
Begin with the partition of one part $\calP_X = \{ [0, K] \}$.
Recursively $\calP_X$ is a partition of $[0, K]$ consisting of
intervals which are either inductive, electric, or undetermined.
Suppose that $[i, j] \in \calP_X$ is undetermined.

\begin{proofclaim}
If $B_X(i,j)$ is empty then $[i,j]$ is electric.  
\end{proofclaim}

\begin{proof}
Since $B_X(i, j)$ is empty, every hole $Y \subsetneq X$ has either
$d_Y(\gamma_i, \gamma_j) < \Lower(X)$ or $Y \nin B_X$.  In the former
case, as $\Lower(X) < \Upper$, we are done.

So suppose the latter holds.  Now, by the reverse triangle inequality
(\reflem{Reverse}),
\[
d_Y(\mu_{r(i)}, \mu_{r(j)}) < d_Y(\mu_0, \mu_N) + 2\Reverse.
\]
Since $r(0) = 0$ and $r(K) = N$ we find:
\[
d_Y(\gamma_i, \gamma_j) < 
   d_Y(\alpha_X, \beta_X) + 2\Reverse + 4\Combin.
\]
Deduce that 
\[
d_Y(\gamma_i, \gamma_j) < \Lower(X) + 2\Reverse + 4\Combin <
\Upper.
\]
This completes the proof.
\end{proof}

Thus if $B_X(i, j)$ is empty then $[i, j] \in \calP_X$ is determined
to be electric.  Proceed on to the next undetermined element.  Suppose
instead that $B_X(i,j)$ is non-empty.  Pick a hole $Y \in B_X(i,j)$ so
that $Y$ has maximal $\xi(Y)$ amongst the elements of $B_X(i,j)$ 

Let $p, q \in [i,j]$ be the first and last indices, respectively, so
that $r(p), r(q) \in J_Y$.  (If $Y$ is paired with $Y'$ then we take
the first and last indices that, after reindexing, lie inside of $J_Y
\cap J_{Y'}$.)  

\begin{proofclaim}
The indices $p, q$ are well-defined.
\end{proofclaim}

\begin{proof}
By assumption $d_Y(\gamma_i, \gamma_j) \geq \Lower(X)$.  By
\refeqn{InductBigger},
\[
\Lower(X) > \Access+2\Combin.
\]
We deduce from \refax{Access} and \refax{Combin} that $J_Y \cap r([i,
j])$ is non-empty.  Thus, if $Y$ is not paired, the indices $p, q$ are
well-defined.

Suppose instead that $Y$ is paired with $Y'$.  Recall that
measurements made in $Y$ and $Y'$ differ by at most the pairing
constant $\Paired$ given in \refdef{PairedHoles}.  By
(\ref{Eqn:LowerBigger}), $$\Lower(X) >\Access + 2\Combin + 2\Paired.$$
We deduce again from \refax{Access} that $J_{Y'} \cap
r([i, j])$ is non-empty.

Suppose now, for a contradiction, that $J_Y \cap J_{Y'} \cap r([i,
j])$ is empty.  Define
$$
h = \max \{ \ell \in [i, j] \st r(\ell) \in J_Y \}, \quad 
k = \min \{ \ell \in [i, j] \st r(\ell) \in J_{Y'} \}
$$ 

Without loss of generality we may assume that $h < k$.  It follows
that $d_{Y'}(\gamma_i, \gamma_h) < \Access + 2\Combin$.  Thus
$d_Y(\gamma_i, \gamma_h) < \Access + 2\Combin + 2\Paired$.  Similarly,
$d_Y(\gamma_h, \gamma_j) < \Access + 2\Combin$.  Deduce
$$ 
d_Y(\gamma_i, \gamma_j) < 2\Access + 4\Combin + 2\Paired < \Lower(X)
,$$ the last inequality by (\ref{Eqn:LowerBigger}).
This is a contradiction to the assumption.  
\end{proof}

\begin{proofclaim}
The interval $[p, q]$ is inductive for $Y$.
\end{proofclaim}

\begin{proof}
We must check that $d_Y(\gamma_p, \gamma_q) \geq \Lower(Y)$.  Suppose
first that $Y$ is not paired.  Then by the definition of $p, q$,
  (2) of \refax{Access},  and the triangle inequality we have
$$ 
d_Y(\mu_{r(i)}, \mu_{r(j)}) \leq d_Y(\mu_{r(p)}, \mu_{r(q)}) + 2\Access.
$$
Thus by \refax{Combin},
$$ 
d_Y(\gamma_i, \gamma_j) \leq d_Y(\gamma_p, \gamma_q) + 2\Access + 4\Combin.
$$
Since by (\ref{Eqn:Order}), 
$$ 
\Lower(Y) + 2\Access + 4\Combin < \Lower(X) \leq d_Y(\gamma_i, \gamma_j)
$$
we are done.

When $Y$ is paired the proof is similar but we must use the slightly
stronger inequality $\Lower(Y) + 2\Access + 4\Combin + 4\Paired <
\Lower(X)$.
\end{proof}

Thus, when $B_X(i, j)$ is non-empty we may find a hole $Y$ and indices
$p, q$ as above.  In this situation, we subdivide the element $[i, j]
\in \calP_X$ into the elements $[i, p-1]$, $[p, q]$, and $[q+1, j]$.
(The first or third intervals, or both, may be empty.)  The interval
$[p, q] \in \calP_X$ is determined to be inductive and associated to
$Y$.  Proceed on to the next undetermined element.  This completes the
construction of $\calP_X$.

As a bit of notation, if $[i, j] \in \calP_X$ is associated to $Y
\subset X$ we will sometimes write $I_Y = [i, j]$.  

\subsection{Properties of the inductive partition}

\begin{lemma}
\label{Lem:HolesContained}
Suppose that $Y, Z$ are holes and $I_Z$ is an inductive element of
$\calP_X$ associated to $Z$.  Suppose that $r(I_Z) \subset J_Y$ (or
$r(I_Z) \subset J_Y \cap J_{Y'}$, if $Y$ is paired).  Then 
\begin{itemize}
\item
$Z$ is nested in $Y$ or 
\item
$Z$ and $Z'$ are paired and $Z'$ is nested in $Y$.
\end{itemize}
\end{lemma}

\begin{proof}
Let $I_Z = [i, j]$.  Suppose first that $Y$ is strictly nested in $Z$.
Then by (ii) of \reflem{Access}, $d_Z(\mu_{r(i)}, \mu_{r(j)})
< \Access$.  Then by Axiom \ref{Ax:Combin}
\[
d_Z(\gamma_i, \gamma_j) < \Access + 2\Combin < \Lower(Z),
\]
a contradiction.  We reach the same contradiction if $Y$ and $Z$
overlap using (iii) of \reflem{Access}.

Now, if $Z$ and $Y$ are disjoint then there are two cases: Suppose
first that $Y$ is paired with $Y'$.  Since all holes interfere, $Y'$
and $Z$ must meet.  In this case we are done, just as in the previous
paragraph. 

Suppose now that $Z$ is paired with $Z'$.  Since all holes interfere,
$Z'$ and $Y$ must meet.  If $Z'$ is nested in $Y$ then we are done.
If $Y$ is strictly nested in $Z'$ then, as $r([i, j]) \subset J_Y$, we
find that as above by Axioms \ref{Ax:Combin} and (ii) of
\reflem{Access} that $$d_{Z'}(\gamma_i, \gamma_j) < \Access +
2\Combin$$ and so $d_Z(\gamma_i, \gamma_j) < \Access + 2\Combin +
2\Paired < \Lower(Z)$, a contradiction.  We reach the same
contradiction if $Y$ and $Z'$ overlap.
\end{proof}

\begin{proposition}
\label{Prop:AtMostThreeInductives}
Suppose $Y \subsetneq X$ is a hole for $\calG$.
\begin{enumerate}
\item
If $Y$ is associated to an inductive interval $I_Y \in \calP_X$ and
$Y$ is paired with $Y'$ then $Y'$ is not associated to any inductive
interval in $\calP_X$.
\item 
There is at most one inductive interval $I_Y \in \calP_X$ associated
to $Y$.
\item 
There are at most two holes $Z$ and $W$, distinct from $Y$ (and
from $Y'$, if $Y$ is paired) such that
\begin{itemize}
\item
there are inductive intervals  $I_Z = [h, i]$ and $I_W = [j, k]$ and
\item
$d_Y(\gamma_h, \gamma_i), d_Y(\gamma_j, \gamma_k) \geq \Induct$.
\end{itemize}
\end{enumerate}
\end{proposition}

\begin{remark}
\label{Rem:AtMostThreeInductive}
It follows that for any hole $Y$ there are at most three inductive
intervals in the partition $\calP_X$ where $Y$ has projection distance
greater than $\Induct$.  
\end{remark}

\begin{proof}[Proof of \refprop{AtMostThreeInductives}]
To prove the first claim: Suppose that $I_Y = [p, q]$ and $I_{Y'} =
[p', q']$ with $q < p'$.  It follows that $[r(p), r(q')] \subset J_Y
\cap J_{Y'}$.  If $q + 1 = p'$ then the partition would have chosen a
larger inductive interval for one of $Y$ or $Y'$.  It must be the case
that there is an inductive interval $I_Z \subset [q + 1, p' - 1]$ for
some hole $Z$, distinct from $Y$ and $Y'$, with $\xi(Z) \geq \xi(Y)$.
However, by \reflem{HolesContained} we find that $Z$ is nested in $Y$
or in $Y'$.  It follows that $Z = Y$ or $Y$, a contradiction.

The second statement is essentially similar.  

Finally suppose that $Z$ and $W$ are the first and last holes, if any,
satisfying the hypotheses of the third claim.  Since $d_Y(\gamma_h,
\gamma_i) \geq \Induct$ we find by \refax{Combin} that
$$d_Y(\mu_{r(h)}, \mu_{r(i)}) \geq \Induct - 2\Combin.$$ By
(\ref{Eqn:InductBigger}), $\Induct - 2\Combin > \Access$ so that
$$J_Y \cap r(I_Z) \neq \emptyset.$$ If $Y$ is paired then, again by
(\ref{Eqn:InductBigger}) we have $\Induct> \Access + 2\Combin +
2\Paired$, we also find that $J_{Y'} \cap r(I_Z) \neq \emptyset$.
Symmetrically, $J_Y \cap r(I_W)$ (and $J_{Y'} \cap r(I_W)$) are also
non-empty.

It follows that the interval between $I_Z$ and $I_W$, after
reindexing, is contained in $J_Y$ (and $J_{Y'}$, if $Y$ is paired).
Thus for any inductive interval $I_V = [p, q]$ between $I_Z$ and $I_W$
the associated hole $V$ is nested in $Y$ (or $V'$ is nested in $Y$),
by \reflem{HolesContained}.  If $V = Y$ or $V = Y'$ there is nothing
to prove. Suppose instead that $V$ (or $V'$) is strictly nested in
$Y$.  It follows that $$d_Y(\gamma_p, \gamma_q) < \Access + 2\Combin <
\Induct.$$  Thus there are no inductive intervals between $I_Z$ and
$I_W$ satisfying the hypotheses of the third claim.
\end{proof}

The following lemma and proposition bound the number of inductive
intervals.  The discussion here is very similar to (and in fact
inspired) the {\em antichains} defined in~\cite[Section
5]{RafiSchleimer09}.  Our situation is complicated by the presence of
non-holes and interfering holes.

\begin{lemma}
\label{Lem:PigeonForHoles}
Suppose that $X, \alpha_X, \beta_X$ are given, as above.  For any
$\ell \geq (3 \cdot \Upper)^{\xi(X)}$, if $\{Y_i\}_{i = 1}^\ell$ is a
collection of distinct strict sub-holes of $X$ each having
$d_{Y_i}(\alpha_X, \beta_X)\geq \Lower(X)$ then there is a hole $Z
\subseteq X$ such that $d_Z(\alpha_X, \beta_X) \geq \Upper - 1$ and
$Z$ contains at least $\Upper$ of the $Y_i$.  Furthermore, for at
least $\Upper - 4(\Reverse + \Access + 2\Access + 2)$ of these
$Y_i$ we find that $J_{Y_i} \subsetneq J_Z$.  (If $Z$ is paired then
$J_{Y_i} \subsetneq J_Z \cap J_{Z'}$.)  Each of these $Y_i$ is
disjoint from a distinct vertex $\eta_i \in [ \pi_Z(\alpha_X),
\pi_Z(\beta_X) ]$.
\end{lemma}

\begin{proof}
Let $g_X$ be a geodesic in $\calC(X)$ joining $\alpha_X, \beta_X$.  By
the Bounded Geodesic Image Theorem (\refthm{BoundedGeodesicImage}),
since $\Lower(X) > \GeodConst$, for every $Y_i$ there is a vertex
$\omega_i\in g_X$ such that $Y_i\subset X \setminus \omega_i$. Thus
$d_X(\omega_i,\partial Y_i) \leq 1$.  If there are at least $\Upper$
distinct $\omega_i$, associated to distinct $Y_i$, then $d_X(\alpha_X,
\beta_X) \geq \Upper - 1$.  In this situation we take $Z = X$.  Since
$J_X = [0, N]$ we are done.

Thus assume there do not exist at least $\Upper$ distinct $\omega_i$.
Then there is some fixed $\omega$ among these $\omega_i$ such that at
least $\frac{\ell}{\Upper}\geq 3 (3 \cdot \Upper)^{\xi(X)-1}$ of the
$Y_i$ satisfy
\[
Y_i \subset (X \setminus \omega).
\] 
Thus one component, call it $W$, of $X\setminus \omega$ contains at
least $(3 \cdot \Upper)^{\xi(X)-1}$ of the $Y_i$.
Let $g_W$ be a geodesic in $\calC(W)$ joining $\alpha_W =
\pi_W(\alpha_X)$ and $\beta_W=\pi_W(\beta_X)$.  Notice that
$$d_{Y_i}(\alpha_W, \beta_W) \geq d_{Y_i}(\alpha_X, \beta_X) - 8$$
because we are projecting to nested subsurfaces.  This follows for
example from \reflem{SubsurfaceProjectionLipschitz}.  Hence
$d_{Y_i}(\alpha_W, \beta_W) \geq \Lower(W)$.

Again apply \refthm{BoundedGeodesicImage}.  Since $\Lower(W) >
\GeodConst$, for every remaining $Y_i$ there is a vertex $\eta_i \in
g_W$ such that
\[
Y_i \subset (W \setminus \eta_i)
\] 
If there are at least $\Upper$ distinct $\eta_i$ then we take $Z =
W$.  Otherwise we repeat the argument.  Since the complexity of each
successive subsurface is decreasing by at least $1$, we must
eventually find the desired $Z$ containing at least $\Upper$ of the
$Y_i$, each disjoint from distinct vertices of $g_Z$.  

So suppose that there are at least $\Upper$ distinct $\eta_i$
associated to distinct $Y_i$ and we have taken $Z = W$.  Now we must
find at least $\Upper - 4(\Reverse + \Access + 2\Access + 2)$
of these $Y_i$ where $J_{Y_i} \subsetneq J_Z$.

To this end we focus attention on a small subset $\{Y^j\}_{j = 1}^5
\subset \{Y_i\}$.  Let $\eta_j$ be the vertex of $g_Z = g_W$
associated to $Y^j$.  We choose these $Y^j$ so that
\begin{itemize}
\item
the $\eta_j$ are arranged along $g_Z$ in order of index and
\item
$d_Z(\eta_j, \eta_{j + 1}) > 
\Reverse + \Access + 2\Access + 2$, 
for $j = 1, 2, 3, 4$.
\end{itemize}
This is possible by (\ref{Eqn:UpperBigger}) because 
\[
\Upper > 4(\Reverse + \Access + 2\Access).
\]
Set $J_j = J_{Y^j}$ and pick any indices $m_j \in J_j$.  (If $Z$ is
paired then $Y^j$ is as well and we pick $m_j \in J_{Y^j} \cap
J_{(Y^j)'}$.)  We use $\mu(m_j)$ to denote $\mu_{m_j}$.  Since $\bdy
Y^j$ is disjoint from $\eta_j$, \refax{Access} and \reflem{Hempel}
imply
\begin{equation}
\label{Eqn:Eta}
d_Z(\mu(m_j), \eta_j) \leq \Access + 1.
\end{equation}

Since the sequence $\pi_Z(\mu_n)$ satisfies the reverse triangle
inequality (\reflem{Reverse}), it follows that the $m_j$
appear in $[0, N]$ in order agreeing with their index.  The triangle
inequality implies that 
\[
d_Z(\mu(m_1), \mu(m_2)) > \Access.
\]
Thus \refax{Access} implies that $J_Z \cap [m_1, m_2]$
is non-empty.  Similarly, $J_Z \cap [m_4, m_5]$ is non-empty.  It
follows that $[m_2, m_4] \subset J_Z$.  (If $Z$ is paired then, after
applying the symmetry $\tau$ to $g_Z$, the same argument proves $[m_2,
m_4] \subset J_{Z'}$.)

Notice that $J_2 \cap J_3 = \emptyset$.  For if $m \in J_2 \cap J_3$
then by (\ref{Eqn:Eta}) both $d_Z(\mu_m, \eta_2)$ and $d_Z(\mu_m,
\eta_3)$ are bounded by $\Access + 1$.  It follows that
\[
d_Z(\eta_2, \eta_3) < 2\Access + 2,
\]
a contradiction.
Similarly $J_3 \cap J_4 = \emptyset$.  We deduce that $J_3 \subsetneq
[m_2, m_4] \subset J_Z$.  (If $Z$ is paired $J_3 \subset J_Z \cap
J_{Z'}$.)  Finally, there are at least
$$
\Upper - 4(\Reverse + \Access + 2\Access + 2)$$
possible $Y_i$'s which satisfy the hypothesis on $Y^3$.   This
completes the proof. 
\end{proof}

Define
$$
\calP_\induct = \{ I \in \calP_X \st \mbox{ $I$ is inductive}\}.
$$

\begin{proposition}
\label{Prop:NumberOfInductives}
The number of inductive intervals is a lower bound for the projection
distance in $X$:
$$ 
d_X(\alpha_X, \beta_X) \geq \frac{|\calP_\induct|}{2(3 \cdot
\Upper)^{\xi(X) - 1} + 1} - 1.
$$
\end{proposition}


\begin{proof}
Suppose, for a contradiction, that the conclusion fails.  Let $g_X$ be
a geodesic in $\calC(X)$ connecting $\alpha_X$ to $\beta_X$.  Then, as
in the proof of \reflem{PigeonForHoles}, there is a vertex $\omega$ of
$g_X$ and a component $W \subset X \setminus \omega$ where at least
$(3 \cdot \Upper)^{\xi(X) - 1}$ of the inductive intervals in $I_X$
have associated surfaces, $Y_i$, contained in $W$.

Since $\xi(X) - 1 \geq \xi(W)$ we may apply \reflem{PigeonForHoles}
inside of $W$.  So we find a surface $Z \subseteq W \subsetneq X$ so
that
\begin{itemize}
\item
$Z$ contains at least $\Upper$ of the $Y_i$,
\item
$d_Z(\alpha_X, \beta_X) \geq \Upper$, and
\item
there are at least  
$\Upper - 4(\Reverse + \Access + 2 \Access + 2)$
 of the $Y_i$ where $J_{Y_i} \subsetneq J_Z$.
\end{itemize}
Since $Y_i \subsetneq Z$ and $Y_i$ is a hole, $Z$ is also a hole.
Since $\Upper > \Lower(X)$ it follows that $Z \in B_X$.  Let $\calY =
\{ Y_i \}$ be the set of $Y_i$ satisfying the third bullet.  Let $Y^1
\in \calY$ and $\eta_1 \in g_Z$ satisfy  $\bdy Y^1 \cap \eta_1 =
\emptyset$ and $\eta_1$ is the first such.  Choose $Y^2$ and $\eta_2$
similarly, so that $\eta_2$ is the last such. By \reflem{PigeonForHoles}
\begin{equation}
\label{Eqn:EtaBigger}
d_Z(\eta_1,\eta_2) \geq L_2 - 4(\Reverse + \Access + 2\Access + 2).
\end{equation}
Let $p = \min I_{Y^1}$ and $q = \max I_{Y^2}$.  Note that $[p, q]
\subset J_Z$.  (If $Z$ is paired with $Z'$ then $[p, q] \subset J_Z
\cap J_{Z'}$.)  Again by (1) of \refax{Access}, and \reflem{Hempel},
\[
d_Z(\mu_{r(p)}, \bdy Y^1) < \Access.
\]
It follows that
\[
d_Z(\mu_{r(p)}, \eta_1) \leq \Access + 1
\]
and the same bound applies to $d_Z(\mu_{r(q)}, \eta_2)$.  Combined
with (\ref{Eqn:EtaBigger}) we find that
\[
d_Z(\mu_{r(p)}, \mu_{r(q)}) \geq
\Upper - 4\Reverse - 4\Access - 10\Access - 10.
\] 
By the reverse triangle inequality (\reflem{Reverse}), for
any $p' \leq p, q \leq q'$,
\[
d_Z(\mu_{r(p')}, \mu_{r(q')}) \geq 
  \Upper - 6\Reverse - 4\Access - 10 \Access - 10.
\]
Finally by \refax{Combin} and the above inequality we have  
\[
d_Z(\gamma_{p'}, \gamma_{q'}) \geq
\Upper - 6\Reverse - 4\Access - 10\Access - 10 - 2\Combin.
\] 
By (\ref{Eqn:UpperBigger}) the right-hand side is greater than
$\Lower(X) + 2\Paired$ so we deduce that $Z \in B_X(p', q')$, for any
such $p', q'$.  (When $Z$ is paired deduce also that $Z' \in B_X(p',
q')$.)


Let $I_V$ be the first inductive interval chosen by the procedure with
the property that $I_V \cap [p, q] \neq \emptyset$.  Note that, since
$I_{Y^1}$ and $I_{Y^2}$ will also be chosen, $I_V \subset [p, q]$.
Let $p', q'$ be the indices so that $V$ is chosen from $B_X(p', q')$.
Thus $p' \leq p$ and $q \leq q'$.  However, since $I_V \subset [p, q]
\subset J_Z$, \reflem{HolesContained} implies that $V$ is strictly
nested in $Z$.  (When pairing occurs we may find instead that $V
\subset Z'$ or $V' \subset Z$.)  Thus $\xi(Z) > \xi(V)$ and we find
that $Z$ would be chosen from $B_X(p', q')$, instead of $V$.  This is
a contradiction.
\end{proof}

\subsection{Electric partition}

The goal of this subsection is to prove:

\begin{proposition}
\label{Prop:Electric}
There is a constant $A$ depending only on $\xi(X)$, so that: if $[i,j]
\subset [0, K]$ is a electric interval then
\[
d_\calG(\gamma_i, \gamma_j) \quasileq d_X(\gamma_i, \gamma_j).
\] 
\end{proposition}


We begin by building a partition of $[i,j]$ into straight and shortcut
intervals.  Define
$$ 
C_X = \{ Y \subsetneq X \st \mbox{$Y$ is a non-hole and }
d_Y(\mu_{r(i)}, \mu_{r(j)}) \geq \Lower(X) \}.
$$ 
We also define, for all $[p, q] \subset [i, j]$
$$
C_X(p,q) = \{ Y \in C_X \st J_Y \cap [r(p), r(q)] \neq \emptyset \}.
$$

Our recursion starts with the partition of one part, $\calP(i, j) =
\{[i, j]\}$.  Recursively $\calP(i, j)$ is a partition of $[i, j]$
into shortcut, straight, or undetermined intervals.  Suppose that $[p,
q] \in \calP(i, j)$ is undetermined.

\begin{proofclaim}
If $C_X(p, q)$ is empty then $[p, q]$ is straight.
\end{proofclaim}

\begin{proof}
We show the contrapositive.  Suppose that $Y$ is a non-hole with
$d_Y(\mu_{r(p)}, \mu_{r(q)}) \geq \Upper$.  Since $\Upper > \Access$,
\refax{Access} implies that $J_Y \cap [r(p), r(q)]$ is non-empty.
Also, the reverse triangle inequality (\reflem{Reverse}) gives:
$$
d_Y(\mu_{r(p)}, \mu_{r(q)}) < d_Y(\mu_{r(i)}, \mu_{r(j)}) + 2\Reverse.
$$ 
Since $\Upper > \Lower(X) + 2\Reverse$, we find that $Y \in C_X$.  It
follows that $Y \in C_X(p,q)$.
\end{proof}

So when $C_X(p, q)$ is empty the interval $[p, q]$ is determined to be
straight.  Proceed onto the next undetermined element of $\calP(i,
j)$.  Now suppose that $C_X(p, q)$ is non-empty.  Then we choose any
$Y \in C_X(p,q)$ so that $Y$ has maximal $\xi(Y)$ amongst the elements
of $C_X(p,q)$.  Notice that by the accessibility requirement that $J_Y
\cap [r(p), r(q)]$ is non-empty. 

There are two cases.  If $J_Y \cap r([p, q])$ is empty then let $p'
\in [p, q]$ be the largest integer so that $r(p') < \min J_Y$.  Note
that $p'$ is well-defined.  Now divide the interval $[p, q]$ into the
two undetermined intervals $[p, p']$, $[p' + 1, q]$.  In this
situation we say $Y$ is associated to a {\em shortcut of length one}
and we add the element $[p' + \frac{1}{2}]$ to $\calP(i, j)$.

Next suppose that $J_Y \cap r([p, q])$ is non-empty.  Let $p', q' \in
[p,q]$ be the first and last indices, respectively, so that $r(p'),
r(q') \in J_Y$.  (Note that it is possible to have $p' = q'$.)
Partition $[p, q] = [p, p'-1] \cup [p', q'] \cup [q'+1, q]$.  The
first and third parts are undetermined; either may be empty.  This
completes the recursive construction of the partition.

Define
$$ 
\calP_\short = \{ I \in \calP(i, j) \st \mbox{$I$ is a shortcut}\}
$$
and
$$
\calP_\str = \{ I \in \calP(i, j) \st \mbox{$I$ is straight}\}.
$$

\begin{proposition}
\label{Prop:NumberOfShortcuts}
With $\calP(i, j)$ as defined above,
$$ 
d_X(\gamma_i, \gamma_j) \geq \frac{|\calP_\short|}{2(3 \cdot
\Upper)^{\xi(X) - 1} + 1} - 1.
$$
\end{proposition}

\begin{proof}
The proof is identical to that of \refprop{NumberOfInductives} with
the caveat that in \reflem{PigeonForHoles} we must use the markings
$\mu_{r(i)}$ and $\mu_{r(j)}$ instead of the endpoints $\gamma_i$ and
$\gamma_j$.
\end{proof}


Now we ``electrify'' every shortcut interval using \refthm{UpperBound}
recursively.

\begin{lemma}
\label{Lem:Shortcut}
There is a constant $\Shortcut = \Shortcut(X, \calG)$, so that for
every shortcut interval $[p, q]$ we have $d_\calG(\gamma_p, \gamma_q)
< \Shortcut$.
\end{lemma}

\begin{proof}
As $[p,q]$ is a shortcut we are given a non-hole $Z \subset X$ so that
$r([p,q]) \subset J_Z$.  Let $Y = X \setminus Z$.  Thus
\refax{Replace} gives vertices $\gamma_p', \gamma_q'$ of $\calG$ lying
in $Y$ or in $Z$, so that $d_\calG(\gamma_p, \gamma_p'),
d_\calG(\gamma_q, \gamma_q') \leq \Replace$.

If one of $\gamma_p', \gamma_q'$ lies in $Y$ while the other lies in
$Z$ then
\[
d_\calG(\gamma_p, \gamma_q) < 2\Replace + 1.
\]
If both lie in $Z$ then, as $Z$ is a non-hole, there is a vertex
$\delta \in \calG(S)$ disjoint from both of $\gamma_p'$ and
$\gamma_q'$ and we have 
\[
d_\calG(\gamma_p, \gamma_q) < 2\Replace + 2.
\]
If both lie in $Y$ then there are two cases.  If $Y$ is not a hole for
$\calG(S)$ then we are done as in the previous case.  If $Y$ is a hole
then by the definition of shortcut interval,
\reflem{LipschitzToHoles}, and the triangle inequality we have
\[
d_W(\gamma_p', \gamma_q') < 6 + 6\Replace + \Upper
\]
for all holes $W \subset Y$.  Notice that $Y$ is strictly contained in
$X$.  Thus we may inductively apply \refthm{UpperBound} with $c = 6 +
6\Replace + \Upper$.  We deduce that all terms on the right-hand side
of the distance estimate vanish and thus $d_\calG(\gamma_p',
\gamma_q')$ is bounded by a constant depending only on $X$ and
$\calG$.  The same then holds for $d_\calG(\gamma_p, \gamma_q)$ and we
are done.
\end{proof}

We are now equipped to give:

\begin{proof}[Proof of \refprop{Electric}]
Suppose that $\calP(i, j)$ is the given partition of the electric
interval $[i, j]$ into straight and shortcut subintervals.  As a bit
of notation, if $[p, q] = I \in \calP(i, j)$, we take $d_\calG(I) =
d_\calG(\gamma_p, \gamma_q)$ and $d_X(I) = d_X(\gamma_p, \gamma_q)$.
Applying \refax{Combin} we have
\begin{align}
\label{Eqn:StraightUpperBound}
d_\calG(\gamma_i, \gamma_j) & \leq \sum_{I \in \calP_\str} d_\calG(I)
 + \sum_{I \in \calP_\short} d_\calG(I) + \Combin|\calP(i, j)|
\end{align}
The last term arises from connecting left endpoints of intervals with
right endpoints.  We must bound the three terms on the right.

We begin with the third; recall that $|\calP(i, j)| = |\calP_\short| +
|\calP_\str|$, that $|\calP_\str| \leq |\calP_\short| + 1$, 
and that $|\calP_\short| \quasileq d_X(\gamma_i, \gamma_j)$.  The
second inequality follows from the construction of the partition while
the last is implied by \refprop{NumberOfShortcuts}.  Thus the third
term of \refeqn{StraightUpperBound} is quasi-bounded above by
$d_X(\gamma_i, \gamma_j)$.  

By \reflem{Shortcut}, the second term of \refeqn{StraightUpperBound}
at most $\Shortcut|\calP_\short|$.  Finally, by \refax{Straight}, for
all $I \in \calP_\str$ we have
$$ 
d_\calG(I) \quasileq d_X(I),
$$ 
Also, it follows from the reverse triangle inequality
(\reflem{Reverse}) that
\[ 
\sum_{I \in \calP_\str} d_X(I) \leq d_X(\gamma_i, \gamma_j) +
(2\Reverse + 2\Combin)|\calP_\str| + 2\Combin.
\] 
We deduce that $\sum_{I \in \calP_\str} d_\calG(I)$ is also
quasi-bounded above by $d_X(\gamma_i, \gamma_j)$.  Thus for a somewhat
larger value of $A$ we find
\[
d_\calG(\gamma_i, \gamma_j) \quasileq d_X(\gamma_i, \gamma_j).
\]
This completes the proof.
\end{proof}

\subsection{The upper bound}

We will need:

\begin{proposition}
\label{Prop:Convert}
For any $c > 0$ there is a constant $A$ with the following property.
Suppose that $[i, j] = I_Y$ is an inductive interval in $\calP_X$.
Then we have:
\[
d_\calG(\gamma_i, \gamma_j) \quasileq \sum_Z [d_Z(\gamma_i,
\gamma_j)]_{c}
\]
where $Z$ ranges over all holes for $\calG$ contained in $X$.
\end{proposition}

\begin{proof}
\refax{Replace} gives vertices $\gamma'_i$, $\gamma'_j \in \calG$,
contained in $Y$, so that $d_\calG(\gamma_i, \gamma'_i) \leq \Replace$
and the same holds for $j$.  Since projection to holes is coarsely
Lipschitz (\reflem{LipschitzToHoles}) for any hole $Z$ we have
$d_Z(\gamma_i, \gamma'_i) \leq 3 + 3\Access$.

Fix any $c > 0$.  Now, since
\begin{align*}
d_\calG(\gamma_i, \gamma_j) 
 & \leq d_\calG(\gamma'_i, \gamma'_j) + 2\Access
\end{align*}
to find the required constant $A$ it suffices to bound
$d_\calG(\gamma'_i, \gamma'_j)$.  Let $c' = c + 6\Access + 6$. Since
$Y \subsetneq X$, induction gives us a constant $A$ so that
\begin{align*}
d_\calG(\gamma'_i, \gamma'_j) 
  & \quasileq \sum_Z [d_Z(\gamma'_i, \gamma'_j)]_{c'}  \\
  & \leq \sum_Z [d_Z(\gamma_i, \gamma_j) + 6\Access + 6]_{c'} \\
  & < (6\Access + 6)N + \sum_Z [d_Z(\gamma_i, \gamma_j)]_c
\end{align*}
where $N$ is the number of non-zero terms in the final sum.  Also, the
sum ranges over sub-holes of $Y$.  We may take $A$ somewhat larger to
deal with the term $(6\Access + 6)N$ and include all holes $Z \subset
X$ to find
\begin{align*}
d_\calG(\gamma_i, \gamma_j) 
  & \quasileq \sum_Z [d_Z(\gamma_i, \gamma_j)]_c
\end{align*}
where the sum is over all holes $Z \subset X$.
\end{proof}

\subsection{Finishing the proof}

Now we may finish the proof of \refthm{UpperBound}.  Fix any constant
$c \geq 0$.  Suppose that $X$, $\alpha_X$, $\beta_X$ are given as
above.  Suppose that $\Gamma = \{ \gamma_i \}_{i = 0}^K$ is the given
combinatorial path and $\calP_X$ is the partition of $[0, K]$ into
inductive and electric intervals.  So we have:
\begin{align}
\label{Eqn:UpperBound}
d_\calG(\alpha_X, \beta_X) 
& \leq \sum_{I \in \calP_\induct} d_\calG(I) + \sum_{I \in
 \calP_\elect} d_\calG(I) + \Combin|\calP_X|
\end{align}
Again, the last term arises from adjacent right and left endpoints of
different intervals.  

We must bound the terms on the right-hand side; begin by noticing that
$|\calP_X| = |\calP_\induct| + |\calP_\elect|$, $|\calP_\elect| \leq
|\calP_\induct| + 1$ and $|\calP_\induct| \quasileq d_X(\alpha_X,
\beta_X)$.  The second inequality follows from the way the partition
is constructed and the last follows from \refprop{NumberOfInductives}.
Thus the third term of \refeqn{UpperBound} is quasi-bounded above by
$d_X(\alpha_X, \beta_X)$.

Next consider the second term of \refeqn{UpperBound}:
\begin{align*}
\sum_{I \in \calP_\elect} d_\calG(I) 
  & \quasileq \sum_{I \in \calP_\elect} d_X(I) \\
  & \leq d_X(\alpha_X, \beta_X) + (2\Reverse + 2\Combin)|\calP_\elect| + 2\Combin 
\end{align*}
with the first inequality following from \refprop{Electric} and the
second from the reverse triangle inequality (\reflem{Reverse}). 

Finally we bound the first term of \refeqn{UpperBound}.  Let $c' = c +
\Induct$.  Thus, 
\begin{align*}
\sum_{I \in \calP_\induct} d_\calG(I) 
  & \leq \sum_{I_Y \in \calP_\induct} \left( A'_Y \left( \sum_{Z
         \subsetneq Y} [d_Z(I_Y)]_{c'} \right) + A'_Y \right) \\ 
  & \leq A'' \left( \sum_{I \in \calP_\induct} \sum_{Z \subsetneq X} [d_Z(I)]_{c'} \right) 
         + A'' \cdot |\calP_\induct| \\
  & \leq A'' \left( \sum_{Z \subsetneq X} \sum_{I \in \calP_\induct} [d_Z(I)]_{c'} \right) 
         + A'' \cdot |\calP_\induct|
\end{align*}
Here $A'_Y$ and the first inequality are given by \refprop{Convert}.
Also $A'' = \max \{ A'_Y \st Y \subsetneq X \}$.  In the last line,
each sum of the form $\sum_{I \in \calP_\induct} [d_Z(I)]_{c'}$ has at
most three terms, by \refrem{AtMostThreeInductive} and the fact that
$c' > \Induct$.  For the moment, fix a hole $Z$ and any three elements
$I, I', I'' \in \calP_\induct$.

By the reverse triangle inequality (\reflem{Reverse}) we find that
\[
d_Z(I) + d_Z(I') + d_Z(I'') < d_Z(\alpha_X, \beta_X) + 6\Reverse + 8\Combin
\]
which in turn is less than $d_Z(\alpha_X, \beta_X) + \Induct$.

It follows that 
\[ 
[d_Z(I)]_{c'} + [d_Z(I')]_{c'} + [d_Z(I'')]_{c'} < [d_Z(\alpha_X,
  \beta_X)]_c + \Induct.
\]
Thus, 
\begin{align*}
\sum_{Z \subsetneq X} \sum_{I \in \calP_\induct} [d_Z(I)]_{c'}
& \leq \Induct \cdot N + \sum_{Z \subsetneq X} [d_Z(\alpha_X, \beta_X)]_c 
\end{align*}
where $N$ is the number of non-zero terms in the final sum.  Also, the
sum ranges over all holes $Z \subsetneq X$.

Combining the above inequalities, and increasing $A$ once again,
implies that
\[
d_\calG(\alpha_X, \beta_X) \quasileq \sum_Z [d_Z(\alpha_X, \beta_X)]_c
\]
where the sum ranges over all holes $Z \subseteq X$.  This completes
the proof of \refthm{UpperBound}. \qed

\section{Background on \Teich space}
\label{Sec:BackgroundTeich}

Our goal in Sections~\ref{Sec:PathsNonorientable}, \ref{Sec:PathsArc}
and~\ref{Sec:PathsDisk} will be to verify the axioms stated in
\refsec{Axioms} for the complex of curves of a non-orientable
surface, for the arc complex, and for the disk complex.  Here we give
the necessary background on \Teich space.

Fix now a surface $S = S_{g,n}$ of genus $g$ with $n$ punctures.  Two
conformal structures on $S$ are equivalent, written $\Sigma \sim
\Sigma'$, if there is a conformal map $f \from \Sigma \to \Sigma'$
which is isotopic to the identity.  Let $\calT = \calT(S)$ be the {\em
\Teich space} of $S$; the set of equivalence classes of conformal
structures $\Sigma$ on $S$.

Define the \Teich metric by,
\[
d_\calT(\Sigma,\Sigma') = 
         \inf_f \left\{ \frac{1}{2} \log K(f) \right\}
\] 
where the infimum ranges over all quasiconformal maps $f \from \Sigma
\to \Sigma'$ isotopic to the identity and where $K(f)$ is the maximal
dilatation of $f$.  Recall that the infimum is realized by a \Teich
map that, in turn, may be defined in terms of a quadratic
differential.

\subsection{Quadratic differentials}

\begin{definition}
A {\em quadratic differential} $q(z)\,dz^2$ on $\Sigma$ is an
assignment of a holomorphic function to each coordinate chart that is
a disk and of a meromorphic function to each chart that is a punctured
disk.  If $z$ and $\zeta$ are overlapping charts then we require
\[
q_z(z) = q_\zeta(\zeta) \left(\frac{d\zeta}{dz}\right)^2
\]
in the intersection of the charts.  The meromorphic function $q_z(z)$
has at most a simple pole at the puncture $z = 0$.
\end{definition}

At any point away from the zeroes and poles of $q$ there is a natural
coordinate $z = x + iy$ with the property that $q_z \equiv 1$.  In
this natural coordinate the foliation by lines $y = c$ is called the
{\em horizontal foliation}.  The foliation by lines $x = c$ is called
the {\em vertical foliation}.

Now fix a quadratic differential $q$ on $\Sigma = \Sigma_0$.  Let $x,
y$ be natural coordinates for $q$.  For every $t \in \RR$ we obtain a
new quadratic differential $q_t$ with coordinates
\[
x_t = e^{t} x, \qquad y_t = e^{-t} y. 
\] 
Also, $q_t$ determines a conformal structure $\Sigma_t$ on $S$.  The
map $t \mapsto \Sigma_t$ is the \Teich geodesic determined by $\Sigma$
and $q$. 

\subsection{Marking coming from a \Teich geodesic}
\label{Sec:MarkingFromTeich}

Suppose that $\Sigma$ is a Riemann surface structure on $S$ and
$\sigma$ is the uniformizing hyperbolic metric in the conformal class
of $\Sigma$.  In a slight abuse of terminology, we call the collection
of shortest simple non-peripheral closed geodesics the {\em systoles}
of $\sigma$.  Fix a constant $\epsilon$ smaller than the Margulis
constant.  The $\epsilon$--thick part of \Teich space consists of
those Riemann surfaces such that the hyperbolic systole has length at
least $\epsilon$.

We define $P = P(\sigma)$, a {\em Bers pants decomposition} of $S$, as
follows: pick $\alpha_1$, any systole for $\sigma$.  Define $\alpha_i$
to be any systole of $\sigma$ restricted to $S \setminus (\alpha_1
\cup \ldots \cup \alpha_{i - 1})$.  Continue in this fashion until $P$
is a pants decomposition.  Note that any curve with length less than
the Margulis constant will necessarily be an element of $P$.

Suppose that $\Sigma, \Sigma' \in \calT(S)$.  Suppose that $P, P'$ are
Bers pants decompositions with respect to $\Sigma$ and $\Sigma'$.
Suppose also that $d_\calT(\Sigma, \Sigma') \leq 1$. Then the curves
in $P$ have uniformly bounded lengths in $\Sigma'$ and conversely.  By
the Collar Lemma, the intersection $\iota(P, P')$ is bounded, solely
in terms of $\xi(S)$.

Suppose now that $\{ \Sigma_t \st t \in [-M, M] \}$ is the \Teich
geodesic defined by the quadratic differentials $q_t$.  Let $\sigma_t$
be the hyperbolic metric uniformizing $\Sigma_t$.  Let $P_t =
P(\sigma_t)$ be a Bers pants decomposition.

We now find transversals in order to complete $P_t$ to a {\em Bers
marking} $\nu_t$.  Suppose that $P_t = \{ \alpha_i \}$.  For each $i$,
let $A^i$ be the annular cover of $S$ corresponding to $\alpha_i$.
Note that $q_t$ lifts to a singular Euclidean metric $q^i_t$ on $A^i$.
Let $\alpha^i$ be a geodesic representative of the core curve of $A^i$
with respect to the metric $q_t^i$.  Choose $\gamma_i \in \calC(A^i)$
to be any geodesic arc, also with respect to $q^i_t$, that is
perpendicular to $\alpha^i$.  Let $\beta_i$ be any curve in $S
\setminus (\{ \alpha_j \}_{j \neq i})$ which meets $\alpha_i$
minimally and so that $d_{A_i}(\beta_i, \gamma_i) \leq 3$.  (See the
discussion after the proof of Lemma~2.4 in~\cite{MasurMinsky00}.)
Doing this for each $i$ gives a complete clean marking $\nu_t = \{
\alpha_i \} \cup \{ \beta_i \}$.

We now have:

\begin{lemma}
\cite[Remark 6.2 and Equation (3)]{Rafi10} 
There is a constant $\Bound = \Bound(S)$ with the following property.
For any \Teich geodesic and for any time $t$, there is a constant
$\delta > 0$ so that if $|t - s| \leq \delta$ then 
\[
\iota(\nu_t, \nu_s) < \Bound.
\]
\end{lemma}

Suppose that $\Sigma_t$ and $\Sigma_s$ are surfaces in the
$\epsilon$--thick part of $\calT(S)$.  We take $\Bound$ sufficiently
large so that if $\iota(\nu_t, \nu_s) \geq \Bound$ then
$d_\calT(\Sigma_t, \Sigma_s) \geq 1$.

\subsection{The marking axiom}

We construct a sequence of markings $\mu_n$, for $n \in [0, N] \subset
\NN$, as follows.  Take $\mu_0 = \nu_{-M}$.  Now suppose that $\mu_n =
\nu_t$ is defined.  Let $s > t$ be the first time that there is a
marking with $\iota(\nu_t, \nu_s) \geq \Bound$, if such a time exists.
If so, let $\mu_{n+1} = \nu_s$.  If no such time exists take $N = n$
and we are done.

We now show that $\mu_n = \nu_t$ and $\mu_{n+1} = \nu_s$ have bounded
intersection.  By the above lemma there is a marking $\nu_r$ with
$t \leq r < s$ and 
\[
\iota(\nu_r,\nu_s) \leq \Bound.
\]
By construction
\[
\iota(\nu_t, \nu_r) < \Bound.
\]
Since intersection number bounds distance in the marking complex we
find that by the triangle inequality, $\nu_t$ and $\nu_s$ are bounded
distance in the marking complex.  Conversely, since distance bounds
intersection in the marking complex we find that
$\iota(\mu_n,\mu_{n+1})$ is bounded.  It follows that $d_Y(\mu_n,
\mu_{n+1})$ is uniformly bounded, independent of $Y \subset S$ and of
$n \in [0, N]$.

It now follows from Theorem 6.1 of \cite{Rafi10} that, for any
subsurface $Y \subset S$, the sequence $\{ \pi_Y(\mu_n) \} \subset
\calC(Y)$ is an unparameterized quasi-geodesic.  Thus the marking path
$\{ \mu_n \}$ satisfies the second requirement of \refax{Marking}.
The first requirement is trivial as every $\mu_n$ fills $S$.

\subsection{The accessibility axiom}

We now turn to \refax{Access}. Since $\mu_n$ fills $S$ for every $n$,
the first requirement is a triviality. 

In Section 5 of \cite{Rafi10} Rafi defines, for every subsurface $Y
\subset S$, an {\em interval of isolation} $I_Y$ inside of the
parameterizing interval of the \Teich geodesic.  Note that $I_Y$ is
defined purely in terms of the geometry of the given quadratic
differentials.  Further, for all $t \in I_Y$ and for all components
$\alpha \subset \bdy Y$ the hyperbolic length of $\alpha$ in
$\Sigma_t$ is less than the Margulis constant.  Furthermore, by
Theorem~5.3~\cite{Rafi10}, there is a constant $\AccessTemp$ so that if
$[s,t] \cap I_Y = \emptyset$ then
\[
d_Y(\nu_s, \nu_t) \leq \AccessTemp.
\]
So define $J_Y \subset [0, N]$ to be the subinterval of the marking
path where the time corresponding to $\mu_n$ lies in $I_Y$.  The third
requirement follows.  Finally, if $m \in J_Y$ then $\bdy Y$ is
contained in $\base(\mu_m)$ and thus $\iota(\bdy Y, \mu_m) \leq 2
\cdot |\bdy Y|$.

\subsection{The distance estimate in \Teich space}

We end this section by quoting another result of Rafi:

\begin{theorem}\cite[Theorem~2.4]{Rafi10}
\label{Thm:TeichDistanceEstimate}
Fix a surface $S$ and a constant $\epsilon > 0$.  There is a constant
$\CutOff = \CutOff(S, \epsilon)$ so that for any $c > \CutOff$ there
is a constant $A$ with the following property.  Suppose that
$\Sigma$ and $\Sigma'$ lie in the $\epsilon$--thick part of
$\calT(S)$.  Then 
\[
d_\calT(\Sigma, \Sigma') \quasieq 
   \sum_X [d_X(\mu, \mu')]_c + \sum_\alpha [\log d_\alpha(\mu, \mu')]_c 
\] 
where $\mu$ and $\mu'$ are Bers markings on $\Sigma$ and $\Sigma'$, $Y
\subset S$ ranges over non-annular surfaces and $\alpha$ ranges over
vertices of $\calC(S)$. \qed
\end{theorem}

\section{Paths for the non-orientable surface}
\label{Sec:PathsNonorientable}

Fix $F$ a compact, connected, and non-orientable surface.  Let $S$ be
the orientation double cover with covering map $\rho_F \from S \to F$.
Let $\tau \from S \to S$ be the associated involution.  Note that
$\calC(F) = \calC^\tau(S)$.  Let $\calC^\tau(S) \to \calC(S)$ be the
relation sending a symmetric multicurve to its components.

Our goal for this section is to prove \reflem{SymmetricSurfaces}, the
classification of holes for $\calC(F)$.  As remarked above,
\reflem{InvariantQI} and \refcor{NonorientableCurveComplexHyperbolic}
follow, proving the hyperbolicity of $\calC(F)$.  

\subsection{The marking path}

We will use the extreme rigidity of \Teich geodesics to find
$\tau$--invariant marking paths.  We first show that $\tau$--invariant
Bers pants decompositions exist.

\begin{lemma}
\label{Lem:TauInvar}
Fix a $\tau$--invariant hyperbolic metric $\sigma$.  Then there is a
Bers pants decomposition $P = P(\sigma)$ which is $\tau$--invariant.
\end{lemma}

\begin{proof}
Let $P_0 = \emptyset$.  Suppose that $0 \leq k < \xi(S)$ curves have
been chosen to form $P_k$.  By induction we may assume that $P_k$ is
$\tau$--invariant.  Let $Y$ be a component of $S \setminus P_k$ with
$\xi(Y) \geq 1$.  Note that since $\tau$ is orientation reversing,
$\tau$ does not fix any boundary component of $Y$.

Pick any systole $\alpha$ for $Y$.  

\begin{proofclaim}
Either $\tau(\alpha) = \alpha$ or $\alpha \cap \tau(\alpha) =
\emptyset$. 
\end{proofclaim}

\begin{proof}
Suppose not and take $p \in \alpha \cap \tau(\alpha)$.  Then $\tau(p)
\in \alpha \cap \tau(\alpha)$ as well, and, since $\tau$ has no fixed
points, $p \neq \tau(p)$.  The points $p$ and $\tau(p)$ divide
$\alpha$ into segments $\beta$ and $\gamma$.  Since $\tau$ is an
isometry, we have
\[
\ell_\sigma(\tau(\alpha)) = \ell_\sigma(\alpha)
\quad \mbox{and} \quad
\ell_\sigma(\tau(\beta)) = \ell_\sigma(\beta).
\] 
Now concatenate to obtain (possibly immersed) loops
\[
\beta' = \beta*\tau(\beta)
\quad \mbox{and} \quad
\gamma' = \gamma*\tau(\gamma).
\]

If $\beta'$ is null-homotopic then $\alpha \cup
\tau(\alpha)$ cuts a monogon or a bigon out of $S$, contradicting our
assumption that $\alpha$ was a geodesic.  Suppose, by way of
contradiction, that $\beta'$ is homotopic to some boundary curve $b
\subset \bdy Y$.  Since $\tau(\beta') = \beta'$, it follows that
$\tau(b)$ and $\beta'$ are also homotopic.  Thus $b$ and $\tau(b)$
cobound an annulus, implying that $Y$ is an annulus, a contradiction.
The same holds for $\gamma'$.

Let $\beta''$ and $\gamma''$ be the geodesic representatives of
$\beta'$ and $\gamma'$.  Since $\beta$ and $\tau(\beta)$ meet
transversely, $\beta''$ has length in $\sigma$ strictly smaller than
$2\ell_\sigma(\beta)$.  Similarly the length of $\gamma''$ is strictly
smaller than $2\ell_\sigma(\gamma)$.  Suppose that $\beta''$ is
shorter then $\gamma''$.  It follows that $\beta''$ strictly shorter
than $\alpha$.  If $\beta''$ is embedded then this contradicts the
assumption that $\alpha$ was shortest.  If $\beta''$ is not embedded
then there is an embedded curve $\beta'''$ inside of a regular
neighborhood of $\beta''$ which is again essential, non-peripheral,
and has geodesic representative shorter than $\beta''$.  This is our
final contradiction and the claim is proved.
\end{proof}

Thus, if $\tau(\alpha) = \alpha$ we let $P_{k+1} = P_k \cup \{ \alpha
\}$ and we are done.  If $\tau(\alpha) \neq \alpha$ then by the above
claim $\tau(\alpha) \cap \alpha = \emptyset$.  In this case let
$P_{k+2} = P_k \cup \{ \alpha, \tau(\alpha) \}$ and \reflem{TauInvar}
is proved. 
\end{proof}

Transversals are chosen with respect to a quadratic differential
metric.  Suppose that $\alpha, \beta \in \calC^\tau(S)$.  If $\alpha$
and $\beta$ do not fill $S$ then we may replace $S$ by the support of
their union.  Following Thurston~\cite{Thurston88} there exists a
square-tiled quadratic differential $q$ with squares associated to the
points of $\alpha \cap \beta$.  (See~\cite{Bowditch06} for analysis of
how the square-tiled surface relates to paths in the curve complex.)
Let $q_t$ be image of $q$ under the \Teich geodesic flow.  We have:

\begin{lemma}
\label{Lem:TauIsom}
$\tau^*q_t = q_t$.
\end{lemma}

\begin{proof}
Note that $\tau$ preserves $\alpha$ and also $\beta$.  Since $\tau$
permutes the points of $\alpha \cap \beta$ it permutes the rectangles
of the singular Euclidean metric $q_t$ while preserving their vertical
and horizontal foliations.  Thus $\tau$ is an isometry of the metric
and the conclusion follows. 
\end{proof}

We now choose the \Teich geodesic $\{ \Sigma_t \st t \in [-M, M] \}$
so that the hyperbolic length of $\alpha$ is less than the Margulis
constant in $\sigma_{-M}$ and the same holds for $\beta$ in
$\sigma_M$.  Also, $\alpha$ is the shortest curve in $\sigma_{-M}$ and
similarly for $\beta$ in $\sigma_M$

\begin{lemma}
Fix $t$.  There are transversals for $P_t$ which are close to being
quadratic perpendicular in $q_t$ and which are $\tau$--invariant.
\end{lemma}

\begin{proof}
Let $P = P_t$ and fix $\alpha \in P$.  Let $X = S \setminus (P
\setminus \alpha)$.  There are two cases: either $\tau(X) \cap X =
\emptyset$ or $\tau(X) = X$.  Suppose the former.  So we choose any
transversal $\beta \subset X$ close to being $q_t$--perpendicular and
take $\tau(\beta)$ to be the transversal to $\tau(\alpha)$.

Suppose now that $\tau(X) = X$.  It follows that $X$ is a four-holed
sphere. The quotient $X/\tau$ is homeomorphic to a twice-holed
$\RRPP^2$.  Therefore there are only four essential non-peripheral
curves in $X/\tau$.  Two of these are cores of M\"obius bands and the
other two are their doubles.  The cores meet in a single
point. Perforce $\alpha$ is the double cover of one core and we take
$\beta$ the double cover of the other.

It remains only to show that $\beta$ is close to being
$q_t$--perpendicular.  Let $S^\alpha$ be the annular cover of $S$ and
lift $q_t$ to $S^\alpha$.  Let $\perp$ be the set of
$q_t^\alpha$--perpendiculars.  This is a $\tau$-invariant diameter one
subset of $\calC(S^\alpha)$.  If $d_\alpha(\perp, \beta)$ is large
then it follows that $d_\alpha(\perp, \tau(\beta))$ is also large.
Also, $\tau(\beta)$ twists in the opposite direction from $\beta$.
Thus
\[
d_\alpha(\beta, \tau(\beta)) - 2d_\alpha(\perp, \beta) = O(1)
\]
and so $d_\alpha(\beta, \tau(\beta))$ is large, contradicting the fact
that $\beta$ is $\tau$--invariant.
\end{proof}


Thus $\tau$--invariant markings exist; these have bounded intersection
with the markings constructed in \refsec{BackgroundTeich}.  It follows
that the resulting marking path satisfies the marking path and
accessibility requirements, Axioms~\ref{Ax:Marking}
and~\ref{Ax:Access}.

\subsection{The combinatorial path}

As in \refsec{BackgroundTeich} break the interval $[-M, M]$ into short
subintervals and produce a sequence of $\tau$-invariant markings $\{
\mu_n \}_{n = 0}^N$.  To choose the combinatorial path, pick $\gamma_n
\in \base(\mu_n)$ so that $\gamma_n$ is a $\tau$--invariant curve or
pair of curves and so that $\gamma_n$ is shortest in $\base(\mu_n)$.

We now check the combinatorial path requirements given in
\refax{Combin}.  Note that $\gamma_0 = \alpha$, $\gamma_N = \beta$;
also the reindexing map is the identity.  
Since
\[
\iota(\gamma_n, \mu_{r(n)}) = \iota(\gamma_n, \mu_n) = 2
\] 
the first requirement is satisfied.  Since $\mu_n$ and $\mu_{n+1}$
have bounded intersection, the same holds for $\gamma_n$ and
$\gamma_{n+1}$.  Projection to $F$, surgery, and \reflem{Hempel} imply
that $d_{\calC^\tau}(\gamma_n, \gamma_{n+1})$ is uniformly bounded.
This verifies \refax{Combin}.

\subsection{The classification of holes}

We now finish the classification of large holes for $\calC^\tau(S)$.
Fix $\LargeProj > 3\Access + 2\Combin + 2\Reverse$.  Note that these
constants are available because we have verified the axioms that give
them.

\begin{lemma}
\label{Lem:SymmetricSurfaces}
Suppose that $\alpha, \beta \in \calC^\tau(S)$.  Suppose that $X
\subset S$ has $d_X(\alpha, \beta) > \LargeProj$.  Then $X$ is
symmetric.
\end{lemma}

\begin{proof}
Let $(\Sigma_t, q_t)$ be the \Teich geodesic defined above and let
$\sigma_t$ be the uniformizing hyperbolic metric.  Since $\LargeProj >
\Access + 2\Combin$ it follows from the accessibility requirement that
$J_X = [m, n]$ is non-empty.  Now for all $t$ in the interval of
isolation $I_X$
\[
\ell_{\sigma_t}(\delta) < \epsilon,
\] 
where $\delta$ is any component of $\bdy X$ and $\epsilon$ is the
Margulis constant.  Let $Y = \tau(X)$.  Since $\tau$ is an isometry
(\reflem{TauIsom}) and since the interval of isolation is metrically
defined we have $I_Y = I_X$ and thus $J_Y = J_X$.  Deduce that $\bdy
Y$ is also short in $\sigma_t$.  This implies that $\bdy X \cap \bdy Y
= \emptyset$.  If $X$ and $Y$ overlap then by (iii) of \reflem{Access}
we have
\[
d_X(\mu_m, \mu_n) < \Access
\]
and so by the triangle inequality,  two applications of
(2) of \refax{Access},  we have
\[
d_X(\mu_0, \mu_N) < 3\Access.
\]
By the combinatorial axiom  it follows that 
\[
d_X(\alpha, \beta) < 3\Access  + 2\Combin
\]
a contradiction.  Deduce that either $X = Y$ or $X \cap Y = \emptyset$
as desired. 
\end{proof}

As noted in \refsec{HolesNonorientable} this shows that the only hole
for $\calC^\tau(S)$ is $S$ itself.  Thus all holes trivially
interfere, verifying \refax{Holes}. 

\subsection{The replacement axiom}

We now verify \refax{Replace} for subsurfaces $Y \subset S$ with
$d_Y(\alpha, \beta) \geq \LargeProj$.  (We may ignore all subsurfaces
with smaller projection by taking $\Lower(Y) > \LargeProj$.)

By \reflem{SymmetricSurfaces} the subsurface $Y$ is symmetric.  If $Y$
is a hole then $Y = S$ and the first requirement is vacuous.  Suppose
that $Y$ is not a hole.  Suppose that $\gamma_n$ is such that $n \in
J_Y$.  Thus $\gamma_n \in \base(\mu_n)$.  All components of $\bdy Y$
are also pants curves in $\mu_n$.  It follows that we may take any
symmetric curve in $\bdy Y$ to be $\gamma'$ and we are done.

\subsection{On straight intervals}

Lastly we verify \refax{Straight}.  Suppose that $[p, q]$ is a
straight interval.  We must show that $d_{\calC^\tau}(\gamma_p,
\gamma_q) \leq d_S(\gamma_p, \gamma_q)$.  Suppose that $\mu_p = \nu_s$
and $\mu_q = \nu_t$; that is, $s$ and $t$ are the times when $\mu_p,
\mu_q$ are short markings.  Thus $d_X(\mu_p, \mu_q) \leq \Upper$ for
every $X \subsetneq S$.  This implies that the \Teich geodesic, along
the straight interval, lies in the thick part of \Teich space.

Notice that $d_{\calC^\tau}(\gamma_p, \gamma_q) \leq \Combin|p - q|$,
since for all $i \in [p, q - 1]$, $d_{\calC^\tau}(\gamma_i,
\gamma_{i+1}) \leq \Combin$.  So it suffices to bound $|p - q|$.  By
our choice of $\Bound$ and because the \Teich geodesic lies in the
thick part we find that $|p - q| \leq d_\calT(\Sigma_s, \Sigma_t)$.
Rafi's distance estimate (\refthm{TeichDistanceEstimate}) gives:
\[
d_\calT(\Sigma_s, \Sigma_t) \quasieq d_S(\nu_s, \nu_t).
\] 
Since $\nu_s = \mu_p$, $\nu_t = \mu_q$, and since $\gamma_p \in
\base(\mu_p)$, $\gamma_q \in \base(\mu_q)$ deduce that
\[
d_S(\mu_p, \mu_q) \leq d_S(\gamma_p, \gamma_q) + 4.
\]
This verifies \refax{Straight}.  Thus the distance estimate holds for
$\calC^\tau(S) = \calC(F)$.  Since there is only one hole for
$\calC(F)$ we deduce that the map $\calC(F) \to \calC(S)$ is a
quasi-isometric embedding.  As a corollary we have:

\begin{theorem}
\label{Thm:NonOrientableCCHyperbolic}
The curve complex $\calC(F)$ is Gromov hyperbolic. \qed
\end{theorem}

\section{Paths for the arc complex}
\label{Sec:PathsArc}


Here we verify that our axioms hold for the arc complex $\calA(S,
\Delta)$.  It is worth pointing out that the axioms may be verified
using \Teich geodesics, train track splitting sequences, or
resolutions of hierarchies.  Here we use the former because it also
generalizes to the non-orientable case; this is discussed at the end
of this section.

First note that \refax{Holes} follows from
\reflem{ArcComplexHolesIntersect}.

\subsection{The marking path}

We are given a pair of arcs $\alpha, \beta \in \calA(X, \Delta)$.
Recall that $\sigma_S \from \calA(X) \to \calC(X)$ is the surgery map,
defined in \refdef{SurgeryRel}.  Let $\alpha' = \sigma_S(\alpha)$ and
define $\beta'$ similarly.  Note that $\alpha'$ cuts a pants off of
$S$.  As usual, we may assume that $\alpha'$ and $\beta'$ fill $X$.
If not we pass to the subsurface they do fill.  

As in the previous sections let $q$ be the quadratic differential
determined by $\alpha'$ and $\beta'$.  Exactly as above, fix a marking
path $\{ \mu_n \}_{n = 0}^N$.  This path satisfies the marking and
accessibility axioms (\ref{Ax:Marking}, \ref{Ax:Access}). 

\subsection{The combinatorial path}

Let $Y_n \subset X$ be any component of $X \setminus \base(\mu_n)$
meeting $\Delta$.  So $Y_n$ is a pair of pants.  Let $\gamma_n$ be any
essential arc in $Y_n$ with both endpoints in $\Delta$.  Since
$\alpha' \subset \base(\mu_0)$ and $\beta' \subset \base(\mu_N)$ we
may choose $\gamma_0 = \alpha$ and $\gamma_N = \beta$.

As in the previous section the reindexing map is the identity.  It
follows immediately that $\iota(\gamma_n, \mu_n) \leq 4$.  This bound,
the bound on $\iota(\mu_n, \mu_{n+1})$, and
\reflem{BoundedProjectionImpliesBoundedIntersection} imply that
$\iota(\gamma_n, \gamma_{n+1})$ is likewise bounded.  The usual
surgery argument shows that if two arcs have bounded intersection
then they have bounded distance.  This verifies \refax{Combin}. 

\subsection{The replacement and the straight axioms}

Suppose that $Y \subset X$ is a subsurface and $\gamma_n$ has $n \in
J_Y$.  Let $\mu_n = \nu_t$; that is $t$ is the time when $\mu_n$ is a
short marking.  Thus $\bdy Y \subset \base(\mu_n)$ and so $\gamma_n
\cap \bdy Y = \emptyset$.  So regardless of the hole-nature of $Y$ we
may take $\gamma' = \gamma_n$ and the axiom is verified. 

\refax{Straight} is verified exactly as in
\refsec{PathsNonorientable}.

\subsection{Non-orientable surfaces}

Suppose that $F$ is non-orientable and $\Delta_F$ is a collection of
boundary components.  Let $S$ be the orientation double cover and
$\tau \from S \to S$ the involution so that $S/\tau = F$. Let $\Delta$
be the preimage of $\Delta_F$.  Then $\calA^\tau(S, \Delta)$ is the
invariant arc complex.

Suppose that $\alpha_F, \beta_F$ are vertices in $\calA(F, \Delta')$.
Let $\alpha, \beta$ be their preimages.  As above, without loss of
generality, we may assume that $\sigma_F(\alpha_F)$ and
$\sigma_F(\beta_F)$ fill $F$.  Note that $\sigma_F(\alpha_F)$ cuts a
surface $X$ off of $F$.  The surface $X$ is either a pants or a
twice-holed $\RRPP^2$.  When $X$ is a pants we define $\alpha' \subset
S$ to be the preimage of $\sigma_F(\alpha_F)$.  When $X$ is a
twice-holed $\RRPP^2$ we take $\gamma_F$ to be a core of one of the
two M\"obius bands contained in $X$ and we define $\alpha'$ to be the
preimage of $\gamma_F \cup \sigma_F(\alpha_F)$.  We define $\beta'$
similarly.  Notice that $\alpha$ and $\alpha'$ meet in at most four
points.

We now use $\alpha'$ and $\beta'$ to build a $\tau$--invariant \Teich
geodesic.  The construction of the marking and combinatorial paths for
$\calA^\tau(S, \Delta)$ is unchanged.  Notice that we may choose
combinatorial vertices because $\base(\mu_n)$ is $\tau$--invariant.
There is a small annoyance: when $X$ is a twice-holed $\RRPP^2$ the
first vertex, $\gamma_0$, is disjoint from but not equal to $\alpha$.
Strictly speaking, the first and last vertices are $\gamma_0$ and
$\gamma_N$; our constants are stated in terms of their
subsurface projection distances.  However, since $\alpha \cap \gamma_0
= \emptyset$, and the same holds for $\beta$, $\gamma_N$, their
subsurface projection distances are all bounded.

\section{Background on train tracks}
\label{Sec:BackgroundTrainTracks}

Here we give the necessary definitions and theorems regarding train
tracks.  The standard reference is~\cite{PennerHarer92}.  See
also~\cite{Mosher03}.  We follow closely the discussion found
in~\cite{MasurEtAl10}.

\subsection{On tracks}

A {\em generic train track} $\tau \subset S$ is a smooth, embedded
trivalent graph.  As usual we call the vertices {\em switches} and the
edges {\em branches}.  At every switch the tangents of the three
branches agree.  Also, there are exactly two {\em incoming} branches
and one {\em outgoing} branch at each switch.  See
Figure~\ref{Fig:TrainTrackModel} for the local model of a switch.

\begin{figure}[htbp]
\labellist
\small\hair 2pt
\pinlabel {incoming} [bl] at 1 74
\pinlabel {incoming} [tl] at 1 1
\pinlabel {outgoing} [bl] at 145 38
\endlabellist
\[
\begin{array}{c}
\includegraphics[height = 2cm]{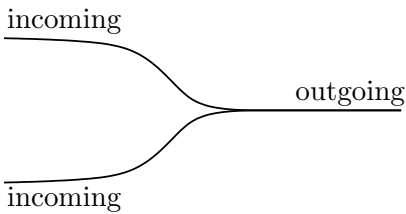}
\end{array}
\]
\caption{The local model of a train track.}
\label{Fig:TrainTrackModel}
\end{figure}

Let $\calB(\tau)$ be the set of branches.  A transverse measure on
$\tau$ is function $w \from \calB \to \RR_{\geq 0}$ satisfying the
switch conditions: at every switch the sum of the incoming measures
equals the outgoing measure.  Let $P(\tau)$ be the projectivization of
the cone of transverse measures.  Let $V(\tau)$ be the vertices of
$P(\tau)$.  As discussed in the references, each vertex measure gives
a simple closed curve carried by $\tau$.

For every track $\tau$ we refer to $V(\tau)$ as the marking
corresponding to $\tau$ (see \refsec{Markings}).  Note that there are
only finitely many tracks up to the action of the mapping class
group.  It follows that $\iota(V(\tau))$ is uniformly bounded,
depending only on the topological type of $S$.  

If $\tau$ and $\sigma$ are train tracks, and $Y \subset S$ is an
essential surface, then define 
\[
d_Y(\tau, \sigma) = d_Y(V(\tau), V(\sigma)).
\]
We also adopt the notation $\pi_Y(\tau) = \pi_Y(V(\tau))$.

A train track $\sigma$ is obtained from $\tau$ by {\em sliding} if
$\sigma$ and $\tau$ are related as in Figure~\ref{Fig:Slide}.  We say
that a train track $\sigma$ is obtained from $\tau$ by {\em splitting}
if $\sigma$ and $\tau$ are related as in Figure~\ref{Fig:Split}.

\begin{figure}[htbp]
\[
\begin{array}{cc}
\includegraphics[height = 1.5cm]{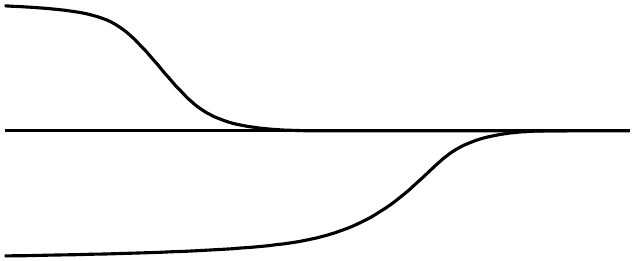} &
\includegraphics[height = 1.5cm]{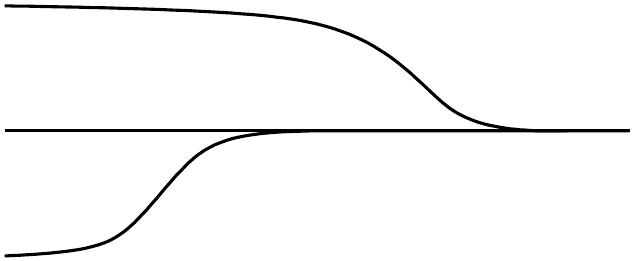}
\end{array}
\]
\caption{All slides take place in a small regular neighborhood of the
 affected branch.}
\label{Fig:Slide}
\end{figure}

\begin{figure}[htbp]
\[
\begin{array}{cc}
\includegraphics[height = 1.5cm]{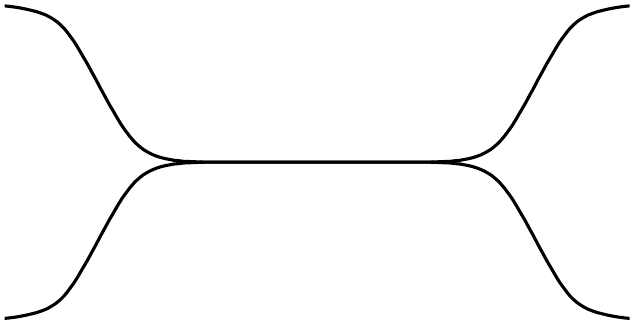} &
\includegraphics[height = 1.5cm]{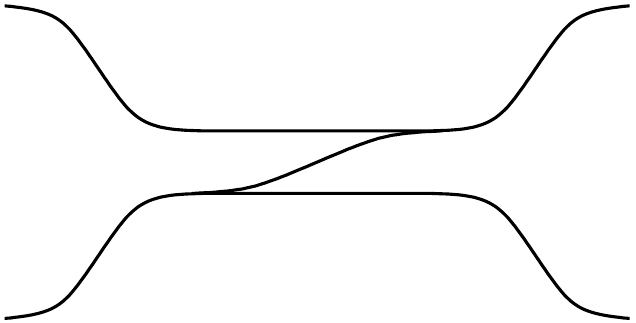} \\
\includegraphics[height = 1.5cm]{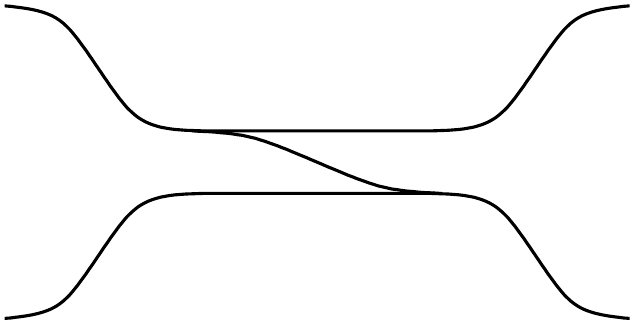} &
\includegraphics[height = 1.5cm]{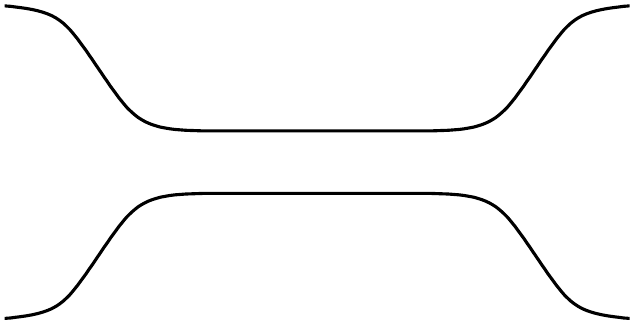} \\
\end{array}
\]
\caption{There are three kinds of splitting: right, left, and
central.}
\label{Fig:Split}
\end{figure}


Again, since the number of tracks is bounded (up to the action of the
mapping class group) if $\sigma$ is obtained from $\tau$ by either a
slide or a split we find that $\iota(V(\tau), V(\sigma))$ is uniformly
bounded. 

\subsection{The marking path}

We will use sequences of train tracks to define our marking path.

\begin{definition}
\label{Def:SplittingSequence}
A {\em sliding and splitting sequence} is a collection $\{ \tau_n
\}_{n = 0}^N$ of train tracks so that $\tau_{n+1}$ is obtained from
$\tau_n$ by a slide or a split.
\end{definition}

The sequence $\{ \tau_n \}$ gives a sequence of markings via the map
$\tau_n \mapsto V_n = V(\tau_n)$.  Note that the support of $V_{n+1}$
is contained within the support of $V_n$ because every vertex of
$\tau_{n+1}$ is carried by $\tau_n$.  Theorem 5.5
of~\cite{MasurEtAl10} verifies the remaining half of \refax{Marking}.

\begin{theorem}
\label{Thm:TrainTrackUnparamGeodesic}
Fix a surface $S$.  There is a constant $A$ with the following
property.  Suppose that $\{ \tau_n \}_{n=0}^N$ is a sliding and
splitting sequence in $S$ of birecurrent tracks.  Suppose that $Y
\subset S$ is an essential surface.  Then the map $n \mapsto
\pi_Y(\tau_n)$, as parameterized by splittings, is an
$A$--unparameterized quasi-geodesic. \qed
\end{theorem}

Note that, when $Y = S$, \refthm{TrainTrackUnparamGeodesic} is
essentially due to the first author and Minsky; see Theorem~1.3
of~\cite{MasurMinsky04}. 

In Section 5.2 of~\cite{MasurEtAl10}, for every sliding and splitting
sequence $\{ \tau_n \}_{n = 0}^N$ and for any essential subsurface $X
\subsetneq S$ an accessible interval $I_X \subset [0,N]$ is defined.
\refax{Access} is now verified by Theorem 5.3 of~\cite{MasurEtAl10}.

\subsection{Quasi-geodesics in the marking graph}

We will also need Theorem~6.1 from~\cite{MasurEtAl10}.
(See~\cite{Hamenstadt09} for closely related work.)

\begin{theorem}
\label{Thm:TrainTrackQuasi}
Fix a surface $S$.  There is a constant $A$ with the following
property.  Suppose that $\{ \tau_n \}_{n=0}^N$ is a sliding and
splitting sequence of birecurrent tracks, injective on slide
subsequences, where $V_N$ fills $S$.  Then $\{ V(\tau_n) \}$ is an
$A$--quasi-geodesic in the marking graph.  \qed
\end{theorem}

\section{Paths for the disk complex}
\label{Sec:PathsDisk}

Suppose that $V = V_g$ is a genus $g$ handlebody.  The goal of this
section is to verify the axioms of \refsec{Axioms} for the disk
complex $\calD(V)$ and so complete the proof of the distance estimate.

\begin{theorem}
\label{Thm:DiskComplexDistanceEstimate}
There is a constant $\CutOff = \CutOff(V)$ so that, for any $c \geq
\CutOff$ there is a constant $A$ with
\[
d_\calD(D, E) \quasieq \sum [d_X(D, E)]_c
\] 
independent of the choice of $D$ and $E$.  Here the sum ranges over
the set of holes $X \subset \bdy V$ for the disk complex.
\end{theorem}

\subsection{Holes}

The fact that all large holes interfere is recorded above as
\reflem{DiskComplexHolesInterfere}.  This verifies \refax{Holes}.

\subsection{The combinatorial path}

Suppose that $D, E \in \calD(V)$ are disks contained in a compressible
hole $X \subset S = \bdy V$.  As usual we may assume that $D$ and $E$
fill $X$.  Recall that $V(\tau)$ is the set of vertices for the track
$\tau \subset X$.  We now appeal to a result of the first author and
Minsky, found in~\cite{MasurMinsky04}.

\begin{theorem}
\label{Thm:DiskSurgerySequence}
There exists a surgery sequence of disks $\{ D_i \}_{i=0}^K$, a
sliding and splitting sequence of birecurrent tracks $\{ \tau_n
\}_{n=0}^N$, and a reindexing function $r \from [0, K] \to [0, N]$ so
that
\begin{itemize}
\item
$D_0 = D$, 
\item
$E \in V_N$, 
\item
$D_i \cap D_{i+1} = \emptyset$ for all $i$, and
\item 
$\iota(\bdy D_i, V_{r(i)})$ is uniformly bounded for all $i$.
\end{itemize}
\qed
\end{theorem}

\begin{remark}
For the details of the proof we refer to~\cite{MasurMinsky04}.  Note
that the double-wave curve replacements of that paper are not needed
here; as $X$ is a hole, no curve of $\bdy X$ compresses in $V$.  It
follows that consecutive disks in the surgery sequence are disjoint
(as opposed to meeting at most four times).  Also, in the terminology
of~\cite{MasurEtAl10}, the disk $D_i$ is a {\em wide dual} for the
track $\tau_{r(i)}$.  Finally, recurrence of $\tau_n$ follows because
$E$ is fully carried by $\tau_N$.  Transverse recurrence follows
because $D$ is fully dual to $\tau_0$.
\end{remark}

Thus $V_n$ will be our marking path and $D_i$ will be our
combinatorial path.  The requirements of \refax{Combin} are now
verified by \refthm{DiskSurgerySequence}.

\subsection{The replacement axiom}

We turn to \refax{Replace}.  Suppose that $Y \subset X$ is an
essential subsurface and $D_i$ has $r(i) \in J_Y$.  Let $n = r(i)$.
From \refthm{DiskSurgerySequence} we have that $\iota(\bdy D_i, V_n)$
is uniformly bounded.  By \refax{Access} we have $Y \subset
\supp(V_n)$ and $\iota(\bdy Y, \mu_n)$ is bounded.  It follows that
there is a constant $K$ depending only on $\xi(S)$ so that
\[
\iota(\bdy D_i, \bdy Y) < K.
\]

Isotope $D_i$ to have minimal intersection with $\bdy Y$.  As in
\refsec{CompressionSequences} boundary compress $D_i$ as much as
possible into the components of $X \setminus \bdy Y$ to obtain a disk
$D'$ so that either
\begin{itemize}
\item
$D'$ cannot be boundary compressed any more into $X \setminus \bdy Y$
or
\item
$D'$ is disjoint from $\bdy Y$.
\end{itemize}
We may arrange matters so that every boundary compression reduces the
intersection with $\bdy Y$ by at least a factor of two.  Thus:
\[
d_\calD(D_i, D') \leq \log_2(K).
\]

Suppose now that $Y$ is a compressible hole.  By
\reflem{XCompressibleImpliesBdyXCompressible} we find that $\bdy D'
\subset Y$ and we are done.

Suppose now that $Y$ is an incompressible hole.  Since $Y$ is large
there is an $I$-bundle $T \to F$, contained in the handlebody $V$, so
that $Y$ is a component of $\bdy_h T$.  Isotope $D'$ to minimize
intersection with $\bdy_v T$.  Let $\Delta$ be the union of components
of $\bdy_v T$ which are contained in $\bdy V$.  Let $\Gamma = \bdy_v T
\setminus \Delta$.  Notice that all intersections $D' \cap \Gamma$ are
essential arcs in $\Gamma$: simple closed curves are ruled out by
minimal intersection and inessential arcs are ruled out by the fact
that $D'$ cannot be boundary compressed in the complement of $\bdy Y$.
Let $D''$ be a outermost component of $D' \setminus \Gamma$.  Then
\reflem{BdyIncompImpliesVertical} implies that $D''$ is isotopic in
$T$ to a vertical disk.  

If $D'' = D'$ then we may replace $D_i$ by the arc $\rho_F(D')$.  The
inductive argument now occurs inside of the arc complex $\calA(F,
\rho_F(\Delta))$.

Suppose that $D'' \neq D'$.  Let $A \in \Gamma$ be the vertical
annulus meeting $D''$.  Let $N$ be a regular neighborhood of $D'' \cup
A$, taken in $T$.  Then the frontier of $N$ in $T$ is again a vertical
disk, call it $D'''$.  Note that $\iota(D''', D') < K - 1$.  Finally,
replace $D_i$ by the arc $\rho_F(D''')$.

Suppose now that $Y$ is not a hole.  Then some component $S \setminus
Y$ is compressible.  Applying
\reflem{XCompressibleImpliesBdyXCompressible} again, we find that
either $D'$ lies in $Z = X \setminus Y$ or in $Y$.  This completes the
verification of \refax{Replace}.

\subsection{Straight intervals}

We end by checking \refax{Straight}.  Suppose that $[p, q] \subset [0,
K]$ is a straight interval.  Recall that $d_Y(\mu_{r(p)}, \mu_{r(q)})
< \Upper$ for all strict subsurfaces $Y \subset X$.  We must check
that $d_\calD(D_p, D_q) \quasileq d_X(D_p, D_q)$.  Since $d_\calD(D_p,
D_q) \leq \Combin |p - q|$ it is enough to bound $|p - q|$.  Note that
$|p - q| \leq |r(p) - r(q)|$ because the reindexing map is increasing.
Now, $|r(p) - r(q)| \quasileq d_{\calM(X)}(\mu_{r(p)}, \mu_{r(q)})$
because the sequence $\{ \mu_n \}$ is a quasi-geodesic in $\calM(X)$
(\refthm{TrainTrackQuasi}).  Increasing $A$ as needed and applying
\refthm{MarkingGraphDistanceEstimate} we have
\[
d_\calM(\mu_{r(p)}, \mu_{r(q)}) \quasileq 
         \sum_Y [d_Y(\mu_{r(p)}, \mu_{r(q)})]_\Upper
\]
and the right hand side is thus less than $d_X(\mu_{r(p)},
\mu_{r(q)})$ which in turn is less than $d_X(D_p, D_q) + 2\Combin$.
This completes our discussion of \refax{Straight} and finishes the
proof of \refthm{DiskComplexDistanceEstimate}.

\section{Hyperbolicity}
\label{Sec:Hyperbolicity}

The ideas in this section are related to the notion of ``time-ordered
domains'' and to the hierarchy machine of~\cite{MasurMinsky00} (see
also Chapters~4 and~5 of Behrstock's thesis~\cite{Behrstock04}).  As
remarked above, we cannot use those tools directly as the hierarchy
machine is too rigid to deal with the disk complex.  

\subsection{Hyperbolicity}

We prove:

\begin{theorem}
\label{Thm:GIsHyperbolic}
Fix $\calG = \calG(S)$, a combinatorial complex.  Suppose that $\calG$
satisfies the axioms of \refsec{Axioms}.  Then $\calG$ is Gromov
hyperbolic.
\end{theorem}

As corollaries we have

\begin{theorem}
\label{Thm:ArcComplexHyperbolic}
The arc complex is Gromov hyperbolic. \qed
\end{theorem}

\begin{theorem}
\label{Thm:DiskComplexHyperbolic}
The disk complex is Gromov hyperbolic. \qed
\end{theorem}

In fact, \refthm{GIsHyperbolic} follows quickly from:

\begin{theorem}
\label{Thm:GoodPathsGiveSlimTriangles}
Fix $\calG$, a combinatorial complex.  Suppose that $\calG$ satisfies
the axioms of \refsec{Axioms}.  Then for all $A \geq 1$ there exists
$\delta \geq 0$ with the following property: Suppose that $T \subset
\calG$ is a triangle of paths where the projection of any side of $T$
into into any hole is an $A$--unparameterized quasi-geodesic.  Then T
is $\delta$--slim.
\end{theorem}

\begin{proof}[Proof of \refthm{GIsHyperbolic}]
As laid out in \refsec{Partition} there is a uniform constant $A$ so
that for any pair $\alpha, \beta \in \calG$ there is a recursively
constructed path $\calP = \{\gamma_i\} \subset \calG$ so that
\begin{itemize}
\item
for any hole $X$ for $\calG$, the projection $\pi_X(\calP)$ is an
$A$--un\-pa\-ram\-e\-ter\-ized quasi-geodesic and
\item
$|\calP| \quasieq d_\calG(\alpha, \beta)$.
\end{itemize}
So if $\alpha \cap \beta = \emptyset$ then $|\calP|$ is uniformly
short.  Also, by \refthm{GoodPathsGiveSlimTriangles}, triangles made
of such paths are uniformly slim.  Thus, by
\refthm{HyperbolicityCriterion}, $\calG$ is Gromov hyperbolic.
\end{proof}

The rest of this section is devoted to proving
\refthm{GoodPathsGiveSlimTriangles}. 

\subsection{Index in a hole}

For the following definitions, we assume that $\alpha$ and $\beta$ are
fixed vertices of $\calG$.

For any hole $X$ and for any geodesic $k \in \calC(X)$ connecting a
point of $\pi_X(\alpha)$ to a point of $\pi_X(\beta)$ we also define
$\rho_k \from \calG \to k$ to be the relation $\pi_X|\calG \from \calG
\to \calC(X)$ followed by taking closest points in $k$.  Since the
diameter of $\rho_k(\gamma)$ is uniformly bounded, we may simplify our
formulas by treating $\rho_k$ as a function.  Define $\ind_X \from
\calG \to \NN$ to be the {\em index} in $X$:
\[
\ind_X(\sigma) = d_X(\alpha, \rho_k(\sigma)).
\]

\begin{remark}
\label{Rem:ChoiceOfGeodesic}
Suppose that $k'$ is a different geodesic connecting $\pi_X(\alpha)$
to $\pi_X(\beta)$ and $\ind'_X$ is defined with respect to $k'$.  Then
\[
|\ind_X(\sigma) - \ind'_X(\sigma)| \leq 17\delta + 4
\]
by \reflem{MovePoint} and \reflem{MoveGeodesic}.  After permitting a
small additive error, the index depends only on $\alpha, \beta,
\sigma$ and not on the choice of geodesic $k$.
\end{remark}

\subsection{Back and sidetracking}

Fix $\sigma, \tau \in \calG$.  We say $\sigma$ {\em precedes} $\tau$
{\em by at least} $K$ in $X$ if
\[
\ind_X(\sigma) + K \leq \ind_X(\tau).
\]

We say $\sigma$ precedes $\tau$ {\em by at most} $K$ if the inequality
is reversed.  If $\sigma$ precedes $\tau$ then we say $\tau$ {\em
succeeds} $\sigma$.

Now take $\calP = \{\sigma_i\}$ to be a path in $\calG$ connecting
$\alpha$ to $\beta$.  Recall that we have made the simplifying
assumption that $\sigma_i$ and $\sigma_{i+1}$ are disjoint.

We formalize a pair of properties enjoyed by unparameterized
quasi-geodesics.  The path $\calP$ {\em backtracks} at most $K$ if for
every hole $X$ and all indices $i < j$ we find that $\sigma_j$
precedes $\sigma_i$ by at most $K$.  The path $\calP$ {\em sidetracks}
at most $K$ if for every hole $X$ and every index $i$ we find that
\[
d_X(\sigma_i, \rho_k(\sigma_i)) \leq K,
\]
for some geodesic $k$ connecting a point of $\pi_X(\alpha)$ to a
point of $\pi_X(\beta)$.

\begin{remark}
\label{Rem:ChoiceOfGeodesicTwo}
As in Remark~\ref{Rem:ChoiceOfGeodesic}, allowing a small additive
error makes irrelevant the choice of geodesic in the definition of
sidetracking.
We note that, if $\calP$ has bounded sidetracking, one may freely use
in calculation whichever of $\sigma_i$ or $\rho_k(\sigma_i)$ is more
convenient.
\end{remark}

\subsection{Projection control}

We say domains $X, Y \subset S$ {\em overlap} if $\bdy X$ cuts $Y$ and
$\bdy Y$ cuts $X$. The following lemma, due to
Behrstock~\cite[4.2.1]{Behrstock04}, is closely related to the notion
of {\em time ordered} domains~\cite{MasurMinsky00}.  An elementary
proof is given in~\cite[Lemma~2.5]{Mangahas10}.

\begin{lemma}
\label{Lem:Zugzwang}
There is a constant $\ZugConst = \ZugConst(S)$ with the following
property.  Suppose that $X, Y$ are overlapping non-simple domains.  If
$\gamma \in \AC(S)$ cuts both $X$ and $Y$ then either $d_X(\gamma,
\bdy Y) < \ZugConst$ or $d_Y(\bdy X, \gamma) < \ZugConst$. \qed
\end{lemma}

We also require a more specialized version for the case where $X$ and
$Y$ are nested.

\begin{lemma}
\label{Lem:ZugzwangTwo}
There is a constant $\ZugConstTwo = \ZugConstTwo(S)$ with the
following property.  Suppose that $X \subset Y$ are nested non-simple
domains.  Fix $\alpha, \beta, \gamma \in \AC(S)$ that cut $X$.
Fix $k = [\alpha', \beta'] \subset \calC(Y)$, a geodesic connecting a
point of $\pi_Y(\alpha)$ to a point of $\pi_Y(\beta)$.  Assume that
$d_X(\alpha, \beta) \geq \GeodConst$, the constant given by
\refthm{BoundedGeodesicImage}.  If $d_X(\alpha, \gamma) \geq
\ZugConstTwo$ then
$$\ind_Y(\bdy X) - 4 \leq \ind_Y(\gamma).$$ 
Symmetrically, we have
$$\ind_Y(\gamma) \leq \ind_Y(\bdy X) + 4$$ 
if $d_X(\gamma, \beta) \geq \ZugConstTwo$. \qed
\end{lemma}




\subsection{Finding the midpoint of a side}

Fix $A \geq 1$.  Let $\calP, \calQ, \calR$ be the sides of a triangle
in $\calG$ with vertices at $\alpha, \beta, \gamma$.  We assume that
each of $\calP$, $\calQ$, and $\calR$ are $A$--unparameterized
quasi-geodesics when projected to any hole.

Recall that $\GeodConst = \GeodConst(S)$, $\ZugConst = \ZugConst(S)$,
and $\ZugConstTwo = \ZugConstTwo(S)$ are functions depending only on
the topology of $S$.  We may assume that if $T \subset S$ is an
essential subsurface, then $\GeodConst(S) > \GeodConst(T)$.

Now choose $\PathConst \geq \max \{ \GeodConst, 4\ZugConst,
\ZugConstTwo, 8 \} + 8\delta$ sufficiently large so that any
$A$--unparameterized quasi-geodesic in any hole back and side tracks at
most $\PathConst$.

\begin{claim}
\label{Clm:PrecedesSucceedsExclusive}
If $\sigma_i$ precedes $\gamma$ in $X$ and $\sigma_j$ succeeds
$\gamma$ in $Y$, both by at least $2\PathConst$, then $i < j$.
\end{claim}

\begin{proof}
To begin, as $X$ and $Y$ are holes and all holes interfere, we need
not consider the possibility that $X \cap Y = \emptyset$.  If $X = Y$
we immediately deduce that $$\ind_X(\sigma_i) + 2\PathConst \leq
\ind_X(\gamma) \leq \ind_X(\sigma_j) - 2\PathConst.$$ Thus
$\ind_X(\sigma_i) + 4\PathConst \leq \ind_X(\sigma_j)$.  Since $\calP$
backtracks at most $\PathConst$ we have $i < j$, as desired.

Suppose instead that $X \subset Y$.  Since $\sigma_i$ precedes
$\gamma$ in $X$ we immediately find $d_X(\alpha, \beta) \geq
2\PathConst \geq \GeodConst$ and $d_X(\alpha, \gamma) \geq 2\PathConst
- 2\delta \geq \ZugConstTwo$.
Apply \reflem{ZugzwangTwo} to deduce $\ind_Y(\bdy X) - 4 \leq
\ind_Y(\gamma)$.  Since $\sigma_j$ succeeds $\gamma$ in $Y$ it follows
that $\ind_Y(\bdy X) - 4 + 2\PathConst \leq \ind_Y(\sigma_j)$.  Again
using the fact that $\sigma_i$ precedes $\gamma$ in $X$ we have that
$d_X(\sigma_i, \beta) \geq \ZugConstTwo$.  We deduce from
\reflem{ZugzwangTwo} that $\ind_Y(\sigma_i) \leq \ind_Y(\bdy X) +
4$.  Thus $$\ind_Y(\sigma_i) - 8 + 2\PathConst \leq
\ind_Y(\sigma_j).$$ Since $\calP$ backtracks at most $\PathConst$ in
$Y$ we again deduce that $i < j$.  The case where $Y \subset X$ is
similar. 

Suppose now that $X$ and $Y$ overlap.  Applying
\reflem{Zugzwang} and breaking symmetry, we may assume that
$d_X(\gamma, \bdy Y) < \ZugConst$.  Since $\sigma_i$ precedes $\gamma$
we have $\ind_X(\gamma) \geq 2\PathConst$.  \reflem{MovePoint} now
implies that $\ind_X(\bdy Y) \geq 2\PathConst - \ZugConst - 6\delta$.
Thus, 
\[
d_X(\alpha, \bdy Y) \geq 2\PathConst - \ZugConst - 8\delta
\geq \ZugConst
\]
where the first inequality follows from \reflem{RightTriangle}. 

Applying \reflem{Zugzwang} again, we find that $d_Y(\alpha,
\bdy X) < \ZugConst$.  Now, since $\sigma_j$ succeeds $\gamma$ in $Y$,
we deduce that $\ind_Y(\sigma_j) \geq 2\PathConst$.  So
\reflem{RightTriangle} implies that $d_Y(\alpha, \sigma_j) \geq
2\PathConst - 2\delta$.  The triangle inequality now gives
\[
d_Y(\bdy X, \sigma_j) \geq 2\PathConst - \ZugConst - 2\delta \geq
\ZugConst.
\]
Applying \reflem{Zugzwang} one last time, we find that
$d_X(\bdy Y, \sigma_j) < \ZugConst$.  Thus $d_X(\gamma, \sigma_j) \leq
2\ZugConst$.  Finally, \reflem{MovePoint} implies that the difference
in index (in $X$) between $\sigma_i$ and $\sigma_j$ is at least
$2\PathConst - 2\ZugConst - 6\delta$.  Since this is greater than the
backtracking constant, $\PathConst$, it follows that $i < j$.
\end{proof}

Let $\sigma_\alpha \in \calP$ be the {\em last} vertex of $\calP$
preceding $\gamma$ by at least $2\PathConst$ in some hole.  If no such
vertex of $\calP$ exists then take $\sigma_\alpha = \alpha$.

\begin{claim}
\label{Clm:LastNearCenter}
 For every hole $X$ and geodesic $h$ connecting $\pi_X(\alpha)$ to
$\pi_X(\beta)$:
\[
d_X(\sigma_\alpha, \rho_h(\gamma)) \leq  3\PathConst+6\delta+1
\]
\end{claim}

\begin{proof}
Since $\sigma_i$ and $\sigma_{i+1}$ are disjoint we have
$d_X(\sigma_i, \sigma_{i+1}) \geq 3$ and so \reflem{MovePoint}
implies that 
\[
|\ind_X(\sigma_{i+1}) - \ind_X(\sigma_i)| \leq 6\delta + 3.
\]
Since $\calP$ is a path connecting $\alpha$ to $\beta$ the image
$\rho_h(\calP)$ is $6\delta + 3$--dense in $h$.  Thus, if
$\ind_X(\sigma_\alpha) + 2\PathConst + 6\delta + 3 < \ind_X(\gamma)$
then we have a contradiction to the definition of $\sigma_\alpha$.

On the other hand, if $\ind_X(\sigma_\alpha) \geq \ind_X(\gamma) +
2\PathConst$ then $\sigma_\alpha$ precedes and succeeds $\gamma$ in
$X$.  This directly contradicts \refclm{PrecedesSucceedsExclusive}.

We deduce that the difference in index between $\sigma_\alpha$ and
$\gamma$ in $X$ is at most $2\PathConst + 6\delta + 3$.  Finally, as
$\calP$ sidetracks by at most $\PathConst$ we have
\[
d_X(\sigma_\alpha, \rho_h(\gamma)) \leq 3\PathConst + 6\delta + 3
\]
as desired.
\end{proof}

We define $\sigma_\beta$ to be the first $\sigma_i$ to succeed
$\gamma$ by at least $2\PathConst$ --- if no such vertex of $\calP$
exists take $\sigma_\beta = \beta$.  If $\alpha = \beta$ then
$\sigma_\alpha = \sigma_\beta$.  Otherwise, from
Claim~\ref{Clm:PrecedesSucceedsExclusive}, we immediately deduce that
$\sigma_\alpha$ comes before $\sigma_\beta$ in $\calP$.  A symmetric
version of Claim~\ref{Clm:LastNearCenter} applies to $\sigma_\beta$:
for every hole $X$
\[
d_X(\rho_h(\gamma), \sigma_\beta) \leq 3\PathConst + 6\delta + 3.
\]

\subsection{Another side of the triangle}

Recall now that we are also given a path $\calR = \{ \tau_i \}$
connecting $\alpha$ to $\gamma$ in $\calG$.  As before, $\calR$ has
bounded back and sidetracking.  Thus we again find vertices
$\tau_\alpha$ and $\tau_\gamma$ the last/first to precede/succeed
$\beta$ by at least $2\PathConst$.  Again, this is defined in terms of
the closest points projection of $\beta$ to a geodesic of the form $h
= [\pi_X(\alpha), \pi_X(\gamma)]$.  By Claim~\ref{Clm:LastNearCenter},
for every hole $X$, $\tau_\alpha$ and $\tau_\gamma$ are close to
$\rho_h(\beta)$.

By \reflem{CenterExists}, if $k = [\pi_X(\alpha), \pi_X(\beta)]$, then
$d_X(\rho_k(\gamma), \rho_h(\beta)) \leq 6\delta$.  We deduce:

\begin{claim}
\label{Clm:BodySmall}
$d_X(\sigma_\alpha, \tau_\alpha) \leq 6\PathConst + 18\delta+2$. \qed
\end{claim}

This claim and \refclm{LastNearCenter} imply that the body of the
triangle $\calP\calQ\calR$ is bounded in size.  We now show that the
legs are narrow.

\begin{claim}
\label{Clm:NearbyVertex}
There is a constant $\NearConstTwo = \NearConstTwo(S)$ with the
following property.  For every $\sigma_i \leq \sigma_\alpha$ in
$\calP$ there is a $\tau_j \leq \tau_\alpha$ in $\calR$ so that
$$d_X(\sigma_i, \tau_j) \leq \NearConstTwo$$
for every hole $X$.
\end{claim}

\begin{proof}
We only sketch the proof, as the details are similar to our previous
discussion.  Fix $\sigma_i \leq \sigma_\alpha$.

Suppose first that no vertex of $\calR$ precedes $\sigma_i$ by more
than $2\PathConst$ in any hole.  So fix a hole $X$ and geodesics $k =
[\pi_X(\alpha), \pi_X(\beta)]$ and $h = [\pi_X(\alpha),
\pi_X(\gamma)]$.  Then $\rho_h(\sigma_i)$ is within distance
$2\PathConst$ of $\pi_X(\alpha)$.  Appealing to
Claim~\ref{Clm:BodySmall}, bounded sidetracking, and hyperbolicity of
$\calC(X)$ we find that the initial segments
\[
[\pi_X(\alpha), \rho_k(\sigma_\alpha)], \quad 
[\pi_X(\alpha), \rho_h(\tau_\alpha)]
\] 
of $k$ and $h$ respectively must fellow travel.  Because of bounded
backtracking along $\calP$, $\rho_k(\sigma_i)$ lies on, or at least
near, this initial segment of $k$.  Thus by \reflem{MoveGeodesic}
$\rho_h(\sigma_i)$ is close to $\rho_k(\sigma_i)$ which in turn is
close to $\pi_X(\sigma_i)$, because $\calP$ has bounded sidetracking.
In short, $d_X(\alpha, \sigma_i)$ is bounded for all holes $X$.  Thus
we may take $\tau_j = \tau_0 = \alpha$ and we are done.

Now suppose that some vertex of $\calR$ precedes $\sigma_i$ by at
least $2\PathConst$ in some hole $X$.  Take $\tau_j$ to be the last
such vertex in $\calR$.  Following the proof of
Claim~\ref{Clm:PrecedesSucceedsExclusive} shows that $\tau_j$ comes
before $\tau_\alpha$ in $\calR$.  The argument now required to bound
$d_X(\sigma_i, \tau_j)$ is essentially identical to the proof of
Claim~\ref{Clm:LastNearCenter}.
\end{proof}

By the distance estimate, we find that there is a uniform neighborhood
of $[\sigma_0, \sigma_\alpha] \subset \calP$, taken in $\calG$, which
contains $[\tau_0, \tau_\alpha] \subset \calP$.  The slimness of
$\calP\calQ\calR$ follows directly.  This completes the proof of
Theorem~\ref{Thm:GoodPathsGiveSlimTriangles}. \qed

\section{Coarsely computing Hempel distance}
\label{Sec:HempelDistance}

We now turn to our topological application. Recall that a {\em
Heegaard splitting} is a triple $(S, V, W)$ consisting of a surface
and two handlebodies where $V \cap W = \bdy V = \bdy W = S$.
Hempel~\cite{Hempel01} defines the quantity 
\[
d_S(V, W) = \min \big\{ d_S(D,E) \st D \in \calD(V), E \in \calD(W) \big\}
\] 
and calls it the {\em distance} of the splitting.  Note that a
splitting can be completely determined by giving a pair of cut
systems: simplices $\DD \subset \calD(V)$, $\EE \subset \calD(W)$
where the corresponding disks cut the containing handlebody into a
single three-ball.  The triple $(S, \DD, \EE)$ is a {\em Heegaard
diagram}.  The goal of this section is to prove:

\begin{theorem}
\label{Thm:CoarselyComputeDistance}
There is a constant $\DistError = \DistError(S)$ and an algorithm
that, given a Heegaard diagram $(S, \DD, \EE)$, computes a number $N$
so that $$|d_S(V, W) - N| \leq \DistError.$$
\end{theorem}

\noindent
Let $\rho_V \from \calC(S) \to \calD(V)$ be the closest points
relation:
\[
\rho_V(\alpha) = \big\{ D \in \calD(V) \st \mbox{ for all $E \in \calD(V)$, 
           $d_S(\alpha, D) \leq d_S(\alpha, E)$ } \big\}.
\]




\noindent
\refthm{CoarselyComputeDistance} follows from:

\begin{theorem}
\label{Thm:CoarselyComputeProjection}
There is a constant $\ProjectError = \ProjectError(V)$ and an
algorithm that, given an essential curve $\alpha \subset S$ and a cut
system $\DD \subset \calD(V)$, finds a disk $C \in \calD(V)$ so that
\[
d_S(C, \rho_V(\alpha)) \leq \ProjectError.
\] 
\end{theorem}


\begin{proof}[Proof of \refthm{CoarselyComputeDistance}]
Suppose that $(S, \DD, \EE)$ is a Heegaard diagram.  Using
Theorem~\ref{Thm:CoarselyComputeProjection} we find a disk $D$ within
distance $\ProjectError$ of $\rho_V(\EE)$.  Again using
Theorem~\ref{Thm:CoarselyComputeProjection} we find a disk $E$ within
distance $\ProjectError$ of $\rho_W(D)$.  Notice that $E$ is defined
using $D$ and not the cut system $\DD$.  


Since computing distance between fixed vertices in the curve complex
is algorithmic~\cite{Leasure02, Shackleton04} we may compute $d_S(D,
E)$.  By the hyperbolicity of $\calC(S)$ (\refthm{C(S)IsHyperbolic})
and by the quasi-convexity of the disk set
(\refthm{DiskComplexConvex}) this is the desired estimate.
\end{proof}



Very briefly, the algorithm asked for in
\refthm{CoarselyComputeProjection} searches an $\Radius$--neighborhood
in $\calM(S)$ about a splitting sequence from $\DD$ to $\alpha$.  Here
are the details.

\begin{algorithm}
\label{Alg:Projection}
We are given $\alpha \in \calC(S)$ and a cut system $\DD \subset
\calD(V)$.  Build a train track $\tau$ in $S = \bdy V$ as follows:
make $\DD$ and $\alpha$ tight.  Place one switch on every disk $D \in
\DD$.  Homotope all intersections of $\alpha$ with $D$ to run through
the switch.  Collapse bigons of $\alpha$ inside of $S \setminus \DD$
to create the branches.  Now make $\tau$ a generic track by combing
away from $\DD$~\cite[Proposition~1.4.1]{PennerHarer92}.  Note that
$\alpha$ is carried by $\tau$ and so gives a transverse measure $w$.

Build a splitting sequence of measured tracks $\{ \tau_n \}_{n = 0}^N$
where $\tau_0 = \tau$, $\tau_N = \alpha$, and $\tau_{n+1}$ is obtained
by splitting the largest switch of $\tau_n$ (as determined by the
measure imposed by $\alpha$).

Let $\mu_n = V(\tau_n)$ be the vertices of $\tau_n$.  For each filling
marking $\mu_n$ list all markings in the ball $B(\mu_n, \Radius)
\subset \calM(S)$, where $\Radius$ is given by
\reflem{DetectingNearbyDisks} below.  (If $\mu_0$ does not fill $S$
then output $\DD$ and halt.)

For every marking $\nu$ so produced we use Whitehead's algorithm (see
\reflem{Whitehead}) to try and find a disk meeting some curve $\gamma
\in \nu$ at most twice.  For every disk $C$ found compute $d_S(\alpha,
C)$~\cite{Leasure02, Shackleton04}.  Finally, output any disk which
minimizes this distance, among all disks considered, and halt.
\end{algorithm}

We use the following form of Whitehead's algorithm~\cite{Berge08}:

\begin{lemma}
\label{Lem:Whitehead}
There is an algorithm that, given a cut system $\DD \subset V$ and a
curve $\gamma \subset S$, outputs a disk $C \subset V$ so that
$\iota(\gamma, \bdy C) = \min \{ \iota(\gamma, \bdy E) \st E \in
\calD(V) \}$. \qed
\end{lemma}

We now discuss the constant $\Radius$.  We begin by noticing that the
track $\tau_n$ is transversely recurrent because $\alpha$ is fully
carried and $\DD$ is fully dual.  Thus by
\refthm{TrainTrackUnparamGeodesic} and by Morse stability, for any
essential $Y \subset S$ there is a stability constant $\Stable$ for
the path $\pi_Y(\mu_n)$.  Let $\delta$ be the hyperbolicity constant
for $\calC(S)$ (\refthm{C(S)IsHyperbolic}) and let $\Quasi$ be the
quasi-convexity constant for $\calD(V) \subset \calC(S)$
(\refthm{DiskComplexConvex}).

Since $\iota(\DD, \mu_0)$ is bounded we will, at the cost of an
additive error, identify their images in $\calC(S)$.  Now, for every
$n$ pick some $E_n \in \rho_V(\mu_n)$.

\begin{lemma}
\label{Lem:DetectingNearbyDisks}
There is a constant $\Radius$ with the following property.  Suppose
that $n < m$, $d_S(\mu_n, E_n), d_S(\mu_m, E_m) \leq \Stable + \delta
+ \Quasi$, and $d_S(\mu_n, \mu_m) \geq 2(\Stable + \delta + \Quasi) +
5$.  Then there is a marking $\nu \in B(\mu_n, \Radius)$ and a curve
$\gamma \in \nu$ so that either:
\begin{itemize}
\item
$\gamma$ bounds a disk in $V$, 
\item
$\gamma \subset \bdy Z$, where $Z$ is a non-hole or
\item
$\gamma \subset \bdy Z$, where $Z$ is a large hole.
\end{itemize}
\end{lemma}

\begin{proof}[Proof of \reflem{DetectingNearbyDisks}]
Choose points $\sigma, \sigma'$ in the thick part of $\calT(S)$ so
that all curves of $\mu_n$ have bounded length in $\sigma$ and so that
$E_n$ has length less than the Margulis constant in $\sigma'$.  As in
\refsec{BackgroundTeich} there is a \Teich geodesic and associated
markings $\{ \nu_k \}_{k = 0}^K$ so that $d_\calM(\nu_0, \mu_n)$ is
bounded and $E_n \in \base(\nu_K)$.  

We say a hole $X \subset S$ is {\em small} if $\diam_X(\calD(V)) <
61$.

\begin{proofclaim}
There is a constant $\RadiusTemp$ so that for any small hole $X$ we
have $d_X(\mu_n, \nu_K) < \RadiusTemp$.
\end{proofclaim}

\begin{proof}
If $d_X(\mu_n, \nu_K) \leq \GeodConst$ then we are done.  If the
distance is greater than $\GeodConst$ then
\refthm{BoundedGeodesicImage} gives a vertex of the
$\calC(S)$--geodesic connecting $\mu_n$ to $E_n$ with distance at most
one from $\bdy X$.  It follows from the triangle inequality that every
vertex of the $\calC(S)$--geodesic connecting $\mu_m$ to $E_m$ cuts
$X$.  Another application of \refthm{BoundedGeodesicImage} gives
\[
d_X(\mu_m, E_m) < \GeodConst.
\] 
Since $X$ is small $d_X(E_m, \DD), d_X(E_n, \DD) \leq 60$.  Since
$\iota(\nu_K, E_n) = 2$ the distance $d_X(\nu_K, E_n)$ is bounded.

Finally, because $p \mapsto \pi_X(\mu_p)$ is an $A$--unparameterized
quasi-geodesic in $\calC(X)$ it follows that $d_X(\DD, \mu_n)$ is also
bounded and the claim is proved.
\end{proof}

Now consider all strict subsurfaces $Y$ so that
\[
d_Y(\mu_n, \nu_M) \geq \RadiusTemp.
\] 
None of these are small holes, by the claim above.  If there are no
such surfaces then \refthm{MarkingGraphDistanceEstimate} bounds
$d_\calM(\mu_n, \nu_M)$: taking the cutoff constant larger than
\[ 
\max \{ \RadiusTemp, \CutOff, \Stable + \delta + \Quasi \}
\] 
ensures that all terms on the right-hand side vanish.  In this case
the additive error in \refthm{MarkingGraphDistanceEstimate} is the
desired constant $\Radius$ and the lemma is proved.

If there are such surfaces then choose one, say $Z$, that minimizes
$\ell = \min J_Z$.  Thus $d_Y(\mu_n, \nu_\ell) < \Access$ for all
strict non-holes and all strict large holes.  Since $d_S(\mu_n, E_n)
\leq \Stable + \delta + \Quasi$ and $\{ \nu_m \}$ is an unparameterized
quasi-geodesic \cite[Theorem~6.1]{Rafi10} we find that $d_S(\mu_n,
\nu_l)$ is uniformly bounded.  The claim above bounds distances in
small holes.  As before we find a sufficiently large cutoff so that all
terms on the right-hand side of \refthm{MarkingGraphDistanceEstimate}
vanish.  Again the additive error of
\refthm{MarkingGraphDistanceEstimate} provides the constant $\Radius$.
Since $\bdy Z \subset \base(\nu_\ell)$ the lemma is proved.
\end{proof}

To prove the correctness of \refalg{Projection} it suffices to show
that the disk produced is close to $\rho_V(\alpha)$.  Let $m$ be the
largest index so that for all $n \leq m$ we have
\[
d_S(\mu_n, E_n) \leq \Stable + \delta + \Quasi.
\] 
It follows that $\mu_{m+1}$ lies within distance $\Stable + \delta$
of the geodesic $[\alpha, \rho_V(\alpha)]$.  Recall that $d_S(\mu_n,
\mu_{n+1}) \leq \Reverse$ for any value of $n$.  A shortcut argument
shows that 
\[
d_S(\mu_m, \rho_V(\alpha)) \leq 2\Reverse + 3\Stable + 3\delta +
\Quasi.
\]
Let $n \leq m$ be the largest index so that 
\[
2(\Stable + \delta + \Quasi) + 5 \leq d_S(\mu_n, \mu_m).
\]
If no such $n$ exists then take $n = 0$.  Now,
\reflem{DetectingNearbyDisks} implies that there is a disk $C$ with
$d_S(C, \mu_n) \leq 4\Radius$ and this disk is found during the
running of \refalg{Projection}.  It follows from the above
inequalities that 
\[
d_S(C, \alpha) \leq 4\Radius + 5\Stable + 5\delta + 3\Quasi + 5 + 
2\Reverse + d_S(\alpha, \rho_V(\alpha)).
\] 
So the disk $C'$, output by the algorithm, is at least this close to
$\alpha$ in $\calC(S)$. Examining the triangle with vertices $\alpha,
\rho_V(\alpha), C'$ and using a final short-cut argument gives
\[ 
d_S(C', \rho_V(\alpha)) \leq 
   4\Radius + 5\Stable + 9\delta + 5\Quasi + 5 + 2\Reverse.
\]
This completes the proof of \refthm{CoarselyComputeProjection}.  \qed


\bibliographystyle{plain}
\bibliography{bibfile}

\def\cprime{$'$}
\begin{thebibliography}{10}

\bibitem{Behrstock04}
Jason Behrstock.
\newblock {\em Asymptotic geometry of the mapping class group and
  {T}eichm\"uller space}.
\newblock PhD thesis, SUNY Stony Brook, 2004.
\newblock
  \href{http://www.math.columbia.edu/~jason/thesis.pdf}{http://www.math.columb%
ia.edu/{$\sim$}jason/thesis.pdf}.

\bibitem{BehrstockEtAl05}
Jason Behrstock, Cornelia Drutu, and Lee Mosher.
\newblock {Thick metric spaces, relative hyperbolicity, and quasi-isometric
  rigidity}.
\newblock \href{http://arxiv.org/abs/math/0512592}{arXiv:math/0512592}.

\bibitem{Berge08}
John Berge.
\newblock Heegaard documentation.
\newblock documentation of computer program available at computop.org.

\bibitem{BestvinaFujiwara07}
Mladen Bestvina and Koji Fujiwara.
\newblock Quasi-homomorphisms on mapping class groups.
\newblock {\em Glas. Mat. Ser. III}, 42(62)(1):213--236, 2007.
\newblock \href{http://arxiv.org/abs/math/0702273}{arXiv:math/0702273}.

\bibitem{Birman06}
Joan~S. Birman.
\newblock The topology of 3-manifolds, {H}eegaard distance and the mapping
  class group of a 2-manifold.
\newblock In {\em Problems on mapping class groups and related topics},
  volume~74 of {\em Proc. Sympos. Pure Math.}, pages 133--149. Amer. Math.
  Soc., Providence, RI, 2006.
\newblock
  \href{http://www.math.columbia.edu/~jb/papers.html}{http://www.math.columbia%
.edu/{$\sim$}jb/papers.html}.

\bibitem{Bowditch06}
Brian~H. Bowditch.
\newblock Intersection numbers and the hyperbolicity of the curve complex.
\newblock {\em J. Reine Angew. Math.}, 598:105--129, 2006.
\newblock
  \href{http://www.warwick.ac.uk/~masgak/papers/bhb-curvecomplex.pdf}{bhb-curv%
ecomplex.pdf}.

\bibitem{BrendleMargalit04}
Tara~E. Brendle and Dan Margalit.
\newblock Commensurations of the {J}ohnson kernel.
\newblock {\em Geom. Topol.}, 8:1361--1384 (electronic), 2004.
\newblock \href{http://arxiv.org/abs/math/0404445}{arXiv:math/0404445}.

\bibitem{Bridson99}
Martin~R. Bridson and Andr{\'e} Haefliger.
\newblock {\em Metric spaces of non-positive curvature}.
\newblock Springer-Verlag, Berlin, 1999.

\bibitem{Brock03a}
Jeffrey~F. Brock.
\newblock The {W}eil-{P}etersson metric and volumes of 3-dimensional hyperbolic
  convex cores.
\newblock {\em J. Amer. Math. Soc.}, 16(3):495--535 (electronic), 2003.
\newblock \href{http://arxiv.org/abs/math/0109048}{arXiv:math/0109048}.

\bibitem{CavicchioliSpaggiari06}
Alberto Cavicchioli and Fulvia Spaggiari.
\newblock A note on irreducible {H}eegaard diagrams.
\newblock {\em Int. J. Math. Math. Sci.}, pages Art. ID 53135, 11, 2006.

\bibitem{ChoiRafi07}
Young-Eun Choi and Kasra Rafi.
\newblock Comparison between {T}eichm\"uller and {L}ipschitz metrics.
\newblock {\em J. Lond. Math. Soc. (2)}, 76(3):739--756, 2007.
\newblock \href{http://arxiv.org/abs/math/0510136}{math.GT/0510136}.

\bibitem{CDP90}
M.~Coornaert, T.~Delzant, and A.~Papadopoulos.
\newblock {\em G\'eom\'etrie et th\'eorie des groupes}.
\newblock Springer-Verlag, Berlin, 1990.
\newblock {L}es groupes hyperboliques de Gromov.

\bibitem{Gilman94}
Robert~H. Gilman.
\newblock The geometry of cycles in the {C}ayley diagram of a group.
\newblock In {\em The mathematical legacy of {W}ilhelm {M}agnus: groups,
  geometry and special functions ({B}rooklyn, {NY}, 1992)}, volume 169 of {\em
  Contemp. Math.}, pages 331--340. Amer. Math. Soc., Providence, RI, 1994.

\bibitem{Gilman02}
Robert~H. Gilman.
\newblock On the definition of word hyperbolic groups.
\newblock {\em Math. Z.}, 242(3):529--541, 2002.

\bibitem{Gromov87}
Mikhael Gromov.
\newblock Hyperbolic groups.
\newblock In {\em Essays in group theory}, pages 75--263. Springer, New York,
  1987.

\bibitem{Hamenstadt09}
Ursula Hamenst{\"a}dt.
\newblock Geometry of the mapping class groups. {I}. {B}oundary amenability.
\newblock {\em Invent. Math.}, 175(3):545--609, 2009.

\bibitem{Hartshorn02}
Kevin Hartshorn.
\newblock Heegaard splittings of {H}aken manifolds have bounded distance.
\newblock {\em Pacific J. Math.}, 204(1):61--75, 2002.
\newblock
  \href{http://nyjm.albany.edu:8000/PacJ/2002/204-1-5nf.htm}{http://nyjm.alban%
y.edu:8000/PacJ/2002/204-1-5nf.htm}.

\bibitem{Harvey81}
Willam~J. Harvey.
\newblock Boundary structure of the modular group.
\newblock In {\em Riemann surfaces and related topics: Proceedings of the 1978
  Stony Brook Conference (State Univ. New York, Stony Brook, N.Y., 1978)},
  pages 245--251, Princeton, N.J., 1981. Princeton Univ. Press.

\bibitem{HatcherThurston80}
A.~Hatcher and W.~Thurston.
\newblock A presentation for the mapping class group of a closed orientable
  surface.
\newblock {\em Topology}, 19(3):221--237, 1980.

\bibitem{Hempel01}
John Hempel.
\newblock 3-manifolds as viewed from the curve complex.
\newblock {\em Topology}, 40(3):631--657, 2001.
\newblock \href{http://arxiv.org/abs/math/9712220}{arXiv:math/9712220}.

\bibitem{Kobayashi88b}
Tsuyoshi Kobayashi.
\newblock Heights of simple loops and pseudo-{A}nosov homeomorphisms.
\newblock In {\em Braids (Santa Cruz, CA, 1986)}, pages 327--338. Amer. Math.
  Soc., Providence, RI, 1988.

\bibitem{Leasure02}
Jason Leasure.
\newblock Geodesics in the complex of curves of a surface.
\newblock Ph.D.~thesis.
  \href{http://repositories.lib.utexas.edu/bitstream/handle/2152/1700/leasurej%
p46295.pdf}{
  http://repositories.lib.utexas.edu/bitstream/handle/2152/1700/leasurejp46295%
.pdf}.

\bibitem{Mangahas10}
Johanna Mangahas.
\newblock Uniform uniform exponential growth of subgroups of the mapping class
  group.
\newblock {\em Geom. Funct. Anal.}, 19(5):1468--1480, 2010.
\newblock \href{http://arxiv.org/abs/0805.0133}{arXiv:0805.0133}.

\bibitem{MasurMinsky99}
Howard~A. Masur and Yair~N. Minsky.
\newblock Geometry of the complex of curves. {I}. {H}yperbolicity.
\newblock {\em Invent. Math.}, 138(1):103--149, 1999.
\newblock \href{http://arxiv.org/abs/math/9804098}{arXiv:math/9804098}.

\bibitem{MasurMinsky00}
Howard~A. Masur and Yair~N. Minsky.
\newblock Geometry of the complex of curves. {II}. {H}ierarchical structure.
\newblock {\em Geom. Funct. Anal.}, 10(4):902--974, 2000.
\newblock \href{http://arxiv.org/abs/math/9807150}{arXiv:math/9807150}.

\bibitem{MasurMinsky04}
Howard~A. Masur and Yair~N. Minsky.
\newblock Quasiconvexity in the curve complex.
\newblock In {\em In the tradition of {A}hlfors and {B}ers, {III}}, volume 355
  of {\em Contemp. Math.}, pages 309--320. Amer. Math. Soc., Providence, RI,
  2004.
\newblock \href{http://arxiv.org/abs/math/0307083}{arXiv:math/0307083}.

\bibitem{MasurEtAl10}
Howard~A. Masur, Lee Mosher, and Saul Schleimer.
\newblock On train track splitting sequences.
\newblock \href{http://arxiv.org/abs/1004.4564}{arXiv:1004.4564}.

\bibitem{McCullough91}
Darryl McCullough.
\newblock Virtually geometrically finite mapping class groups of
  {$3$}-manifolds.
\newblock {\em J. Differential Geom.}, 33(1):1--65, 1991.

\bibitem{Minsky10}
Yair Minsky.
\newblock The classification of {K}leinian surface groups. {I}. {M}odels and
  bounds.
\newblock {\em Ann. of Math. (2)}, 171(1):1--107, 2010.
\newblock \href{http://arxiv.org/abs/math/0302208}{arXiv:math/0302208}.

\bibitem{Mosher03}
Lee Mosher.
\newblock Train track expansions of measured foliations.
\newblock 2003.
\newblock
  \href{http://newark.rutgers.edu/~mosher/}{http://newark.rutgers.edu/{$\sim$}%
mosher/}.

\bibitem{PennerHarer92}
R.~C. Penner and J.~L. Harer.
\newblock {\em Combinatorics of train tracks}, volume 125 of {\em Annals of
  Mathematics Studies}.
\newblock Princeton University Press, Princeton, NJ, 1992.

\bibitem{Penner88}
Robert~C. Penner.
\newblock A construction of pseudo-{A}nosov homeomorphisms.
\newblock {\em Trans. Amer. Math. Soc.}, 310(1):179--197, 1988.

\bibitem{Rafi10}
Kasra Rafi.
\newblock Relative hyperbolicity in {T}eichm\"uller space, 2010.
\newblock Preprint.
  \href{http://www.math.ou.edu/~rafi/research/Fellow.pdf}{http://www.math.ou.e%
du/{$\sim$}rafi/research/Fellow.pdf}.

\bibitem{RafiSchleimer09}
Kasra Rafi and Saul Schleimer.
\newblock Covers and the curve complex.
\newblock {\em Geom. Topol.}, 13(4):2141--2162, 2009.
\newblock \href{http://arxiv.org/abs/math/0701719}{arXiv:math/0701719}.

\bibitem{Scharlemann82}
Martin Scharlemann.
\newblock The complex of curves on nonorientable surfaces.
\newblock {\em J. London Math. Soc. (2)}, 25(1):171--184, 1982.

\bibitem{Schleimer06b}
Saul Schleimer.
\newblock Notes on the curve complex.
\newblock
  \href{http://www.warwick.ac.uk/~masgar/Maths/notes.pdf}{http://www.warwick.a%
c.uk/~masgar/Maths/notes.pdf}.

\bibitem{Shackleton04}
Kenneth~J. Shackleton.
\newblock {Tightness and computing distances in the curve complex}.
\newblock \href{http://arxiv.org/abs/math/0412078}{arXiv:math/0412078}.

\bibitem{Thurston88}
William~P. Thurston.
\newblock On the geometry and dynamics of diffeomorphisms of surfaces.
\newblock {\em Bull. Amer. Math. Soc. (N.S.)}, 19(2):417--431, 1988.

\end{thebibliography}
\end{document}